\documentclass[11pt]{amsart}
\usepackage{color}
\usepackage{hyperref}
 
\usepackage{amsthm,amsfonts,amsmath,amscd,amssymb,epsfig,verbatim,
}

\usepackage{parskip}    
\usepackage{graphicx}
\usepackage{pinlabel}
\usepackage{epstopdf}
\usepackage{enumitem}
\theoremstyle{plain}
\newtheorem{theorem}{Theorem}[section]
\newtheorem{thm}[theorem]{Theorem}
\newtheorem{corollary}[theorem]{Corollary}
\newtheorem{cor}[theorem]{Corollary}
\newtheorem{proposition}[theorem]{Proposition}
\newtheorem{prop}[theorem]{Proposition}
\newtheorem{lemma}[theorem]{Lemma}

\newtheorem{mainthm}{Theorem}
\newtheorem{definition}[theorem]{Definition}
\newtheorem{maincor}[mainthm]{Corollary}

\theoremstyle{remark}

\newtheorem{remark}[theorem]{Remark}
\newtheorem*{remark*}{Remark}
\newtheorem{example}[theorem]{Example}
\newtheorem*{example*}{Example}

\newcommand{\id}{{{\mathchoice {\rm 1\mskip-4mu l} {\rm 1\mskip-4mu l}
{\rm 1\mskip-4.5mu l} {\rm 1\mskip-5mu l}}}}

\newcommand{\om}{\omega}

\newcommand{\Z}{{\mathbb{Z}}}
\newcommand{\R}{{\mathbb{R}}}
\newcommand{\C}{{\mathbb{C}}}
\newcommand{\CP}{{\mathbb{C}P}}
\newcommand{\Q}{{\mathbb{Q}}}

\newcommand{\TT}{{\mathbb{T}}}

\newcommand{\ind}{\OP{ind}}

\renewcommand{\max}{{\rm max}}

\newcommand{\std}{{\rm std}}

\newcommand{\pt}{\mathrm{pt}}

\newcommand{\OP}{\operatorname}

\begin{document}

\title{Lagrangian isotopy of tori in $S^2\times S^2$ and $\C P^2$}

\author[G. Dimitroglou Rizell]{Georgios Dimitroglou Rizell}
\author[E. Goodman]{Elizabeth Goodman}
\author[A. Ivrii]{Alexander Ivrii}

\thanks{The first author is supported by the grant KAW 2013.0321 from the Knut and Alice Wallenberg Foundation}

\address{Centre for Mathematical Sciences\\
University of Cambridge\\
Wilberforce Road\\
Cambridge, CB3 0WB\\
United Kingdom}
\email{g.dimitroglou@maths.cam.ac.uk}

\address{Department of Mathematics\\
450 Serra Mall\\
Building 380\\
Stanford, CA 94305-2125\\
USA}
\email{egoodman@math.stanford.edu}

\address{IBM Research - Haifa\\
IBM R \& D Labs in Israel\\
University of Haifa Campus\\
Mount Carmel, Haifa\\
3498825\\
Israel}
\email{alexi@il.ibm.com}

\maketitle

\begin{abstract}
We show that, up to Lagrangian isotopy, there is a unique Lagrangian torus inside each of the following uniruled symplectic four-manifolds: the symplectic vector space $\R^4$, the projective plane $\CP^2$, and the monotone $S^2 \times S^2$. The result is proven by studying pseudoholomorphic foliations while performing the splitting construction from symplectic field theory along the Lagrangian torus. A number of other related results are also shown. Notably, the nearby Lagrangian conjecture is established for $T^*\TT^2$, i.e.~it is shown that every closed exact Lagrangian submanifold in this cotangent bundle is Hamiltonian isotopic to the zero-section.
\end{abstract}

\section{Introduction}
A symplectic manifold $(X,\omega)$ is a smooth even-dimensional manifold endowed with a closed non-degenerate two-form $\omega$. A half-dimensional submanifold $L \subset (X,\omega)$ is called \emph{Lagrangian} if $\omega|_{TL} \equiv 0$. In this paper we mainly study Lagrangian tori in various symplectic four-manifolds up to \emph{Lagrangian isotopy}, i.e.~smooth isotopy through Lagrangian submanifolds.

The following uniruled symplectic four-dimensional manifolds are treated here: The four-dimensional vector space $(\R^4,\omega_0)$ with the standard symplectic two-form $\omega_0=dx_1\wedge dy_1 + dx_2 \wedge dy_2$, the projective plane $(\CP^2,\omega_{\OP{FS}})$ endowed with the Fubini--Study two-form, and the monotone ruled surface $(S^2 \times S^2,\omega_1 \oplus \omega_1)$ where $\omega_1$ is an area-form on $S^2$ of total area $\int_{S^2}\omega_1=1$. In addition, we also consider the cotangent bundle $(T^*\TT^2=\TT^2 \times \R^2,d\lambda)$ of the torus $\TT^2=(S^1)^2$ endowed with its canonical symplectic form. Here $\lambda=p_1 d\theta_1 +p_2d\theta_2$ is the Liouville one-form, where $(p_1,p_2) \in \R^2$ denote the standard coordinates, while $\theta_i \colon (S^1)^2 \to \R / 2\pi \Z$ denotes the standard angular coordinate on the $i$:th factor, $i=1,2$.

\subsection{Statement of the results}

Our main theorem is the following.

\begin{mainthm}\label{thm:Lagrangian-isotopy}
Let $(X,\om)$ denote either of the symplectic manifolds $(\R^4,\omega_0)$, $(\CP^2,\omega_{\OP{FS}})$, or $(S^2\times S^2,\omega_1 \oplus \omega_1)$. Any two Lagrangian tori inside $(X,\omega)$ are Lagrangian isotopic.
\end{mainthm}

Using the same methods, in Section \ref{sec:nearby} we also show an a priori much stronger statement for the cotangent bundle of a torus.
\begin{mainthm}\label{thm:nearby}
Any closed exact Lagrangian submanifold $L\subset (T^*\TT^2,d\lambda)$, where $\lambda$ denotes the Liouville form, is a torus which is Hamiltonian isotopic to the zero-section.
\end{mainthm}

As an important step in the proof of Theorem \ref{thm:Lagrangian-isotopy}, we establish the following result which essentially reduces the cases of $\CP^2$ and $S^2 \times S^2$ to the one of $\R^4$.

\begin{mainthm}\label{thm:complement}
Let $(X,\om)$ denote either of the symplectic manifolds $(\CP^2,\omega_{\OP{FS}})$ or $(S^2\times S^2,\omega_1 \oplus \omega_1)$. In the case $X=\CP^2$ let $D_\infty$ denote the line at infinity, while in the case $X= S^2\times S^2$ let $D_\infty$ denote the union $(S^2\times \{\infty\}) \cup (\{\infty\}\times S^2)$ of two holomorphic lines. In either case, a Lagrangian torus $L \subset (X,\omega)$ can be placed inside the complement $X \setminus D_\infty$ by a Hamiltonian isotopy.
\end{mainthm}

We also prove a result which, as discussed in Section \ref{sec:classification} below, can be seen as a potential starting point for the classification of monotone Lagrangian tori in $(S^2 \times S^2,\omega_1\oplus\omega_1)$ up to \emph{Hamiltonian isotopy}. In the following we will use
\[A_1:=[S^2 \times \{ \pt \}], \: A_2:=[\{\pt \} \times S^2] \in H_2(S^2 \times S^2)\]
to denote the canonical set of generators of homology induced by the product structure. Recall that a Lagrangian submanifold $L \subset (S^2 \times S^2,\omega_1\oplus\omega_1)$ is \emph{monotone} given that $\int_u\omega = \frac{1}{4}\mu_L(u)$ holds for each disk $u\in \pi_2(X,L)$, where $\mu_L(u)$ denotes the Maslov index of the disk $u$.

When we speak about a \emph{symplectic $F$-fibration $p \colon (X,\omega) \to B$} we mean a smooth locally trivial fibration over a surface $B$ all whose fibers are symplectic submanifolds of $(X,\omega)$ diffeomorphic to $F$, i.e.~what is usually called a symplectic fibration compatible with $\omega$. A symplectic fibration $p$ as above is said to be \emph{compatible} with a Lagrangian submanifold $L \subset (X,\omega)$ if $p|_L \colon L \to p(L)$ is a smooth and locally trivial $S^1$-fibration over an embedded closed curve $S^1 \cong p(L) \subset B$ in the base. See Figure \ref{fig:compatible} for a schematic picture.

\begin{figure}[htp]
\centering
\vspace{3mm}
\labellist
\pinlabel $p^{-1}(q)$ at 90 83
\pinlabel $q$ at 50 38
\pinlabel $\color{blue}p(L)$ at 67 18
\pinlabel $\color{blue}L\cap p^{-1}(l)$ at 193 162
\pinlabel $p^{-1}(l)$ at 230 83
\pinlabel $\color{blue}l$ at 184 38
\endlabellist
\includegraphics{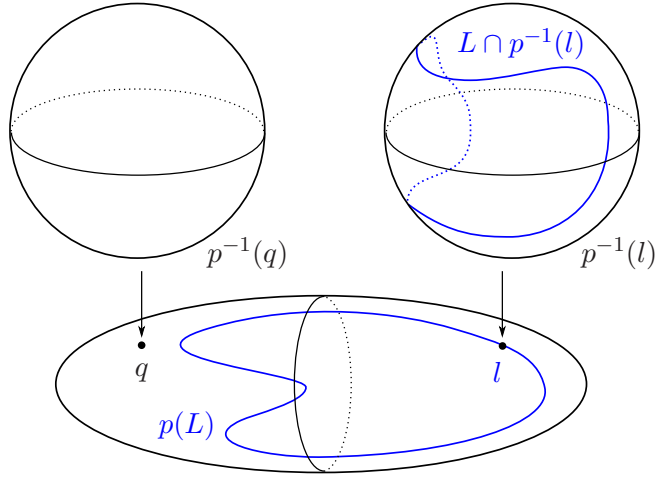}
\caption{A symplectic fibration $p \colon (S^2 \times S^2,\omega_1 \oplus \omega_1) \to S^2$ compatible with a Lagrangian torus $L$. The fibers that intersect $L$ correspond to broken spheres as shown on the left in Figure \ref{fig:breaking}.}
\label{fig:compatible}
\end{figure}

The following theorem gives a more precise characterization of monotone Lagrangian tori in $(S^2\times S^2,\omega_1\oplus\omega_1)$.
\begin{mainthm}\label{thm:fibration}
Let $L$ be a monotone Lagrangian torus in $(S^2\times S^2,\omega_1\oplus\omega_1)$. There exist symplectic $S^2$-fibrations $p_i \colon (S^2 \times S^2,\omega_1 \oplus \omega_1) \to S^2$, $i=1,2$, compatible with $L$, where the fibers of $p_1$ and $p_2$ are in the homology classes $A_2$ and $A_1$, respectively.

Moreover, given closed subsets of the form
\[ (U \times S^2) \: \cup \: (S^2 \times V) \subset S^2 \times S^2 \setminus L \]
contained in the complement of $L$, we may assume that each $\{u\} \times S^2$, $u \in U$, are fibers of $p_1$ and sections of $p_2$, respectively, and conversely that each $S^2 \times \{v\}$, $v \in V$, are fibers of $p_2$ and sections of $p_1$, respectively.
\end{mainthm}
\begin{remark}
In the case of a general (not necessarily monotone) Lagrangian torus $L \subset (S^2 \times S^2,\omega_1 \oplus \omega_1)$, we expect that the same methods can be used to produce a fibration $p \colon (S^2 \times S^2,\omega_1 \oplus \omega_1) \to S^2$ as above, but for which $p(L) \subset S^2$ is an \emph{immersion} of $S^1$, the double points of which correspond to exceptional broken configurations as shown on the right in Figure \ref{fig:breaking} below. For generic data this immersed curve should moreover be in general position. In order to prove this statement, the smoothing procedure in Proposition \ref{prop:smoothdisks} must first be extended in order to handle also the exceptional configurations.
\end{remark}
We also give a proof of the following corollary, whose proof follows without too much effort from the proof of Theorem \ref{thm:fibration}.
\begin{maincor}
\label{cor:fibration}
Let $L$ be a monotone Lagrangian torus in $(S^2 \times S^2,\omega_1\oplus\omega_1)$, it follows that $L$ is Hamiltonian isotopic to a torus
\[L' \subset (S^2 \setminus \{ 0, \infty\}) \: \times \: (S^2 \setminus \{ 0, \infty\}) \subset S^2 \times S^2\]
contained in a product of annuli, such that the two canonical projections moreover induce surjective maps $H_1(L') \rightrightarrows H_1(S^2 \setminus \{ 0, \infty\}) \simeq \Z$ in homology.
\end{maincor}
In the following remark we elaborate slightly on the topological implications for the embedding of a monotone torus.
\begin{remark}
\label{rmk:fibration}
The inclusion of $L'$ produced by the above corollary induces a map
\[H_1(L') \rightarrow H_1((S^2 \setminus \{ 0, \infty\}) \times (S^2 \setminus \{ 0, \infty\})) \simeq \Z^2\]
which either is
\begin{enumerate}[label=(\arabic*)]
\item an isomorphism, or
\item a map whose image is equal to $\Z(1,\pm 1) \subset \Z^2 \simeq H_1((S^2 \setminus \{ 0, \infty\})^2)$.
\end{enumerate}
Furthermore, a Lagrangian torus $L'$ as in Case (2) above with image of $H_1(L')$ equal to $\Z(1,-1)$ can be mapped to a torus for which the image is equal to $\Z(1,1)$ by a Hamiltonian diffeomorphism, and vice versa. The Hamiltonian diffeomorphism is an explicit rotation of the second $S^2$-factor that fixes $(S^2 \times \{ 0, \infty\}) \: \cup \: (\{ 0, \infty\} \times S^2) \subset S^2 \times S^2$ set-wise.
\end{remark}
The standard monotone product torus $S^1 \times S^1 \subset S^2 \times S^2$, also called the Clifford torus, clearly is of the form described in Case (1) of Remark \ref{rmk:fibration}. The Chekanov torus is a monotone Lagrangian torus which is not Hamiltonian isotopic to the Clifford torus. In Appendix \ref{sec:appendix} we give a construction of this torus which is of the form described in Case (2) of Remark \ref{rmk:fibration}. In fact, in Section \ref{sec:classification} below we show that any monotone Lagrangian torus satisfying Case (1) of Remark \ref{rmk:fibration} is Hamiltonian isotopic to the Clifford torus.

\subsection{History of the problem}\label{sec:history}

Lagrangian tori inside symplectic four-manifolds are rather well-studied objects. On the one hand, by the Darboux's theorem, any symplectic four-manifold contains plenty of Lagrangian tori. On the other hand, any oriented null-homologous Lagrangian surface $L$ has to be a torus by the adjunction formula for Lagrangians: the tangent bundle of $L$ (with the given orientation) is isomorphic to the normal bundle (with orientation opposite to the one induced by the symplectic form together with the orientation of $L$), and hence the Euler characteristic of $L$ satisfies $\chi(L)=-[L]\cdot [L]=0$ by the assumption that $[L]=0$ in homology.

\subsubsection{Inside symplectic vector spaces}
The Clifford torus $S^1 \times S^1 \subset (\C^2=\R^4, \omega_0)$ and its rescalings were the first examples of monotone Lagrangian tori in the standard symplectic vector space. As shown by C. Viterbo in \cite{Viterbo:ANewObstruciton} as well as L. Polterovich \cite{Polterovich:Maslov}, the Maslov class on any Lagrangian torus in $(\R^4,\omega_0)$ evaluates to $2$ on some disk of positive symplectic area having boundary on the torus. Together with the $h$-principle for Lagrangian immersions due to M. Gromov \cite{Gromov:PartialDifferential} and A. Lees \cite{Lees}, this implies that all Lagrangian tori in  $(\R^4,\omega_0)$ are regular homotopic through Lagrangian immersions. In recent work K. Cieliebak and K. Mohnke \cite{Cieliebak-Mohnke:Punctured} found the following extension of the former result to arbitrary dimensions: For any Lagrangian torus inside a symplectic vector spaces or a projective space, there exists a disk with boundary on the torus on which the Maslov class takes the value $2$, and which is of positive symplectic area. In other words, the so-called Audin conjecture holds.  The proof of the latter result relies on the same splitting construction and compactness theorem as used in this paper.

In \cite{Chekanov} Y. Chekanov constructed new monotone Lagrangian tori in each $(\R^{2n},\omega_0)$, $n \ge 2$.  The corresponding monotone Lagrangian torus $L_{\OP{Ch}} \subset \R^4$ is now called the \emph{Chekanov torus}. Moreover, Chekanov proved that the Chekanov torus cannot be mapped to any (rescaling of) $S^1 \times S^1$ by a global symplectomorphism. On the other hand, $L_{\OP{Ch}}$ can be explicitly seen to be Lagrangian isotopic to $S^1 \times S^1$.

The problem of unknottedness of Lagrangian tori in $\R^4$ was first considered by K. Luttinger: see \cite{Luttinger}, where he has given strong restrictions on the smooth isotopy class of the fundamental group of the complement $\R^4 \setminus L$ of a Lagrangian torus $L\subset(\R^4,\omega_0)$. In particular, it follows from these conditions that a torus in $\R^4$ obtained by spinning a one-dimensional knot $K \subset \R^3$ admits a Lagrangian representative if and only if $K$ is the unknot.

In high-dimensional symplectic vector spaces, by contrast, there exist monotone Lagrangian tori which are not smoothly isotopic; see \cite[Section 8]{Dimitroglou-Evans}. There also exist closed simply connected four-manifolds for which infinitely many smooth isotopy classes of tori, all in a fixed homology class, admitting Lagrangian representatives. The first construction is due to S. Vidussi \cite{Vidussi}, who considered a symplectic four-manifold homotopy equivalent to $E(2)$.

H. Hofer and  K. Luttinger  also proposed a program of proving smooth unknottedness of Lagrangian tori in $(\R^4,\omega_0)$ using ideas from symplectic field theory. A major step towards realization of this program (in a more general setup of rational and ruled symplectic four-manifolds; see Section \ref{sec:historycp2}) was done in the PhD dissertation \cite{Ivrii-thesis} of A. Ivrii. The project, completed independently by E. Goodman in her PhD dissertation \cite{Goodman-thesis} and by G. Dimitroglou Rizell, is presented in the current paper.


\subsubsection{Inside the cotangent bundle of the torus and the nearby Lagrangian conjecture}
\label{sec:historycotangent}
V. Arnold has shown that any Lagrangian torus inside $(T^*\TT^n,d\lambda)$ which is not null-homologous is homologous to the zero-section, and hence homotopic to the zero-section; see \cite{Arnold:FirstSteps} as well as \cite[Theorem 4.1.1]{Audin-Lalonde-Polterovich}. The same result was later re-proven in \cite{Giroux} by Giroux for $T^*\TT^2$ using a different low-dimensional approach. In the latter four-dimensional setting Y. Eliashberg and L. Polterovich \cite{Eliashberg-Polterovich:Unknottedness} extended this result by showing that any torus in $(T^*\TT^2,d\lambda)$ that is not null-homologous in fact is smoothly isotopic to the zero-section.

The so-called \emph{nearby Lagrangian conjecture} more generally states that any closed exact Lagrangian submanifold $L \subset (T^*N,d\lambda_N)$ of the cotangent bundle of a closed manifold is Hamiltonian isotopic to the zero-section. Very little is known about to what extent this conjecture is true. However, work by M. Abouzaid and T. Kragh \cite{Kragh}, going back to a series of work by M. Abouzaid \cite{Abouzaid}, K. Fukaya, P. Seidel, and I. Smith \cite{Fukaya-Seidel-Smith:first}, \cite{Fukaya-Seidel-Smith:second}, and D. Nadler \cite{Nadler}, shows that the following strong partial result holds: The canonical projection of the cotangent bundle restricts to a homotopy equivalence $L \hookrightarrow T^*N \to N$. Combining this statement with Theorem \ref{T22} proven below, it thus follows that the nearby Lagrangian conjecture holds in the case of $(T^*\TT^2,d\lambda)$; see Theorem \ref{thm:nearby}. Note that, in this particular case, the above homotopy equivalence can also be established using the fact that there are no Lagrangian Klein bottles in $\R^4$ (and thus neither inside $T^*\TT^2$), as first shown by V. Shevchishin in \cite{Shevchishin} and later by S. Nemirovski in \cite{NoKlein}, combined with \cite{Arnold:FirstSteps}.

Previously the nearby Lagrangian conjecture was only known in the case of $T^*S^2$, as well as the relative case of $T^*\R^2$. The former case follows by the work of R. Hind in \cite{Hind} together with the fact that an exact Lagrangian embedding inside $T^*S^2$ is orientable, originally proven in \cite{Ritter} by A. Ritter (of course, this also follows from the more recent result by Abouzaid--Kragh mentioned above). The relative case of $T^*\R^2$ similarly follows from the work of Eliashberg-Polterovich in \cite{Eliashberg-Polterovich:LocalLagrangian}.

\subsubsection{Inside $\CP^2$ and the monotone $S^2\times S^2$}
\label{sec:historycp2}
The Clifford and Chekanov tori inside $(\R^4,\omega_0)$ of an appropriate monotonicity constant can be embedded inside both $(\CP^2,\omega_{\OP{FS}})$ and $(S^2 \times S^2,\omega_1\oplus\omega_1)$ as monotone Lagrangian tori. Recall that the Clifford torus is given by the product $S^1 \times S^1 \subset S^2 \times S^2$ of the equators, while we refer to Appendix \ref{sec:appendix} for a construction of the Chekanov torus. We follow the description of the Chekanov torus given in \cite{Chekanov-Schlenk}. In fact, there are different incarnations of the Chekanov torus inside the latter closed symplectic manifolds and we refer to \cite{Gadbled} for a thorough treatment by A. Gadbled.

It was shown by Y. Chekanov and F. Schlenk in \cite{Chekanov-Schlenk} that the Clifford and Chekanov tori live in different Hamiltonian isotopy classes also when considered inside $(\CP^2,\omega_{\OP{FS}})$. This result was obtained by counting the number of families of holomorphic disks of Maslov index $2$ having boundary on these tori. M. Entov and L. Polterovich were the first to distinguish the monotone Clifford and Chekanov torus in $(S^2 \times S^2,\omega_1\oplus\omega_1)$, which was done in \cite[Example 1.22]{Entov-Polterovich}. This argument is also based upon a pseudoholomorphic disk count. We also refer to \cite{FOOO:ToricDegeneration} for an alternative proof of the same statement.

Recent results by R. Vianna \cite{Vianna:first}, \cite{Vianna:second} provide a major breakthrough in the understanding of monotone Lagrangian tori. He establishes the existence of an infinite number of Hamiltonian isotopy classes of monotone Lagrangian tori inside $(\CP^2,\omega_{\OP{FS}})$. In an even more recent result \cite{Vianna:third} Vianna establishes an infinite number of distinct monotone tori inside $(S^2 \times S^2,\omega_1\oplus\omega_1)$ as well.

Finally, our result in Theorem \ref{thm:complement} can be seen as a refinement of a result by D. Auroux, D. Gayet, and J.-P. Mohsen from \cite{Auroux-Gayet-Mohsen} in the closed symplectic manifolds under consideration here. The latter result states that any Lagrangian submanifold sits inside the complement of a Donaldson hypersurface.

\subsection{Progress on the classification problem in $S^2 \times S^2$}
\label{sec:classification}
Even though the classification of Lagrangian tori up to Hamiltonian isotopy is completely open for the monotone $S^2 \times S^2$, the result obtained here provides a potential starting point.

Let $L\subset S^2 \times S^2$ be a monotone Lagrangian torus and assume that we are given a symplectic fibration $p \colon (S^2 \times S^2,\omega_1 \oplus \omega_1) \to S^2$ compatible with $L$ as produced by Theorem \ref{thm:fibration}. Recall that this is a fibration for which $L$ is fibered over the equator $S^1 \subset S^2$ of the base. In particular, the intersection of $L$ and a fiber above a point $\theta \in S^1$ on the equator is an embedded closed curve bounding the families $D_1(\theta),D_2(\theta) \subset S^2$ of symplectic disks of Maslov index 2.

In \cite{Cieliebak-Schwingenheuer}, K. Cieliebak and M. Schwingenheuer gave the following criterion for when a Lagrangian torus $L \subset S^2 \times S^2$ is Hamiltonian isotopic to the Clifford torus:
\begin{itemize}
\item The above fibration $p$ admits two symplectic sections $\sigma_1, \sigma_2$ satisfying the additional properties that
\begin{enumerate}
\item the sections $\sigma_i$, $i=1,2$, are disjoint from $L$ and do not intersect, and
\item the section $\sigma_i$ passes through the disk $D_i(\theta)$ for each $i=1,2$.
\end{enumerate}
\end{itemize}
Their result was proven using a more sophisticated version of the inflation technique described here in Section \ref{sec:inflation}.

Observe that Theorem \ref{thm:fibration} provides symplectic sections $\sigma_i$, $i=1,2$, satisfying Condition (i), but that (ii) is not necessarily satisfied.

Using Corollary \ref{cor:fibration}, any monotone Lagrangian torus $L \subset S^2 \times S^2$ can be Hamiltonian isotoped into the product of annuli
\[(T^*_{1/4\pi}S^1 \times T^*_{1/4\pi}S^1,d\lambda_{S^1}\oplus d\lambda_{S^1}) \cong ((S^2 \setminus \{ 0, \infty\}) \times(S^2 \setminus \{ 0, \infty\}),\omega_1 \oplus \omega_1).\]
Here $T_rS^1=\R/2\pi\Z \times (-r,r)$ and $\lambda_{S^1}=p\,d\theta$, where $p$ denotes the standard coordinate on the $(-r,r)$-factor, and $\theta \colon S^1 \to \R /2\pi\Z$ denotes the standard angular coordinate.

Moreover, it is readily seen that both Conditions (i) and (ii) above are satisfied in the case when Case (1) in Remark \ref{rmk:fibration} holds, i.e.~when the map $H_1(L) \to H_1(T^*_{1/4\pi}S^1 \times T^*_{1/4\pi}S^1)$ is an isomorphism. Namely, by Theorem \ref{thm:fibration} we construct a symplectic fibration $p$ for which the two spheres $\{ 0, \infty\} \times S^2$ are fibers, while the two spheres $S^2 \times \{ 0, \infty\}$ are sections. By topological reasons it now follows that the latter sections in fact fulfill Conditions (i) and (ii).

To conclude, the case which remains to be treated in order to provide a classification of monotone Lagrangian tori in $(S^2 \times S^2,\omega_1 \oplus \omega_1)$ up to Hamiltonian isotopy is when Case (2) of Remark \ref{rmk:fibration} is satisfied. In other words, the case when the torus $L \subset T^*_{1/4\pi}S^1 \times T^*_{1/4\pi}S^1$ satisfies the property that the image of $H_1(L)$ is generated by the diagonal class of $H_1(T^*_{1/4\pi}S^1 \times T^*_{1/4\pi}S^1)=H_1(T^*_{1/4\pi}S^1) \oplus H_1(T^*_{1/4\pi}S^1)$.

\subsection{Strategy of the proof}
The main results in this paper are shown by considering limits of pseudoholomorphic foliations under the splitting construction from \cite{EGH}, where this construction has been applied to a hypersurface of contact type corresponding to an embedding of the unit normal bundle of the Lagrangian torus $L$. The latter hypersurface is contactomorphic to the unit cotangent bundle $S^*L$, and its existence follows by an elementary application of Weinstein's Lagrangian neighborhood theorem carried out in Section \ref{sec:splitting-prelim}. The limits of pseudoholomorphic curves under this construction consists of so-called pseudoholomorphic buildings contained inside the split symplectic manifold
\[ (X \setminus L,\omega) \: \cup \: (T^*L,d\lambda) \]
of two components, where each component has non-compact cylindrical ends.

For the symplectic manifolds $(X,\omega)$ under consideration, the existence of pseudoholomorphic foliations was established by Gromov \cite{Gromov}. The compactness theorem for spaces of pseudoholomorphic curves under the splitting construction was shown in \cite{BEHWZ} by F. Bourgeois, Y. Eliashberg, H. Hofer, C. Wysocki, and E. Zehnder, and independently in \cite{Cieliebak-Mohnke:Compactness} by K. Cieliebak and K. Mohnke; see Theorem \ref{thm:compactness} for the formulation that we will be using.

Roughly speaking, the compactness theorem states that sequences of pseudoholomorphic curves converge to a split pseudoholomorphic curve (also called pseudoholomorphic building) when performing the splitting construction. We refer to Section \ref{sec:building} for the definition for a split curve, which consists of a number of pseudoholomorphic curves contained in the different components of the above split symplectic manifold.

The Hamiltonian isotopy provided by Theorem \ref{thm:complement}, which places a given Lagrangian torus inside the complement of the standard holomorphic divisor $D_\infty$, is constructed by finding a suitable \emph{pseudoholomorphic divisor} in the complement of the torus. Gromov's classification of pseudoholomorphic foliations \cite{Gromov}, together with the classical result Corollary \ref{cor:hamiso}, can then be used to construct the sought Hamiltonian isotopy of the torus into the complement of the standard divisor. The existence of the pseudoholomorphic divisor follows since most limits of the leaves of the foliation will be non-split spheres.  (A sphere which is not split is a compact pseudoholomorphic sphere contained inside $X \setminus L$.) Here we rely on the additivity of the Fredholm index together with a transversality result. We refer to Section \ref{sec:proofcomplement} for more details.

The Lagrangian isotopy connecting any two given Lagrangian tori in Theorem \ref{thm:Lagrangian-isotopy} is proven by a further study of the above limit of the pseudoholomorphic foliations. Namely, we show that the split pseudoholomorphic spheres contain a designated component in $X \setminus L$ which is a pseudoholomorphic plane. Furthermore, the plane is of expected dimension one and cannot bubble -- it's moduli space is compact. By the choices made in the splitting construction, this plane is moreover asymptotic to a geodesic on $L$ for the flat metric. The totality of the one-dimensional family of such planes is shown to form a smoothly embedded solid torus having boundary equal to $L$. Here we need the automatic transversality result \cite{Wendl} by C. Wendl together with the asymptotic intersection results shown in \cite{Hind-Lisi} by R. Hind and S. Lisi. We also make heavy use of positivity of intersection for pseudoholomorphic curves; see the work \cite{McDuff:Local} by D. McDuff.

Observe that a \emph{smooth} isotopy of the Lagrangian torus to a standard representative can be constructed using the produced solid torus. In order to upgrade this to a \emph{Lagrangian} isotopy, we must in addition ensure that the characteristic distribution of the above solid torus induces a trivial monodromy map on its symplectic disk leaves; such a solid torus is foliated by Lagrangian tori together with its core circle. As shown by Theorem \ref{adjust}, the existence of a Lagrangian isotopy now follows. The needed modification of the solid torus is performed using Theorem \ref{monodromyidentity}, which is based upon the explicit construction in Lemma \ref{diskmonodromy}. However, we must first apply the so-called inflation technique provided by Theorem \ref{inflation}. Inflation guarantees that there is enough space in order to perform the modification.

\subsection{Acknowledgments}
All three authors are deeply grateful to Yakov Eliashberg for sharing his insight concerning both the problem and the techniques from symplectic field theory. His guidance has been crucial.

\section{The splitting construction}

Let $(X,\om)$ denote either $(\CP^2,\omega_{\OP{FS}})$ with the Fubini--Study symplectic form, or the monotone $(S^2\times S^2,\omega_1 \oplus \omega_1)$ endowed with the split symplectic form, and let $L\subset (X,\om)$ be a Lagrangian torus. Consider a hypersurface of contact type which is contactomorphic to the unit cotangent bundle $S^*L$, and which represents an embedding of the unit normal bundle of $L$. The central idea used in this paper is to study the symplectic topology of $L$ by applying the \emph{splitting construction} to this hypersurface, sometimes also called \emph{stretching the neck}, and then to study the obtained limits of pseudoholomorphic spheres in $X$. The splitting construction was first described in~\cite[Section 1.3]{EGH} in the setting of symplectic field theory. Below we present the details of this procedure and give a survey of known facts specialized to the case at hand; for more details we refer to~\cite{EGH}, \cite{BEHWZ} and \cite{Cieliebak-Mohnke:Compactness}.

The splitting construction applied to the unit normal bundle of a Lagrangian submanifold has previously been used in \cite[Theorem 1.7.5]{EGH}, \cite{Hind}, \cite{Dimitroglou-Evans}, \cite{Hind-Lisi}, and \cite{Cieliebak-Mohnke:Punctured}, among others. Manifestly, it is an efficient tool for obtaining strong obstructions to Lagrangian embeddings inside uniruled symplectic manifolds.

\subsection{Preliminaries}\label{sec:splitting-prelim}
We fix coordinates $(\theta_1, \theta_2, p_1, p_2)=(\boldsymbol{\theta},\mathbf{p})$ on $T^*\TT^2 = (S^1)^2 \times \R^2$, where $\theta_i \colon S^1 \times S^1 \to \R/2\pi\Z$ denotes the standard angular coordinate coordinate on the $i$:th $S^1$-factor, while $p_i$ is the standard coordinate on the $i$:th $\R$-factor, $i=1,2$. The Liouville one-form can then be written as
\[\lambda = p_1 d\theta_1 + p_2 d\theta_2 \in \Omega^1(T^*\TT^2),\]
and $d\lambda$ is the canonical symplectic form on $T^*\TT^2$. For each $r>0$ we also consider the open co-disk bundle
\[T^*_r\TT^2 = \{ ||{\bf p}|| < r \} \subset T^*\TT^2,\]
as well as the corresponding co-sphere bundle $S^*_r\TT^2 = \partial\overline{T^*_r\TT^2} =  \{ ||{\bf p}|| =  r \}$.

The hypersurface $S^*_r\TT^2 \subset (T^*\TT^2,d\lambda)$ is a hypersurface of contact type, since the vector field $p_1 \partial_{p_1}+p_2 \partial_{p_2}$ symplectically dual to the Liouville-form $\lambda$ is transverse to it. This means that the pull-back $\alpha=\lambda|_{T(S^*_r\TT^2)}$ of $\lambda$ is a \emph{contact form}, i.e.~$\alpha \wedge d\alpha$ is a volume form. We also write
\[(S^*\TT^2,\alpha_0):=(S^*_1\TT^2,\lambda|_{T(S^*_1\TT^2)})\]
for the corresponding contact form on the spherical cotangent bundle.

Using the identification $S^*\TT^2=(S^1)^2 \times S^1 \subset (S^1)^2 \times \R^2$, with the angular coordinate $\theta$ on the last factor $S^1 \subset \R^2$ (the fiber), we can write the contact form as $\alpha_0 = \cos(\theta)d\theta_1 + \sin(\theta) d\theta_2$. A choice of contact form induces the so-called \emph{Reeb vector field} $R$ on $S^*\TT^2$, which is uniquely determined by the equations $i_{R}\alpha_0=1$ and $i_{R}d\alpha_0=0$. In our case we have $R_{\boldsymbol{\theta}_0,\theta} = \cos(\theta) \partial_{\theta_1} + \sin(\theta) \partial_{\theta_2}$.

The canonical projection  $S^*\TT^2 \rightarrow \TT^2$ induces a one-to-one correspondence between the Reeb orbits of $R$ and oriented geodesics on $\TT^2$ for the flat metric obtained by pushing forward the Euclidean metric under the covering map $\R^2 \to \R^2 /(2\pi\Z)^2 =\TT^2$. The periodic Reeb orbits (of different multiplicities) correspond to $\theta$ for which $\tan(\theta) \in \Q\cup\{\infty\}$ and form $1$-dimensional manifolds. Furthermore, these manifolds of orbits are non-degenerate in the Morse--Bott sense; see \cite{Bourgeois:AMorseBott}.

Since the Bott manifolds of periodic oriented geodesics on $\TT^2$ correspond bijectively to the non-zero homology classes $\eta \in H_1(\TT^2) \setminus \{0\},$ the analogous statement thus holds for the families of periodic Reeb orbits as well. We use $\Gamma_\eta \cong S^1$ to denote the family of periodic Reeb orbits projecting to the oriented geodesics in homology class $\eta \in H_1(\TT^2) \setminus \{0\}$.

By Weinstein's Lagrangian neighborhood theorem originally proven in \cite{Weinstein:Lagrangian} (also, see \cite{McDuff:IntroSympTop}), any Lagrangian embedding $\psi \colon L \hookrightarrow (X,\omega)$ can be extended to a symplectic embedding
\begin{gather*}
\Psi: (T^*_{4\epsilon} L ,d\lambda) \hookrightarrow (X,\omega),\\
\Psi|_{0_L}=\psi,
\end{gather*}
of a neighborhood $T^*_{4\epsilon}L \subset T^*L$, $\epsilon>0$, of the zero-section $0_L \subset T^*L$. In this way, any Lagrangian embedding $L \hookrightarrow (X,\omega)$ gives rise to an embedding $\Psi(S^*_{3\epsilon} L) \subset (X,\omega)$ of its unit normal bundle as a hypersurface of contact type. This hypersurface divides the symplectic manifold $(X,\omega)$ into two components diffeomorphic to $X \setminus L$ and $T^*\TT^2$, respectively.

\subsection{Symplectic manifolds with cylindrical ends}
\label{sec:cylindrical}

Let $(Y, \alpha)$ be a contact manifold. The \emph{symplectization} of $Y$ is the symplectic manifold $(\R \times Y, d(e^t\alpha))$, where $t$ is the coordinate on the factor $\R$. 

We now define the notion of a non-compact \emph{symplectic manifold with cylindrical ends}. This is a symplectic manifold $(W,\omega)$ containing a compact domain $\overline{W}$ with contact boundary, for which the complement $(W \setminus \mathrm{int}\overline{W},\omega)$,
 is symplectomorphic to half symplectizations
\begin{gather*}
((-\infty, A] \times Y_-, d(e^t\alpha_-)),\\
([B,+\infty) \times Y_+, d(e^t\alpha_+)).
\end{gather*}
Here either of $Y_\pm$, but not both, may be empty. These half-cylinders are called the \emph{concave} and \emph{convex cylindrical ends} of $(W,\omega)$, respectively. In our case both the contact manifolds $Y_\pm$ as well as the symplectic manifold $W$ will always be connected.

A compatible almost complex structure $J$ on the symplectization $(\R \times Y, d(e^t\alpha))$ is called \emph{cylindrical} given that
\begin{itemize}
\item $J$ is translation invariant,
\item $J\partial_t$ is the Reeb vector field associated to $\alpha$, and
\item $J(\ker\alpha) = \ker \alpha$.
\end{itemize}
On a symplectic manifold with cylindrical ends, we will always consider tame almost complex structures that are cylindrical outside of a compact subset. Such almost complex structures exist and form a contractible space by a standard result \cite{Gromov}. 

In this paper only the following symplectic manifolds with cylindrical ends are considered:
\begin{example}
\label{ex:cyl}
\begin{enumerate}
\item The cotangent bundle $(T^*\TT^2, d\lambda)$. This is an exact symplectic manifold having a single convex cylindrical end over the contact manifold $(S^*\TT^2,\alpha_0)$;
\item The symplectization $(\R\times S^*\TT^2, d(e^t\alpha_0))$. This is an exact symplectic manifold having a single convex cylindrical end together with a single concave cylindrical end, both over the contact manifold $(S^*\TT^2,\alpha_0)$.
\item The symplectic manifold $(X\setminus L, \omega)$. This is a symplectic manifold with a single negative cylindrical end over the contact manifold $(S^*\TT^2,\alpha_0)$. To that end, we note that there is an exact symplectomorphism
\begin{gather*}
(\R\times S^*\TT^2, d(e^t\alpha_0)) \xrightarrow{\cong} (T^*\TT^2 \setminus 0_L,d\lambda),\\
(t,\boldsymbol{\theta},\mathbf{p}) \mapsto (\boldsymbol{\theta},e^t\mathbf{p})
\end{gather*}
identifying the symplectization with the complement of the 0-section in the cotangent bundle.
\end{enumerate}
\end{example}

\subsection{Punctured pseudoholomorphic spheres}
Let $(W,\omega)$ be a symplectic manifold with convex and concave cylindrical ends over the contact manifolds $(Y_+,\alpha_+)$ and $(Y_-,\alpha_-)$, respectively. Let $J$ be a tame almost complex structure on $W$ which is cylindrical on its ends. A \emph{punctured pseudoholomorphic sphere} in $(W,\omega)$ consists of the following data:
\begin{itemize}
\item The Riemann sphere $(S^2, i)$ with distinct points $Z = \{ z_1, \dots, z_n \} \subset S^2$, called \emph{punctures}. We denote by $\dot{S^2} := S^2 \setminus Z$ the corresponding punctured Riemann sphere.
\item A proper pseudoholomorphic map $\phi: \dot{S^2} \rightarrow W$ that satisfies a non-linear Cauchy-Riemann equation $d\phi \circ i = J \circ d\phi$;
\item A (not necessarily simple) periodic Reeb orbit $\gamma_k \in (Y_\pm,\alpha_\pm)$ assigned to each of the punctures $z_k$, $k=1,\dots,n$; the Reeb orbit is parametrized by integrating the Reeb vector field, and we denote by $T_k >0$ the period of $\gamma_k$. Observe that each Reeb orbit $\gamma_k$ lives in either the convex or concave cylindrical end; in the former case $z_k$ is called a \emph{positive puncture of $\phi$}, while in the latter case it is called a \emph{negative puncture of $\phi$}.
\item The above map $\phi$ is finally required to be \emph{asymptotic} to $\gamma_k$ at each puncture $z_k$ in the following sense. In local holomorphic coordinates $D^2 \setminus \{0\} \subset \C$ near each puncture $z_k$, we can write $\phi = (a, u) : D^2 \setminus 0 \rightarrow \R \times Y_\pm$, where $\lim_{\rho \to  0} a(\rho e^{i\theta}) = \infty$ (resp. $-\infty$) and $\lim_{\rho\rightarrow 0} u(\rho e^{i\theta}) = \gamma_k(T_k\theta)$ (resp. $\gamma_k(-T_k\theta)$) holds given that $z_k$ is a positive (resp. negative) puncture. These limits are required to hold in the uniform metric.
\end{itemize} 
It is shown in \cite{HWZI}, \cite{HWZ} that a proper pseudoholomorphic sphere satisfies the above asymptotic properties if and only if its so-called Hofer energy is finite. Pseudoholomorphic spheres as above are therefore usually called \emph{finite energy spheres.} All pseudoholomorphic spheres considered in this paper will be assumed to be of finite energy.

A pseudoholomorphic sphere having a single puncture will be referred to as a \emph{plane} while a pseudoholomorphic sphere having precisely two punctures will be referred to as a \emph{cylinder}.

We note that the symplectic area $0<\int_\phi \omega\le +\infty$ of a non-constant punctured pseudoholomorphic curve is positive (and finite if and only if it has no positive puncture). Moreover, inside the symplectization $(\R \times Y,d(e^t\alpha))$ endowed with a cylindrical almost complex structure, a punctured pseudoholomorphic curve has a non-negative $d\alpha$-energy, a quantity defined by the integral
\[ 0\le \int_\phi d\alpha = \ell(\gamma_1^+) + \hdots + \ell(\gamma_{k_+}^+)-(\ell(\gamma_1^-) + \hdots + \ell(\gamma_{k_-}^-)).\]
Here, we have used $\{\gamma_i^\pm\}$ to denote the positive and negative punctures of $\phi$, and
\[\ell(\gamma_i^\pm):=\int_{\gamma_i^\pm}\alpha>0\]
for the length of a Reeb orbit. The expression for the $d\alpha$-energy in terms of the Reeb orbit lengths is a simple application of Stokes' theorem. Recall the standard fact that the $d\alpha$-energy vanishes if and only if $\phi$ is either a trivial pseudoholomorphic cylinder $\R \times \gamma \subset \R \times Y$ over a periodic Reeb orbit $\gamma$, or a (possibly branched) multiple cover of such a cylinder. Namely, the tameness of a cylindrical almost complex structure restricted to the contact planes implies that $d\alpha$ pulls back to a positive form wherever the curve is not tangent to the Reeb vector field.

\subsection{Split pseudoholomorphic curves}
\label{sec:building}
A \emph{punctured nodal sphere} is a finite union of Riemann spheres together with a finite set of punctures, a designated subset of which are called \emph{nodes}. A fixed-point free involution of the nodes is also part of the data of a punctured nodal sphere (and in particular there is an even number of nodes), and we require that the manifold obtained by performing connected sums at all pairs of nodes related by the involution is a \emph{sphere}.

A \emph{pseudoholomorphic building} or a \emph{split pseudoholomorphic curve} inside the \emph{split symplectic manifold}
\[(X \setminus L,\omega) \:\: \sqcup \:\: (T^*L,d\lambda) \]
is a collection of punctured holomorphic spheres, parametrized by a punctured nodal sphere, with each punctured sphere pertaining to a certain symplectic manifold with cylindrical ends called a \emph{level}. Fix a tame almost complex structure $J_{\OP{cyl}}$ on $(\R \times S^*L,d(e^t\alpha_0))$ and let $J_\infty$ and $J$ denote tame almost complex structures on $X\setminus L$ and $T^*L$ coinciding with a cylindrical almost complex structure $J_{\OP{cyl}}$ outside of a compact subset of their respective cylindrical ends. In our case, the components of the building and the corresponding manifolds are as follows.
\begin{itemize}
\item {\bf Top level:} A finite number of punctured $J_\infty$-holomorphic spheres in $X\setminus L$.
\item {\bf Middle levels:} A finite (possibly zero) number of consecutive ordered levels, each level consisting of a non-zero number of punctured $J_{\OP{cyl}}$-holomorphic spheres in $\R\times S^*L$. Every level is moreover required to contain at least one component which is not a trivial cylinder over a periodic Reeb orbit.
\item {\bf Bottom level:} A finite (possibly zero) number of punctured $J$-holomorphic spheres in $T^*L$.
\end{itemize}
In addition, the enumeration of the levels is required to satisfy the following properties. Two components of the nodal sphere containing an orbit of the involution (i.e.~which share a node) belong to two consecutive levels $i$ and $i+1$, where their respective parametrizations $\phi_i$ and $\phi_{i+1}$ are required to satisfy the following behavior near the node. The node is a positive and negative puncture of $\phi_i$ and $\phi_{i+1}$, respectively, asymptotic to the \emph{same} parametrized periodic Reeb orbit in $(S^*L,\alpha_0)$.

The properties of a pseudoholomorphic building imply that the components in the top level $X \setminus L$ can be completed by adding chains in $L$ obtained from the images of the remaining components under the projections $\R \times S^*L \to L$ and $T^*L \to L$ (suitably compactified at each puncture). In the case when all of the punctures of the components of the building correspond to nodes, such a split curve is called a \emph{split sphere}. In this manner we produce a cycle in $X$ and, by the \emph{homology class} of a split sphere we mean the latter homology class.

By pseudo-convexity and exactness, it follows that the middle and bottom levels may not contain any (non-constant) component without punctures. In other words, any split sphere must have a component in the top level. In addition, the following crucial property is also satisfied for split spheres in the case under consideration. Since every periodic Reeb orbit in $(S^*\TT^2,\alpha_0)$ is induced by the flat metric on $\TT^2$, it is non-trivial in homology even when considered inside $S^*\TT^2= \partial \overline{T^*_1\TT^2}\subset\overline{T^*_1\TT^2}$. It hence follows by the asymptotic properties for a punctured pseudoholomorphic sphere that any component inside the middle and bottom levels must have \emph{at least two} punctures.

A split sphere consisting of more than one non-empty level will be called \emph{broken}. A consequence of the above discussion is that a broken split sphere must consist of \emph{at least two} components in its top level that are planes. On the contrary, a split sphere will be called \emph{non-broken} if it consists of a single component; this component necessarily lives inside the top level and has no punctures (i.e.~it is an ordinary pseudoholomorphic sphere).

\subsection{Producing a split symplectic manifold by stretching the neck}
\label{sec:splitting}

Here we describe the splitting construction in the special case when the contact hypersurface is taken to be the unit cotangent bundle of a Lagrangian embedding of a torus. Let $L \subset (X,\omega)$ be a Lagrangian torus, and fix an identification of a Weinstein neighborhood of $L$ in $(X,\omega)$ given by the symplectic embedding
\[ \Psi \colon (T^*_{4\epsilon} \TT^2,d\lambda) \hookrightarrow (X,\omega),\]
which restricts to the embedding of $L$ along the zero-section $0_{\TT^2} \subset T^*\TT^2$; see Section~\ref{sec:splitting-prelim} for more details. Also, consider the induced dividing hypersurface $\Psi(S^*_{3\epsilon}\TT^2) \subset (X,\omega)$ of contact type. After rescaling all symplectic forms by multiplication with $1/\epsilon$, it suffices to consider the case $\epsilon=1$.

In Section \ref{sec:cylinders} we construct an explicit cylindrical almost complex structure $J_{\OP{cyl}}$ on $(\R \times S^*_1\TT^2,d(e^t\alpha_0)) \cong (T^*\TT^2 \setminus 0_{\TT^2},d\lambda)$, as well as an explicit tame almost complex structure $J_{\OP{std}}$ on $T^*\TT^2$ coinciding with $J_{\OP{cyl}}$ outside of the subset $T^*_2\TT^2 \subset T^*\TT^2$.

Using these explicitly defined almost complex structures, we fix a tame almost complex structure $J_\infty$ on $(X\setminus L,\omega)$ coinciding with $J_{\OP{cyl}}$ in the coordinates near $L$ given by $\Psi$ above. We then consider a family $J_\tau$, $\tau \ge 0$, of tame almost complex structures on $(X,\omega)$ determined uniquely by the following properties:
\begin{itemize}
\item In the complement of $\Psi(T^*_4\TT^2) \subset X$, we have $J_\tau \equiv J_\infty$ for all $\tau \ge 0$;
\item Inside $\Psi(T^*_2\TT^2) \subset X$, we have $J_\tau \equiv J_{\OP{std}}$ for all $\tau \ge 0$; and
\item Inside $\Psi(T^*_4\TT^2 \setminus T^*_2\TT^2) \subset X$, using the above exact symplectic identification of $(T^*_4\TT^2 \setminus T^*_2\TT^2,d\lambda)$ with
\[ ([\log{2},\log{4}) \times S^*\TT^2,d(e^t\alpha_0))\]
(see Part (3) of Example \ref{ex:cyl}) we identify $J_{\tau}$ with the pull-back of $J_{\OP{cyl}}$ under a diffeomorphism induced by an identification
\[[\log{2},\log{4})\cong [\log{2},\log{4}+\tau)\]
of the first factor.
\end{itemize}
The compactness theorem~\cite{Gromov}, \cite{HWZ}, \cite{BEHWZ}, \cite{CM2} shows that sequences of $J_\tau$-holomorphic spheres in $X$ have convergent subsequences to split pseudoholomorphic spheres in the following sense as $\tau\rightarrow\infty$ .
\begin{theorem}[\cite{BEHWZ}, \cite{CM2}]
\label{thm:compactness}
Consider a sequence $u_i$ of $J_{\tau_i}$-holomorphic spheres in $(X,\omega)$ in a fixed homology class $A \in H_2(X)$, where $\tau_i \to +\infty$ as $i \to \infty$. There exists a subsequence that converges to a pseudoholomorphic building in the split symplectic manifold
\[(X \setminus L,\omega) \:\: \sqcup \:\: (T^*L,d\lambda) \]
endowed with the almost complex structures $J_\infty$ and $J_{\OP{std}}$, respectively, which is a split sphere in the homology class $A$. Moreover, the subsequence is $C^\infty$-uniformly convergent on compact subsets of the complement of the nodes of the limit domain (this is a nodal sphere).
\end{theorem}
We refer to the two cited papers for detailed statements.

\section{Index computations and the proof of Theorem \ref{thm:complement}}
\label{sec:index}
Theorem \ref{thm:complement} is proven by analyzing the dimensions of the moduli spaces of the components of a split sphere in the case when its homology class is one of minimal symplectic area. Recall that the expected dimension of a moduli space is given by the Fredholm index of the corresponding linearized problem, where the linearization is performed at a solution in the moduli space. For that reason we start by recalling properties of the Fredholm index for punctured pseudoholomorphic spheres in the symplectic manifolds under consideration.

When talking about the Fredholm index of an asymptotic problem, we always consider the problem with \emph{unconstrained} ends in the Bott manifold of periodic Reeb orbits. In the case of a transversely cut out solution, this index is hence equal to the dimension of the moduli space of unparametrized curves in a neighborhood of this solution, for which the ends moreover are allowed to move freely inside the Bott manifolds.

\subsection{The Fredholm index}

In the following we let $(W,\omega)$ be a non-compact symplectic manifold with cylindrical ends over $(S^*\TT^2,\alpha_0)$, as described in Section \ref{sec:cylindrical}. Fix the choice of a symplectic trivialization $\Phi$ of the contact planes $\ker \alpha_0$ on the cylindrical ends. Observe that there is an induced symplectic trivialization of $TW =T(\R\times S^*\TT^2) \cong \C \oplus \ker \alpha_0$ on the cylindrical ends.  We denote by $c_{1,\OP{rel}}^\Phi$ the relative first Chern class of the complex bundle $TW \to W$ determined by this trivialization. Concretely, for a punctured pseudoholomorphic curve $u \colon \dot\Sigma \to W$, the number $c_{1,\OP{rel}}^\Phi(u)$ is defined to be the algebraic number of zeroes of a generic section of the line bundle $u^*(TW) \wedge_\C u^*(TW) \to \dot\Sigma$, where we require the section to be constant and non-vanishing close to the punctures with respect to the trivialization of $u^*(TW) \wedge_\C u^*(TW)$ induced by $\Phi$.

For a Bott manifold $\Gamma$ of Reeb periodic orbits, together with a complex trivialization $\Phi$ of the contact planes $\ker \alpha_0$ along $\Gamma$, recall the definition of the Conley--Zehnder index given in e.g.~\cite[Section 3.2]{Wendl}. In the non-degenerate case, i.e.~when $\dim \Gamma=0$, this index is the classical Conley--Zehnder index defined in e.g.~\cite[Remark 5.3]{MaslovIndex}. In the degenerate case the index is defined by, first, perturbing the degenerate asymptotic operator corresponding to the linearized Reeb flow by adding the linear operator $\epsilon \id_{\ker \alpha_0}$ for sufficiently small number $\epsilon>0$ and, second, computing the classical Conley--Zehnder index for the perturbed problem (which now is generic). The resulting index will be denoted by $\mu_{\OP{CZ}}^\Phi(\Gamma;\epsilon) \in \Z$.

The Conley--Zehnder index was also generalized directly to the degenerate case in \cite{MaslovIndex}, and is then usually called the Robbin--Salamon index $\OP{RS}^\Phi(\Gamma)$. In the current setting where $\dim \Gamma=1$, we have the identity
\[\mu_{\OP{CZ}}^\Phi(\Gamma;\epsilon)=\OP{RS}^\Phi(\Gamma)+1/2\]
relating these two indices.

The Fredholm index for a punctured pseudoholomorphic sphere in the current setting is well known; see e.g.~\cite{Bourgeois:AMorseBott}, \cite[(2.1)]{Wendl}, \cite[Theorem 7.1]{Hind-Lisi}, or \cite[(3)]{Cieliebak-Mohnke:Punctured}. Assume that the almost complex structure is cylindrical with respect to the contact form $\alpha_0$ on the cylindrical ends of $(W,\omega)$. Assume that we are given a punctured pseudoholomorphic sphere $u \colon \dot{S^2} \to W$ having positive punctures asymptotic to periodic Reeb orbits in the families $\Gamma^+_1,\hdots,\Gamma^+_{k^+}$ and negative punctures asymptotic to periodic Reeb orbits in the families $\Gamma^-_1,\hdots,\Gamma^-_{k^-}$. Its Fredholm index is then given by
\begin{equation}
\ind(u)=-2+k^++k^-+\sum_{i=1}^{k^+}\mu_{\OP{CZ}}^\Phi(\Gamma^+_i;\epsilon)-\sum_{i=1}^{k^-}(\mu_{\OP{CZ}}^\Phi(\Gamma^-_i;\epsilon)-1)+2c_{1,\OP{rel}}^\Phi(u),\label{eq:index}
\end{equation}
for any choice of trivialization $\Phi$ as above.

In the case when the symplectic manifold $(W,\omega)$ under consideration has a single concave end, the index formula can be seen to take the following convenient form. Observe that such a symplectic manifold is of the form $(W,\omega)=(X \setminus L,\omega)$ for a closed symplectic manifold $(X,\omega)$, where $L \subset (X,\omega)$ is a Lagrangian torus. The appropriate compactification of a punctured pseudoholomorphic sphere $u$ in $X \setminus L$ having punctures asymptotic to the families $\Gamma_1,\hdots,\Gamma_k$ of Reeb orbits produces a surface $\overline{u}$ in $X$ having boundary on $L$, where the boundary components of this compactification moreover are equal to the geodesics corresponding to these asymptotic orbits. A standard calculation (see \cite{Viterbo:ANewObstruciton} or \cite[Lemma 2.1]{Cieliebak-Mohnke:Punctured}) shows that the Maslov class $\mu_L \in H^2(X,L)$ of $L$ evaluates to
\begin{equation}
\label{eq:maslov}
\mu_L(\overline{u})=-\sum_i(\mu_{CZ}^\Phi(\Gamma;\epsilon)-1)+2c_{1,\OP{rel}}^\Phi(u)
\end{equation}
given any trivialization $\Phi$ of the contact planes as specified above.

We proceed with the following standard calculation of the Conley--Zehnder index.
\begin{lemma}\label{CZ} Let $(W,\omega)$ denote either the symplectic manifold $(T^*\TT^2,d\lambda)$ or $(\R \times S^*L,d(e^t\alpha_0))$. Using the complex trivialization $\Phi$ of the contact planes induced by complexifying the trivialization of $\TT^2$, it follows that $c_{1,\OP{rel}}^\Phi(u)=0$ holds for all punctured curves, while the Conley--Zehnder index of any Bott manifold of periodic Reeb orbits satisfies $\mu_{\OP{CZ}}^\Phi(\Gamma;\epsilon)=1$.
\end{lemma}
\begin{proof}
This index computation has been carried out in e.g.~\cite[Appendix A]{Hind-Lisi}. One can either perform a direct computation, or use the fact that the Robbin--Salamon index relates to the Morse index $\iota_\mu$ and nullity $\iota_\nu$ of the corresponding geodesics on $L$ via the formula
\[ \mu_{\OP{CZ}}^\Phi(\Gamma;\epsilon) = \OP{RS}^\Phi(\Gamma) +1/2 = \iota_\mu+(1/2)\iota_\nu+1/2\] 
in \cite[Equation 60]{Cieliebak-Frauenfelder}. In the case under consideration we have $\iota_\mu=0$, since each geodesic is of minimal length in its homology class, while $\iota_\nu=\dim \Gamma=1$.
\end{proof}

\subsection{Non-negativity results}
In the four-dimensional setting that we are considering here, the Fredholm index tends to be non-negative (at least generically). In this section we establish several results of this type. We start with the following direct application of the index Formula \eqref{eq:index} together with Lemma \ref{CZ}.
\begin{lemma}\label{lm:index-formula}
A punctured pseudoholomorphic sphere $u$ with unconstrained ends in either $T^*L$ or $\R \times S^*L$, having $k_-$ negative and $k_+$ positive punctures, has Fredholm index
\begin{equation}
\ind(u) = -2+2k_++k_-.\label{eq:indexinside}
\end{equation}
Since there are no contractible Reeb orbits, $k_+ + k_- \ge 2$, and this index thus is always positive.
\end{lemma}
The following non-negativity result will also be crucial.
\begin{lemma}\label{complementpositive}
Under assumption of regularity of $J_\infty$ for somewhere injective curves in the complement $X \setminus L$, we have $\ind(u) \ge 0$ for any punctured pseudoholomorphic sphere. Moreover, if the domain of $u$ is a plane, it follows that $\ind(u)\ge 1$, while $\ind(u)\ge 3$ holds in the case when this plane is a non-trivial branched cover.
\end{lemma}
\begin{proof}
If $u$ is a punctured $J_\infty$-holomorphic curve which is somewhere injective then, by our assumption of regularity, it follows that $\ind(u)\geq 0$. We are left to consider the case of a punctured sphere $\widetilde u$ being a $d$-fold branched cover of a simply covered punctured sphere $u$. Write
$$b:=\sum\limits_p (m_p-1),$$
where the sum is taken over all branch points counted with multiplicities $m_p>1$. Let $k_u$ and $k_{\widetilde u}$ be the number of punctures of $u$ and $\widetilde u$, respectively. By the Riemann--Hurwitz formula we have $2=d2-b$, and one checks that the number of punctures satisfies $dk_u-k_{\widetilde u} \le b$. Since the Maslov class evaluates to $\mu_L(u)=d\mu_L(\widetilde u)$ on the corresponding compactifications, Formula \eqref{eq:maslov} together with index Formula \eqref{eq:index} now gives the inequality
\begin{eqnarray*}
\lefteqn{\ind(\widetilde u)=}\\
& = &-2+k_{\widetilde u}+\mu_L(\widetilde u) \\
& = & d(-2+\mu_L(u))+k_{\widetilde u}+b\\
& \ge &  d(-2+k_u+\mu_L(u)) = d \ind(u) \geq 0,
\end{eqnarray*}
which establishes the first claim.

We finish by considering the case when $u$ is a plane. Since $L$ is orientable, the Maslov class evaluated on its compactification is \emph{even}. By Formulas \eqref{eq:maslov} and \eqref{eq:index}, together with $\ind(u) \ge 0$, we conclude that $\mu_L(u) \ge 2$ and hence $\ind(u) \ge 1$. In the case when $u$ is a non-trivial branched cover of a plane, we moreover conclude that $\mu_L(u) \ge 4$ and hence $\ind(u) \ge 3$.
\end{proof}
As a simple consequence we now obtain:
\begin{lemma}\label{lm:simple-orbits}
Any pseudoholomorphic plane inside $X \setminus L$ of index 1 is simply covered. If this plane, moreover, is contained inside $(X \setminus D_\infty,\omega) \subset (\R^4,\omega_0)$, then its asymptotic orbit is also simply covered.
\end{lemma}
\begin{proof}
The first statement is simply a reformulation of Lemma \ref{complementpositive}. The second statement follows from Formula \eqref{eq:index} for the Fredholm index, expressed in terms of the Maslov index using Formula \eqref{eq:maslov}, together with an application of the connecting isomorphism $H_2(\R^4,L) \xrightarrow{\simeq} H_1(L)$ in the long exact sequence of a pair.
\end{proof}
Combining the above results with the fact that there are no contractible Reeb orbits, we obtain the following crucial restriction on the components of a split pseudoholomorphic sphere.
\begin{proposition}\label{prop:positive-index}
Assume that we are given a generic choice of a tame almost complex structure $J_\infty$ on $X \setminus L$ as above, and consider a split pseudoholomorphic sphere which:
\begin{itemize}
\item In the case $X=S^2 \times S^2$ is in either of the homology classes $A_1,A_2 \in H_2(S^2 \times S^2)$, corresponding to $[S^2 \times \{\pt\}]$ and $[\{\pt\}\times S^2]$, respectively; or
\item In the case of $X=\CP^2$ is in the homology class of the generator of $H_2(\CP^2)$ in degree one, while this building moreover is required to pass through a fixed \emph{generic} point $p \in \CP^2 \setminus L$.
\end{itemize}
In either case, each component of this split sphere is either a plane or a cylinder. Moreover, a plane in such a building has Fredholm index equal to $1$, a cylinder in $X\setminus L$ has Fredholm index equal to $0$, while a cylinder in $T^*L$ has Fredholm index equal to $2$. (For the component in $X=\CP^2 \setminus L$ satisfying the point constraint at $p$, we mean the Fredholm index for an unparametrized solution required to pass through this point.)
\end{proposition}

\begin{proof}
First we note that the Fredholm index for a non-split such sphere $u \colon S^2 \to X$ in the specified homology class $A \in H_2(X)$ is equal to
\[ -2+2c_1([u])=2\]
in the case of $X=S^2 \times S^2$, while it is equal to
\[ -2+2c_1([u])=4\]
in the case of $X=\CP^2$. Observe that the Fredholm index hence is equal to $2$ in the latter case as well, given that we consider the problem satisfying the point constraint at $p \in \CP^2 \setminus L$. In the following we consider the constrained problem in the latter case. (From the point of view of the Fredholm index, this can be achieved by considering a blow-up of the manifold at the point $p$. Observe that the value of the above Chern class is decreased by one under this operation.)

Consider a building consisting of the components $u_1,\hdots,u_N$ in the different levels, and which lives in the homology class $A \in H_2(X)$. It is immediate from the definition of the relative Conley--Zehnder index that
\[ \sum_{i=1}^N c_{1,\OP{rel}}^\Phi(u_i)=c_1(A), \]
where the right hand side denotes the ordinary first Chern class.

By the definition of a pseudoholomorphic building, every asymptotic of $u_i$ corresponding to a positive (resp. negative) puncture is also the asymptotic of some $u_j$, $j \neq i$, as a negative (resp. positive) puncture. Moreover, the components glue together topologically to form a sphere. Using these facts, we compute that
\begin{equation}
\sum_{i=1}^N \ind(u_i)=-2-(N-1)2+3K+2c_1(A)=-2+2c_1(A)+K,\label{eq:index1}
\end{equation}
where $K$ is the total number of asymptotic Reeb orbits of the components appearing in the building. For the last equality we have used the identity $N-1=K$, which follows from the assumption that the components glue together to form a sphere, whose Euler characteristic thus can be computed as
\[ \chi(S^2)=2=N-K+1.\]

By $v_i$, $i=1,\hdots,N_1$ we denote the components in the top level $X \setminus L$, while by $w_i$, $i=1,\hdots,N_2$ we denote the components in the remaining levels. Clearly $N_1+N_2=N$ holds. The formula for the Conley--Zehnder index given in Lemma \ref{lm:index-formula}, together with the index Formula \eqref{eq:indexinside}, now gives the equality
\[\sum_{i=1}^{N_2} \ind(w_i)-K = -\sum_{i=1}^{N_2} \chi( w_i),\]
where $\chi( w_i)$ denotes the Euler characteristic of the punctured sphere being the domain of $w_i$. Together with Formula \eqref{eq:index1} we then compute the identity
\[ \sum_{i=1}^{N_1} \ind(v_i) -\sum_{i=1}^{N_2} \chi( w_i) =-2+2c_1(A)=2\]
relating the Fredholm indices of the components in the top level and the Euler characteristic of the components in the remaining levels.

Since there are no contractible periodic Reeb orbits in $(S^*L,\alpha_0)$, we necessarily have $\chi( w_i) \le 0$. The statement can finally be seen to follow by combining the non-negativity of the indices provided by Lemma \ref{complementpositive}, together with a topological consideration using the fact that all components join to form a \emph{sphere}.
\end{proof}

\subsection{Proof of Theorem \ref{thm:complement}}
\label{sec:proofcomplement}

We are now ready to prove Theorem \ref{thm:complement} in the two cases $(X,\omega)=(S^2 \times S^2,\omega_1 \oplus \omega_1)$ and $(\CP^2,\omega_{\OP{FS}})$. First we state the following technique, which allows us to pass from smooth isotopies of symplectic hypersurfaces to Hamiltonian isotopies.
\begin{prop}[Proposition 0.3 in \cite{OnHolomorphicity}]
\label{prop:hamiso}
A smooth isotopy $\Sigma_t \subset (X^4,\omega)$, $t \in [0,1]$, of a symplectic surface can be generated by a Hamiltonian isotopy, i.e.~$\Sigma_t=\phi^t_{H_t}(\Sigma_0)$ for some $H_t \colon X \to \R$. Given a closed subset $V \subset X$ which possesses a neighborhood $U \subset X$ in which $\Sigma_t \cap U=\Sigma_0 \cap U$ is fixed for all $t \in [0,1]$, we may moreover assume that $H_t|_V \equiv 0$ holds.
\end{prop}
By a nodal symplectic surface we mean a symplectic immersion of a closed surface having a finite number of transverse double points -- so called nodes, for which the local intersection number defined at each node moreover is required to be positive. The above proposition has the following corollary concerning families of nodal symplectic surfaces.
\begin{cor}
\label{cor:hamiso}
A smooth (ambient) isotopy $\Sigma_t \subset (X^4,\omega)$, $t \in [0,1]$, of nodal symplectic surfaces can be generated by a Hamiltonian isotopy, i.e.~$\Sigma_t=\phi^t_{H_t}(\Sigma_0)$, after a deformation of the family $\Sigma_t$ supported inside an arbitrarily small neighborhood of the nodes.
\end{cor}
\begin{proof}
After an appropriate deformation it suffices to consider the case when the family $\Sigma_t$ is constant near the nodes, after which the statement becomes a direct consequence of Proposition \ref{prop:hamiso}. To find the required deformation, we follow these steps:

{\em Step 1:} Any path $\gamma \colon I \to X$ can be generated by a Hamiltonian isotopy, i.e.~$\gamma(t)=\phi^t_{H_t}(\gamma(0))$. This means that we can assume the nodes to be fixed during the isotopy;

{\em Step 2:} Given smooth a family $t \mapsto P_t \subset T_{\gamma(t)}X$ of linear symplectic 2-planes, the above Hamiltonian isotopy may moreover be assumed to satisfy $P_t=D\phi^t_{H_t} (P_0)$. This means that we moreover can assume one of the tangent planes of the surface at each node to be fixed during the isotopy; and

{\em Step 3:} The linear symplectic 2-planes that are transverse to a given symplectic 2-plane in $\R^4$ consists of two contractible components (determined by the sign of the intersection number of the planes). This means that we can replace the node with any given standard model inside any arbitrarily small neighborhood.

To construct the deformation in Step 3, one can apply a suitably cut-off Hamiltonian isotopy of a punctured neighborhood of the node; this is a subset $((-\infty,A] \times S^3,d(e^t\alpha_{\OP{std}}))$ of the symplectization of the standard contact sphere for $A \ll 0$. The needed Hamiltonian isotopy can constructed by lifting a contact isotopy of the standard contact sphere that acts suitably on the link of the node; this is a transverse Hopf link inside the small contact sphere $(\{A \} \times S^3,\alpha_{\OP{std}})$ given by its intersection with the symplectic surface.

Note that the deformation in Step 3 may change the Hamiltonian isotopy class of the embedded nodal surface. Indeed, two nodes consisting of pairs of symplectic planes intersecting positively and transversely need not be symplectomorphic.
\end{proof}

\subsubsection{In the case $(X,\omega)=(S^2\times S^2,\omega_1 \oplus \omega_1)$}
\label{sec:hamisos2s2}
Recall the following classical result in \cite{Gromov} due to M. Gromov. Given any tame almost complex structure $J$ on $(S^2 \times S^2,\omega_1\oplus\omega_1)$, there exists a unique embedded $J$-holomorphic sphere in the homology class $A_i$, $i=1,2$, satisfying the requirement of passing through any fixed point. Moreover, the moduli space $\mathcal M_J(A_i)$ of $J$-holomorphic spheres in this homology class is two-dimensional and, by positivity of intersections, provides a foliation of $S^2 \times S^2$.

Perform a splitting construction along an embedding of the cosphere bundle of $L$ as described in Section \ref{sec:splitting}. The compactness result in Theorem \ref{thm:compactness} can be used to extract split spheres in the homology class $A_i$ passing through any given point in $S^2 \times S^2 \setminus L$.

For a generic almost complex structure, somewhere injective curves are regular (see \cite{McDuff-Salamon:curves}). Given that $J_\infty$ on $S^2 \times S^2 \setminus L$ is regular, Proposition \ref{prop:positive-index} above shows that that the components of a \emph{broken} split sphere in either of the classes $A_i$, $i=1,2$, only can fill a $3$-dimensional stratified subspace of $S^2 \times S^2 \setminus L$. Indeed, a component in such a split sphere being a $J_\infty$-holomorphic plane in $S^2 \times S^2 \setminus L$ is of index $1$ and is hence not multiply covered by Lemma \ref{complementpositive}. Since this plane is regular, it moreover lives inside a $1$-dimensional moduli space. A component in such a split curve which is a $J_\infty$-holomorphic cylinder in $S^2 \times S^2 \setminus L$ has index $0$. While we have not yet shown that it cannot be multiply covered, the underlying simple $J_\infty$-holomorphic cylinder is regular and hence comes in a zero-dimensional moduli space. In conclusion, since the totality of the broken split spheres form a subset of real codimension 1 inside $S^2 \times S^2 \setminus L$, we can find limits $\ell_i \subset S^2 \times S^2 \setminus L$ being \emph{non-broken} split $J_\infty$-holomorphic spheres in each of the homology classes $A_i$, $i=1,2$, satisfying appropriate generic point constraints.

The nodal sphere $\ell_1 \cup \ell_2$ is pseudoholomorphic for some tame almost complex structure on $(S^2 \times S^2,\omega_1 \oplus \omega_1)$. We can hence use Gromov's result, together with the contractibility of the space of tame almost complex structures, in order to produce a smooth family of nodal symplectic spheres connecting $\ell_1 \cup \ell_2$ and $(S^2 \times \{\infty\}) \cup (\{\infty\} \times S^2)=D_\infty$. Corollary \ref{cor:hamiso} then produces the sought Hamiltonian isotopy placing the torus $L$ inside $(S^2 \times S^2 \setminus D_\infty)$.
\qed

\subsubsection{In the case $(X,\omega)=(\CP^2,\omega_{\OP{FS}})$}
Recall the following classical result in \cite{Gromov} due to M. Gromov. Given any tame almost complex structure $J$ on $(\CP^2,\omega_{\OP{FS}})$, there exists a unique embedded $J$-holomorphic sphere of degree one passing through any fixed pair of distinct points. Moreover, the moduli space $\mathcal M_J(p)$ of the degree one spheres that pass through the fixed point $p \in \CP^2$ is two-dimensional and, by positivity of intersections, provides a foliation of $\CP^2\setminus \{p\}$.

We pick a Weinstein neighborhood of $L$ disjoint from $p$, and perform a splitting construction along an embedding of the cosphere bundle of $L$ as described in Section \ref{sec:splitting}. The compactness result in Theorem \ref{thm:compactness} can be used to extract a split sphere of degree one passing through the points $p,q \in \CP^2 \setminus L$, where $q \in \CP^2 \setminus L$ is arbitrary.

Analyzing the split spheres of degree one as in the case of $S^2\times S^2$, we obtain the following. Given that the choice of point $p \in \CP^2 \setminus L$ as well as the almost complex structure $J_\infty$ were both generically chosen, Proposition \ref{prop:positive-index} again implies that there are points in $\CP^2 \setminus L$ through which no broken split sphere of degree one can pass. It follows that there exist limit $J_\infty$-holomorphic spheres of degree one inside $\CP^2\setminus L$ which are non-broken.

The Hamiltonian isotopy taking $L$ into the complement of the line at infinity $D_\infty \subset \CP^2$ is finally constructed as in the case $(X,\omega)=(S^2 \times S^2,\omega_1 \oplus \omega_1)$ above, i.e.~by again alluding to Gromov's result and subsequently applying Proposition \ref{prop:hamiso}.
\qed

\section{Analysis of pseudoholomorphic curves inside $T^*\TT^2$}
\label{sec:cylinders}

Following \cite{Ivrii-thesis}, we consider the tame almost complex structure $J_{\OP{std}}$ on $T^*\TT^2=\TT^2 \times \R^2$ determined by
\[J_{\OP{std}}\partial_{\theta_i}=-\rho(\|(p_1,p_2)\|)\partial_{p_i}, \:\:i=1,2,\]
where $\rho \colon \R_{\ge 0} \to \R_{\ge 0}$ is a smooth function that is required to satisfy
\begin{itemize}
\item $\rho(t)>0$ and $\rho'(t) \ge 0$ for all $t \ge 0$,
\item $\rho(t) = 1$ for all $0 \le t \le 1$, and
\item $\rho(t)=t$ for all $t \ge 2$.
\end{itemize}
We begin with the following observations. In the neighborhood $T^*_1\TT^2 \subset T^*\TT^2$ of the zero-section, $J_{\OP{std}}$ is given by the standard product complex structure inherited from the universal cover $\C^2 = T^*\R^2 \to T^*\TT^2$. In the subset $T^*\TT^2 \setminus T^*_2 \TT^2$ containing the convex cylindrical end, $J_{\OP{std}}$ coincides with an almost complex structure $J_{\OP{cyl}}$ which is cylindrical with respect to the hypersurface $(S^*\TT^2,\alpha_0) \subset (T^*\TT^2,d\lambda)$ of contact type. This cylindrical almost complex structure is determined by
\[J_{\OP{cyl}}\partial_{\theta_i}=-\|(p_1,p_2)\|\partial_{p_i}, \:\: i=1,2,\]
on $(T^*\TT^2 \setminus 0_{\TT^2},d\lambda) \cong (\R \times S^*\TT^2,d(e^t\alpha_0))$ (see Part (3) of Example \ref{ex:cyl}).
\begin{lemma}
The almost complex structure $J_{\OP{std}}$ is tamed by the symplectic form on $(T^*\TT^2,d\lambda)$, while the almost complex structure $J_{\OP{cyl}}$ is compatible with the symplectic form on $(T^*\TT^2 \setminus 0_{\TT^2},d\lambda)$ and, moreover, cylindrical with respect to $\alpha_0$.
\end{lemma}
\begin{proof}
Note that
\[J_{\OP{cyl}}\partial_{p_i}=\rho(\|(p_1,p_2)\|)^{-1}\partial_{\theta_i}, \:\: i=1,2,\]
is satisfied. The tameness of $J_{\OP{std}}$ is now an easy matter to check.

Recall the identification $S^*\TT^2=(S^1)^2 \times S^1 \ni (\theta_1,\theta_2,\theta)$ for which the contact form is given by $\alpha_0=\cos(\theta)d\theta_1+\sin(\theta)d\theta_2$, with the induced Reeb vector field $\cos(\theta)\partial_{\theta_1}+\sin(\theta)\partial_{\theta_2}$. Using the identification in Part (3) of Example \ref{ex:cyl}, it follows that
$$ J_{\OP{cyl}} \partial_t = \cos(\theta)\partial_{\theta_1}+\sin(\theta)\partial_{\theta_2},$$
where the identity
$$\partial_t=\|(p_1(t,\theta),p_2(t,\theta))\|(\cos(\theta)\partial_{p_1}+\sin(\theta)\partial_{p_2})$$
has been used. Furthermore, since
$$ \partial_\theta = \|(p_1(t,\theta),p_2(t,\theta))\|(-\sin(\theta)\partial_{p_1}+\cos(\theta)\partial_{p_2}) $$
we see that $J_{\OP{cyl}}$ preserves $\ker\alpha_0 \cap TS^*\TT^2$. Indeed, $\langle \partial_\theta, J_{\OP{cyl}} \partial_\theta \rangle=\ker \alpha_0 \cap TS^*\TT^2$ is a basis of the contact distribution which, moreover, is invariant under translation of the $t$-coordinate.

Since the contact distribution is two-dimensional, the tameness actually implies compatibility with $d\alpha_0$, and hence with $d\lambda$ as well. In other words, $J_{\OP{cyl}}$ is cylindrical as sought.
\end{proof}

It is possible to explicitly describe the moduli space of all $J_{\std}$-holomorphic cylinders in $T^*\TT^2$. First, observe that for any non-zero element $\mathbf{m} = (m_1,m_2) \in \Z^2 \setminus \{0\}$, and points $\mathbf{p} \in \R^2$ and $\boldsymbol{\theta} \in \TT^2$, we can define the immersed cylinder
\begin{gather*}
u^{\mathbf{m}}_{\boldsymbol{\theta},\mathbf{p}} \colon \R \times S^1 \to \TT^2=S^1 \times S^1 \times \R^2,\\
(t,\theta) \mapsto (\boldsymbol{\theta}+\mathbf{m}\theta,\mathbf{p}+\mathbf{m}t)
\end{gather*}
It can be checked that, for a suitable conformal structure on the domain, the above cylinder becomes a finite energy two-punctured pseudoholomorphic sphere. It is moreover the case that this cylinder has punctures asymptotic to the two Reeb orbits in $\Gamma_{\pm\mathbf{m}}$ corresponding to the two geodesics $\theta \mapsto (\boldsymbol{\theta}\pm \mathbf{m}\theta)$ on $\TT^2$ living in the homology classes $\pm\mathbf{m} \in \Z^2 = H_1(\TT^2)$. (Or, from an alternative point of view, a single geodesic equipped with its two different orientations.)

Given that $\|\mathbf{p}\| \gg 0$ is sufficiently large and that $\mathbf{p}$ and $\mathbf{m}$ are not collinear, it follows that the above cylinder is contained inside $(T^*\TT^2 \setminus 0_{\TT^2},d\lambda) \cong (\R \times S^*\TT^2,d(e^t\alpha_0))$. This cylinder is moreover holomorphic for the cylindrical almost complex structure $J_{\OP{cyl}}$.

It can be explicitly seen that the above cylinders form a foliation of $T^*\TT^2$ for each non-zero $\mathbf{m} \in \Z^2$. We proceed to show that:
\begin{lemma}
\label{cotangentcylinders}
All $J_{\OP{std}}$-holomorphic cylinders inside $T^*\TT^2$ of finite energy are of the above form $u^\mathbf{m}_{\boldsymbol{\theta},\mathbf{p}}$ for some $\mathbf{m} \in \Z^2 \setminus \{0\}$, $\boldsymbol{\theta} \in \TT^2$, and $\mathbf{p} \in \R^2$.
\end{lemma}
\begin{proof}
By topological reasons, any such pseudoholomorphic cylinder $C \subset T^*\TT^2$ must have punctures asymptotic to periodic Reeb orbits in the two families $\Gamma_{\pm \mathbf{m}}$ for some non-zero homology class $\mathbf{m} \in \Z^2 = H_1(\TT^2)$.

Assuming that $C$ is not a cylinder of the above form, by genericity we can find a cylinder $u^{\mathbf{m}}_{\boldsymbol{\theta}_0,\mathbf{p}_0}$ which intersects $C$ in a discrete and non-empty set, but whose two asymptotic orbits both are disjoint from those of $C$. 

The cylinder $C$ is arbitrarily $C^0$-close to trivial cylinders over its asymptotic Reeb orbits outside any sufficiently big compact subset, as follows by its asymptotic properties. It can be explicitly seen that any cylinder $u^{\mathbf{m}}_{\boldsymbol{\theta}_0,\mathbf{p}}$ for an arbitrary $\mathbf{p}\in\R^2$ is disjoint from the aforementioned trivial (semi-infinite) half-cylinders. In other words, intersection points cannot escape to infinity and, using positivity of intersection, we conclude that $C \cap u^{\mathbf{m}}_{\boldsymbol{\theta}_0,\mathbf{p}_0+s\mathbf{k}} \neq \emptyset$ must hold for all $s \in \R$ and $\mathbf{k} \in \R^2$.

Now take any $\mathbf{k}=(k_1,k_2) \in \R^2$ that satisfies
\[\det\begin{pmatrix} m_1 & k_1 \\ m_2 & k_2 \end{pmatrix} \neq 0,\]
and consider the family $u^{\mathbf{m}}_{\boldsymbol{\theta}_0,\mathbf{p}_0+s\mathbf{k}}$, $s \in \R$. Since the cylinder $u^{\mathbf{m}}_{\boldsymbol{\theta}_0,\mathbf{p}_0+s\mathbf{k}}$ can be assumed to be contained outside of any given compact subset whenever $|s| \gg 0$ is chosen sufficiently large, we conclude that these cylinders are disjoint from $C$ in this case. This gives the sought contradiction.
\end{proof}

\begin{cor}
\label{cor:cylint}
Assume we are given two $J_{\OP{std}}$-holomorphic cylinders inside $T^*\TT^2$ which are asymptotic to Reeb orbits corresponding to geodesics in the non-zero homology classes $\pm(m_1,m_2)$ and $\pm(n_1,n_2)$ in $\Z^2 \simeq H_1(\TT^2)$, respectively. Unless the images of these cylinders coincide, they intersect in an algebraic number
\[ \left|\det \begin{pmatrix} m_1 & n_1 \\
m_2 & n_2
\end{pmatrix}\right|=|m_1n_2-m_2n_1| \]
of points. In particular, the two cylinders intersect in a non-empty and discrete set whenever $(m_1,m_2)$ and $(n_1,n_2)$ are not collinear.
\end{cor}

\begin{lemma}
\label{symplectizationcylinders}
Any $J_{\OP{cyl}}$-holomorphic cylinder inside $\R \times S^*\TT^2$ of finite energy is of the above form in the case when it has two positive punctures, while it is a trivial cylinder over a periodic Reeb orbit in the case when it has precisely one positive puncture.
\end{lemma}

\begin{proof}
The case of cylinders having two positive punctures is treated as in the case of $T^*\TT^2$.

By topological reasons, a cylinder having one positive and one negative puncture must have both punctures asymptotic to Reeb orbits contained in the same family. Since the $d\alpha_0$-energy of such a cylinder vanishes, the projection to $S^*\TT^2$ of the cylinder must be contained inside a periodic Reeb orbit.
\end{proof}

\section{Split spheres arising from a foliation of $S^2 \times S^2$}
\label{sec:fibration}

Let $L \subset S^2 \times S^2$ be a Lagrangian torus which is disjoint from the two lines $D_\infty:= (S^2 \times \{p\}) \cup (\{q\}\times S^2)$. Recall that this imposes no restriction by Theorem \ref{thm:complement}. In this section we deduce useful properties of split pseudoholomorphic spheres in the split symplectic manifold
\[(S^2 \times S^2 \setminus L,\omega_1\oplus\omega_1) \:\: \sqcup \:\: (T^*L,d\lambda). \]
More precisely, we are interested in buildings that live in the homology classes
\[A_1:=[S^2 \times \{\pt\}], \: A_2:=[\{\pt\} \times S^2] \in H_2(S^2 \times S^2),\]
of minimal symplectic area.

In order to infer that such split curves exist, we rely on Gromov's result concerning the existence of pseudoholomorphic foliations of $(S^2 \times S^2,\omega_1\oplus \omega_1)$ from \cite{Gromov} together with the splitting construction described in Section \ref{sec:splitting}. The latter construction ``stretches the neck'' along a hypersurface $S^*\TT^2 \hookrightarrow S^2 \times S^2 \setminus D_\infty$ of contact type representing an embedding of the unit normal bundle of $L$, i.e.~we consider a particular limit of a family $J_\tau$, $\tau \ge 0$, of tame almost complex structures. Moreover, this sequence of almost complex structures converges to an almost complex structure $J_\infty$ on $S^2\times S^2\setminus L$ as described in Section \ref{sec:splitting}.

From now on the almost complex structure $J_\infty$ on $S^2 \times S^2 \setminus L$ will also be required to coincide with the standard integrable complex structure $i$ in some small neighborhood of $D_\infty$. It thus follows that the latter divisor $D_\infty$ is holomorphic. We also choose $J_\infty$ to be regular for simply covered punctured pseudoholomorphic spheres, i.e.~so that Proposition \ref{prop:positive-index} can be applied. Observe that these two conditions indeed can be achieved simultaneously using standard transversality techniques; see e.g.~\cite[Section 3]{McDuff-Salamon:curves}.

\subsection{The generic and exceptional split sphere}
The pseudoholomorphic buildings shown in Figure \ref{fig:breaking} will play an important role when considering limits of spheres of minimal area in the splitting construction. We begin by introducing names for them.
\begin{definition} A split sphere of \emph{type I} in homology class $A_i \in H_2(S^2 \times S^2)$ is a split sphere for which:
\begin{itemize}
\item The top level consists of two planes of index 1 in $S^2 \times S^2 \setminus L$; and
\item The bottom (or middle) level consists of a cylinder inside $T^*L$ (or $\R \times S^*L$).
\end{itemize}
\end{definition}
Since a punctured spheres compactifies to a null-homology of its asymptotic orbits, there exists an element $\eta \in H_1(L)$ such that the asymptotic orbits of the two planes are contained in $\Gamma_\eta$ and $\Gamma_{-\eta}$, respectively. Positivity of intersection together with the holomorphicity of $D_\infty $ implies that precisely one of the planes passes through $D_\infty$, and moreover does so in a single transverse intersection.
\begin{definition} A split sphere of \emph{type II} in class $A_i \in H_2(S^2 \times S^2)$ is a split sphere for which:
\begin{itemize}
\item The top level consists of two planes of index 1 in $S^2 \times S^2 \setminus L$ which both are disjoint from $D_\infty$, together with a cylinder of index 0; and
\item The bottom (and/or middle) level consists of two cylinders inside $T^*L$ (and/or $\R \times S^*L$).
\end{itemize}
\end{definition}
Since a punctured spheres compactifies to a null-homology of its asymptotic orbits, there exist elements $\zeta,\eta \in H_1(L)$ such that the asymptotic orbits of the two planes are contained in $\Gamma_{\eta}$ and $\Gamma_{\zeta}$, respectively, while the asymptotic orbits of the cylinder in $S^2 \times S^2 \setminus L$ are contained in $\Gamma_{-\zeta}$ and $\Gamma_{-\eta}$. As above, it follows that the cylinder in the top level must have a single transverse intersection with $D_\infty$.

Observe that the split sphere of type II is ``exceptional'' due to the presence of a rigid component. As follows from the results in Section \ref{sec:cylinders}, we expect to find a finite number of such buildings (after forgetting about the components in the middle and bottom levels).

\begin{figure}[htp]
\centering
\vspace{3mm}
\hspace{20mm}
\labellist
\pinlabel $S^2\times S^2\setminus L$ at -38 58
\pinlabel $T^*L$ at -18 19
\pinlabel $\infty$ at 70 83
\pinlabel $1$ at 70 53
\pinlabel $1$ at 30 53
\pinlabel $2$ at 50 12
\pinlabel $\color{red}\eta$ at 30 23
\pinlabel $\color{red}-\eta$ at 67 23
\pinlabel $\infty$ at 205 83
\pinlabel $1$ at 266 53
\pinlabel $1$ at 145 53
\pinlabel $2$ at 165 12
\pinlabel $2$ at 245 12
\pinlabel $0$ at 190 60
\pinlabel $\color{red}\zeta$ at 145 22
\pinlabel $\color{red}-\zeta$ at 183 22
\pinlabel $\color{red}-\eta$ at 228 22
\pinlabel $\color{red}\eta$ at 266 22
\endlabellist
\includegraphics{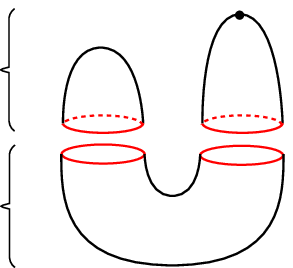}
\hspace{15mm}
\includegraphics{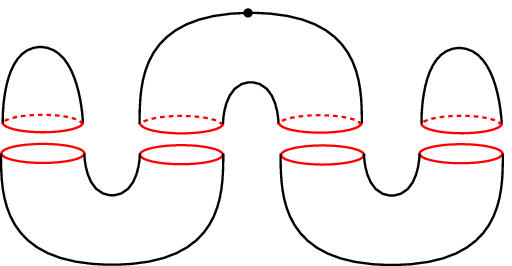}
\caption{On the left: a split sphere of type I (the generic configuration). On the right: a split sphere of type II (the exceptional configuration). The numbers on each component denotes its respective Fredholm index, while the homology classes $\pm\zeta,\pm\eta \in H_1(L)$ designate the families containing the respective asymptotic orbits.}
\label{fig:breaking}
\end{figure}

The goal in this section is establishing the following result concerning a split sphere in the above homology classes.
\begin{prop}
\label{prop:splitsphere}
For any generic almost complex structure $J_\infty$ as above there exists fixed homology classes $\eta_i \in H_1(L)$, $i=1,2$, for which the following is satisfied. Any split sphere in the homology class $A_i \in H_2(S^2 \times S^2)$ is either of type I or type II, and satisfies the properties that:
\begin{enumerate}
\item Every component is simply covered and has punctures asymptotic to simply covered orbits;
\item Every asymptotic orbit of a plane and cylinder inside $S^2 \times S^2 \setminus L$ is contained inside $\Gamma_{\eta_i}$ and $\Gamma_{-\eta_i}$, respectively.
\end{enumerate}
\end{prop}
Since symplectic area is positive proportional to the Maslov index for monotone Lagrangian tori, it follows that a cylinder of index 0 must have vanishing symplectic area, and hence such cylinders do not exist. In other words:
\begin{corollary}\label{cor:monotone}
If $L\subset (S^2\times S^2,\omega_1 \oplus \omega_1)$ is a \emph{monotone} Lagrangian torus then all split spheres of minimal symplectic area are of type I.
\end{corollary}

\subsubsection{Proof of Proposition \ref{prop:splitsphere}}

The first step consists of showing the above statement for pseudoholomorphic buildings that arise as the limit of pseudoholomorphic spheres when stretching the neck. To that end, we need to make use of the following lemma.
\begin{lemma}
\label{lem:nodiscrete}
Consider a split pseudoholomorphic sphere in either of the homology classes $A_i$, $i=1,2$, which arises as the limit of $J_\tau$-holomorphic spheres under the splitting construction. It follows that two components of the building cannot intersect in a discrete and nonempty set.
\end{lemma}
\begin{proof}
Using positivity of intersection \cite{McDuff:Local} together with the nature of convergence, the existence of such a discrete intersection point would imply that some $J_\tau$-holomorphic sphere in the class $A_i$ for $\tau \gg 0$ sufficiently large has a self-intersection. However, again alluding to positivity of intersection, such a self-intersection would contradict $A_i \bullet A_i=0$.
\end{proof}
\begin{proposition} \label{prop:simple-cylinders}
For any generic almost complex structure $J_\infty$ as above, any split sphere in the homology class $A_i \in H_2(S^2 \times S^2)$, $i=1,2$, which, moreover, arises as the limit of $J_\tau$-holomorphic spheres when stretching the neck is either of type I or type II, and satisfies the properties that:
\begin{enumerate}
\item Every component is simply covered and has punctures asymptotic to simply covered orbits; \label{i}
\item Every asymptotic orbit of a plane and cylinder inside $S^2 \times S^2 \setminus L$ is contained inside $\Gamma_{\eta_i}$ and $\Gamma_{-\eta_i}$, respectively, for some homology class $\eta_i \in H_1(L)$ (which a priori depends on the specific building). \label{ii}
\end{enumerate}
\end{proposition}
\begin{proof}
Positivity of intersection together with $A_i \bullet [D_\infty]=1$ implies that each split curve in the class $A_i$ consists of precisely one component intersecting $D_\infty$. This intersection moreover consists of a single transverse double point. From this it is immediate that the component passing through $D_\infty$ is simply covered.

In the case when the split curve contains no cylinders in $S^2 \times S^2 \setminus L$, Proposition \ref{prop:positive-index} clearly shows that the obtained building is of type I. What remains in this case is to check Properties \eqref{i} and \eqref{ii}, of which the second is immediate. To see that the planes and asymptotic orbits are simply covered we argue as follows. The asymptotic orbit of the plane disjoint from $D_\infty$ contained in $\Gamma_\eta$ is simply covered by Lemma \ref{lm:simple-orbits}. An elementary topological consideration now implies that the asymptotic orbit of the plane intersecting $D_\infty$ is contained in $\Gamma_{-\eta}$, and is hence in particular simply covered as well.

We are left with the case when the building consists of at least one component $C \subset S^2 \times S^2 \setminus L$ which is a cylinder. We let $B_1, B_2 \subset T^*L$ (or $\R \times S^*L$) denote the two cylinders that are connected to $C$. Here we have used Proposition \ref{prop:positive-index} in order to infer that the latter curves indeed are cylinders. We also write $\Gamma_{-\zeta_i}$, $i=1,2$, for the families of periodic Reeb orbits containing the asymptotics at which $C$ and $B_i$ are connected, where $\zeta_i \in H_1(L)$, $i=1,2$.

Below we will show that $C$ necessarily intersects $D_\infty$ and, in particular, that there is precisely one cylinder in the top level. By the following argument we may then conclude that the building is of type II. Since all planes in the building necessarily are disjoint from $D_\infty$, Property \eqref{i} follows as in the case of a building of type I. We continue by establishing Property \eqref{ii}. Recall that both planes are disjoint from $D_\infty$, of index 1, and asymptotic to orbits in the families $\Gamma_{\zeta_i}$, $i=1,2$. Using index Formula \eqref{eq:index} in terms of the Maslov class as in Fromula \eqref{eq:maslov}, we infer that $\zeta_2 = \zeta_1+\zeta_0$ for a homology class $\zeta_0 \in H_1(L)$ that bounds a disk in $H_2(\R^4,L)$ of Maslov index 0. Given that $\zeta_1 \neq \zeta_2$, Corollary \ref{cor:cylint} now implies that $B_1$ and $B_2$ (or a trivial cylinder contained inside a middle level) intersect in a discrete and nonempty set. This is however in contradiction with Lemma \ref{lem:nodiscrete}.

In view of the above, what remains is to prove that $C$ must intersect $D_\infty$. We argue by contradiction, and take a cylinder $C \subset S^2 \times S^2 \setminus L$ that is disjoint from $D_\infty$ and connected to a $J_\infty$-holomorphic plane $D \subset S^2 \times S^2 \setminus L$ via a single cylinder $B_1 \subset T^*L$, where the plane $D$ also is disjoint from $D_\infty$. That this is possible follows by Proposition \ref{prop:positive-index}. Again, let $B_2 \subset T^*L$ denote the second cylinder connected to $C$, and use $\Gamma_{-\zeta_i}$ to denote the space of periodic orbits containing the asymptotic at which $C$ and $B_i$ are connected, $i=1,2$. See Figure \ref{fig:badbreaking} for a schematic picture.

In the case when $\zeta_1,\zeta_2 \in H_1(L)$ are not collinear, an argument as above again leads to a contradiction with Lemma \ref{lem:nodiscrete}. Indeed, in this case $B_1$ and $B_2$ must intersect in a discrete and nonempty set by Corollary \ref{cor:cylint}. Since the cylinder $C$ is of index 0, we use the index Formula \eqref{eq:index} to show that, moreover, $\zeta_1=-\zeta_2$. However, using the exactness of $\omega_1 \oplus \omega_1$ when restricted to $S^2 \times S^2 \setminus D_\infty$, we conclude that $\int_C \omega_1 \oplus \omega_1 = 0$. This clearly contradicts the tameness of $J_\infty$. 
\end{proof}

\begin{figure}[htp]
\centering
\vspace{3mm}
\hspace{20mm}
\labellist
\pinlabel $S^2\times S^2\setminus L$ at -38 58
\pinlabel $T^*L$ at -18 19
\pinlabel $C$ at 90 63
\pinlabel $D$ at 29 54
\pinlabel $B_2$ at 130 12
\pinlabel $B_1$ at 50 12
\pinlabel $\infty$ at 257 83
\pinlabel $\color{red}\zeta_1$ at 30 22
\pinlabel $\color{red}-\zeta_1$ at 68 22
\pinlabel $\color{red}-\zeta_2$ at 112 22
\pinlabel $\color{red}\zeta_2$ at 149 22
\endlabellist
\includegraphics{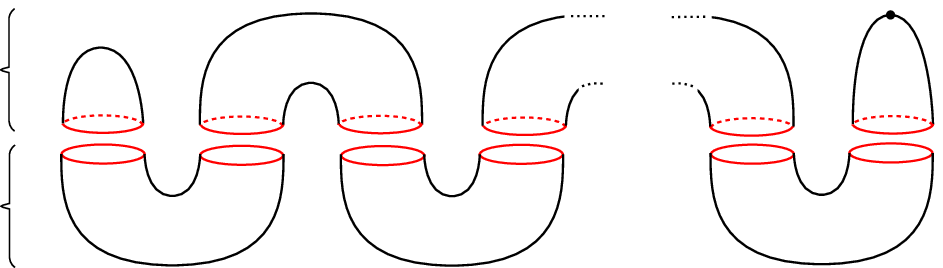}
\caption{A pseudoholomorphic plane $D \subset \R^4 \setminus L$ connected to a pseudoholomorphic cylinder $C \subset \R^4 \setminus L$ via a pseudoholomorphic cylinder $B_1 \subset T^*L$. Given that $\mathrm{index}(D)=1$ and $\mathrm{index}(C)=0$, it follows that $B_1$ intersects the other cylinder $B_2 \subset T^*L$ that is connected to $C$ in a discrete and nonempty set. The homology classes $\pm\zeta_i \in H_1(L)$ designate the families containing the respective asymptotic orbits.}
\label{fig:badbreaking}
\end{figure}

The following result shows that there indeed exist split spheres to which the above proposition can be applied.
\begin{prop}
\label{prop:existssplit}
There exists a split sphere in the homology class $A_i \in H_2(S^2 \times S^2)$ for each $i=1,2$ passing through any given point on $L$. Moreover, under the additional assumption that the almost complex structure $J_\infty$ is generic, there exists an open and dense subset of $L$, such that the split spheres passing through this subset all are of type I.
\end{prop}
\begin{proof}
Recall the existence of a unique $J_\tau$-holomorphic sphere in class $A_i$ passing through any fixed point $p \in S^2 \times S^2$, as shown by Gromov in \cite{Gromov}. Applying the compactness result Theorem \ref{thm:compactness} the existence of split spheres follows.

The fact that a split sphere passing through a generic point on $L$ is of type I can be seen as follows. The properties of split spheres proven in Proposition \ref{prop:simple-cylinders}, together with the classification of the cylinders given by Lemma \ref{cotangentcylinders}, shows that a building of type II only can pass through a finite number of closed geodesics on $L$. Here we rely on the fact that the simply covered cylinders of index zero under consideration form a compact zero-dimensional manifold. This is the case by the assumption of $J_\infty$ being regular.
\end{proof}

In order to prove Proposition \ref{prop:splitsphere} we need the following positivity of intersection result for split pseudoholomorphic curves.
\begin{lemma}
\label{lem:nodiscrete2}
Consider two split pseudoholomorphic curves, both contained in the homology class $A_i \in H_2(S^2 \times S^2)$ for either $i=1$ or $2$. After adding a number of middle levels consisting of only trivial cylinders to one of the two buildings, we may assume that they both have the same number of levels. Given that two asymptotic orbits coming from two different components of the two different buildings never coincide, it follows that two such components cannot intersect in a discrete and nonempty set.
\end{lemma}
\begin{proof}
Assume that both buildings consist of a $k \ge 0$ number middle levels after the above procedure. There is a smooth open embedding
\[S^2 \times S^2 \setminus L \: \sqcup \: \underbrace{\R \times S^*L \: \sqcup \hdots \sqcup \:\R \times S^*L}_k \: \sqcup \: T^*L  \hookrightarrow S^2 \times S^2\]
with image $S^2 \times S^2 \setminus \Sigma$, where $\Sigma \subset S^2 \times S^2$ is an embedded hypersurface consisting of $k+1$ parallel (and disjoint) copies of the unit cotangent bundle $S^*L$.

Under a suitable compactification of the above embedding, the two split spheres each give rise to a cycle in $S^2 \times S^2$ in the homology class $A_i$. These cycles are smooth and pseudoholomorphic outside of $\Sigma$, and they intersect $\Sigma$ in a number of closed curves corresponding to the asymptotic Reeb orbits of the different components.

Using positivity of intersection \cite{McDuff:Local} together with the asymptotic convergence to Reeb orbits, a discrete intersection of two components would imply that the two cycles have a positive intersection number. A conclusion that clearly contradicts $A_i \bullet A_i=0$.
\end{proof}

We now commence to prove Proposition \ref{prop:splitsphere}. We start by establishing the properties prescribed by Proposition \ref{prop:simple-cylinders} for a \emph{general} split sphere in the homology class $A_i \in H_2(S^2 \times S^2)$, i.e.~a split sphere which does not necessarily arise as the limit of pseudoholomorphic spheres when stretching the neck. Examining the proof, we see that this assumption only is used in order to exclude an intersection of two cylinders inside $T^*L$ contained in the building; this is an argument that relies on Lemma \ref{lem:nodiscrete}. However, since Proposition \ref{prop:existssplit} provides a split sphere of type I passing through a generic point $p \in L$, such a self-intersection can also be excluded using Lemma \ref{lem:nodiscrete2}.

More precisely, assume that two cylinders in $T^*L$ or $\R \times S^*L$, both arising as the components of a split sphere as in the assumptions, intersect in a discrete and non-empty set. Using the intersection properties of cylinders provided by Corollary \ref{cor:cylint}, it follows that at least one of these cylinders also intersects the cylinder contained in a split sphere of type I that exists by Proposition \ref{prop:existssplit}. In order to apply Lemma \ref{lem:nodiscrete2}, we must choose the latter split sphere of type I with some care, so that none of its asymptotic orbits coincides with an asymptotic orbit of the former split sphere. In other words, the point $p \in L$ in Proposition \ref{prop:existssplit} must be chosen generically.

What is left is showing Property (ii) asserted by the theorem. Consider two limit buildings in the homology class $A_i$ for some fixed $i=1,2$. By the above, we know that the asymptotic orbits of the buildings are contained in the families $\Gamma_\alpha \cup \Gamma_{-\alpha}$ and $\Gamma_\beta \cup \Gamma_{-\beta}$ of simply covered Reeb orbits, respectively. We claim that $\alpha=\beta$ holds, as sought. Namely, if not, Corollary \ref{cor:cylint} together with Lemma \ref{lem:nodiscrete2} again gives a contradiction.
\qed

\subsection{The moduli space of small planes and its compactness}
\label{sec:planes}

For each fixed homology class $A_i \in H_2(S^2 \times S^2)$, $i=1,2$, an application of Proposition \ref{prop:splitsphere} shows the following. There exists a unique homology class $\eta_i \in H_1(L)$ such that all $J_\infty$-holomorphic planes in $S^2 \times S^2 \setminus L$ that are disjoint from $D_\infty$ and arise as a component of a split sphere in the homology class $A_i$ must have its asymptotic orbit contained in $\Gamma_{\eta_i}$. This leads to the following definitions.
\begin{definition}
\label{def:smallbig}
Given the previously mentioned designated families $\Gamma_{\pm \eta_i}$ of periodic Reeb orbits, we define:
\begin{itemize}
\item The moduli space $\mathcal{M}^0(\Gamma_{\eta_i})$, $i=1,2$, of \emph{small planes} consisting of those pseudoholomorphic planes in $S^2 \times S^2 \setminus L$ that are disjoint from $D_\infty$ and have its asymptotic orbit contained in the family $\Gamma_{\eta_i}$;
\item The moduli space $\mathcal{M}^i(\Gamma_{-\eta_i})$, $i=1,2$, of \emph{big planes} consisting of those pseudoholomorphic planes in $S^2 \times S^2 \setminus L$ that satisfy
\begin{gather*}
u \bullet (S^2 \times \{\infty\}) = \begin{cases}
0, & i = 1,\\
1, & i = 2,
\end{cases}\\
u \bullet (\{\infty\} \times S^2) = \begin{cases}
1, & i = 1,\\
0, & i = 2,
\end{cases}
\end{gather*}
and which have its asymptotic orbit contained in the family $\Gamma_{-\eta_i}$.
\end{itemize}
See Figure \ref{fig:smallbig}.
\end{definition}

\begin{figure}[htp]
\centering
\vspace{3mm}
\labellist
\pinlabel $S^2\times S^2\setminus L$ at -38 58
\pinlabel $T^*L$ at -18 19
\pinlabel $\infty$ at 124 83
\pinlabel $1$ at 124 53
\pinlabel $\mathcal{M}^0(\Gamma_{\eta_i})\ni$ at 40 53
\pinlabel $\in\mathcal{M}^i(\Gamma_{-\eta_i})$ at 173 53
\pinlabel $1$ at 84 53
\pinlabel $2$ at 104 12
\pinlabel $\color{red}\eta_i$ at 84 23
\pinlabel $\color{red}-\eta_i$ at 121 23
\endlabellist
\includegraphics{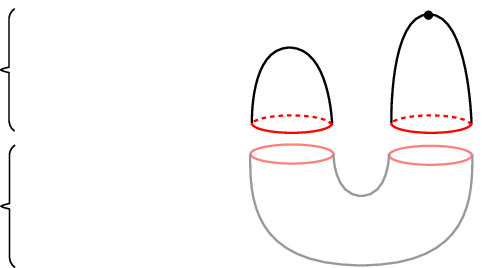}
\caption{For each $i=1,2$, a small and big plane from $\mathcal{M}^0(\Gamma_{\eta_i})$ and $\mathcal{M}^i(\Gamma_{\eta_i})$, respectively, can be completed to form a cycle in the class $A_i \in H_2(S^2 \times S^2)$.}
\label{fig:smallbig}
\end{figure}
Recall the following facts about the small and big planes which are consequences of Proposition \ref{prop:splitsphere}. The index of either a big or a small plane is equal to one. The asymptotic Reeb orbits in $\Gamma_{\eta_i}$ are simple, and hence all of the moduli spaces above consist of simply covered planes. By genericity we may thus assume that these moduli spaces are transversely cut out manifolds of dimension one. Finally, these moduli spaces are all non-empty by Proposition \ref{prop:existssplit} above.
\begin{remark}
The case $\eta_1=\eta_2$ may indeed occur for certain Lagrangian tori. In fact, given a monotone Lagrangian torus for which $\eta_1 \neq \eta_2$ is satisfied, one can use the techniques from Section \ref{sec:classification} (in particular, see the proof of Corollary \ref{cor:fibration}) in order to conclude that this torus is Hamiltonian isotopic to the Clifford torus.
\end{remark}

In this section we establish the following result concerning the moduli space of small planes.
\begin{proposition}\label{prop:small-planes}
The moduli spaces $\mathcal{M}^0(\Gamma_{\eta_i})$, $i=1,2$, of small planes satisfies the following properties:
\begin{enumerate}
\item The asymptotic evaluation map $\mathcal{M}^0(\Gamma_{\eta_i}) \to \Gamma_{\eta_i} \cong S^1$ taking a plane to its asymptotic orbit is a diffeomorphism; and
\item The evaluation map $\mathcal{M}^0(\Gamma_{\eta_i}) \times \C \to S^2 \times S^2 \setminus L$, obtained after fixing a family of parametrizations of $\C$, is a smooth embedding.
\end{enumerate}
In the case when the Lagrangian torus $L \subset S^2 \times S^2$ is monotone, the same properties are also satisfied for the moduli spaces $\mathcal{M}^i(\Gamma_{-\eta_i})$ of big planes. For a general non-monotone Lagrangian torus the same holds, with the exception that the moduli space of big planes possibly is non-compact, in which case its asymptotic evaluation map is a non-surjective open embedding.
\end{proposition}
\begin{proof}
We let $\mathcal{M}(\Gamma)$ denote either the moduli space of small or big planes or. In the case of the big planes, we will make the further assumption that the Lagrangian torus is monotone. The argument concerning the big planes in the general case follows by a similar argument.

We start by showing that this moduli spaces is compact. In the monotone case, this follows by the fact that there are no pseudoholomorphic cylinders of index 0. In the non-monotone case, and hence when considering a moduli space of small planes, we argue as follows: Any (possibly broken) pseudoholomorphic plane in the compactification of $\mathcal{M}^0(\Gamma)$ can be completed to a split sphere in class $A_i$ by adjoining components produced by Lemma \ref{lem:smallplanes} below. Observe that the buildings which arise in the compactification of $\mathcal{M}^0(\Gamma)$ consist of components that are disjoint from $D_\infty$. By Proposition \ref{prop:splitsphere} no such broken plane can thus exist, since the split sphere produced above then would be of neither type I nor II. In other words, the SFT compactness theorem \cite{BEHWZ} implies that the moduli space $\mathcal{M}^0(\Gamma)$ is compact.

Lemma \ref {lem:asymptoticintersection} below shows that the asymptotic evaluation map $\mathcal{M}(\Gamma) \to \Gamma \cong S^1$ is injective. Together with Lemma \ref{lem:smallplanes} we now conclude that $\mathcal{M}(\Gamma)$ consists of only embedded planes.

(i): By the properties established above, it suffices to show that the asymptotic evaluation map is a local diffeomorphism. This follows from the automatic transversality result established in \cite[Theorem 1]{Wendl}. To that end, since $\mathcal{M}(\Gamma)$ consists of embedded (and thus, in particular, immersed) planes of index one, the inequality in \cite[Remark 1.2]{Wendl} is satisfied even with a constraint on the asymptotic orbit. In other words, this moduli space is transversely cut out and evaluates by a submersion to $\Gamma$.

(ii): By Lemma \ref{lem:smallplanes} together with Part (i), each plane in $\mathcal{M}(\Gamma)$ is embedded and, moreover, two different planes in $\mathcal{M}(\Gamma)$ are disjoint. We are left with showing that the evaluation map is a local diffeomorphism.

The local diffeomorphism property follows from the infinitesimal positivity of intersection result shown in \cite{Hofer:OnGenericity} (which is a major ingredient in the proof of automatic transversality in dimension four). More precisely, the latter result states that the elements in the kernel of the Cauchy-Riemann operator linearized at a solution $u \in \mathcal{M}(\Gamma)$ are sections of the normal bundle of $u$, all whose zeros contribute \emph{positively} to the intersection with the zero section. By Part (i), any such section which does not vanish constantly is necessarily non-vanishing near the puncture. Since two planes in $\mathcal{M}(\Gamma)$ are disjoint, we thus conclude that a non-zero element in the kernel is a \emph{non-vanishing} section. This establishes the local diffeomorphism property for the evaluation map, as sought.
\end{proof}

\begin{lemma}
\label{lem:smallplanes}
For each orbit $\gamma \in \Gamma_{\eta_i}$, there exists a split sphere in the homology class $A_i \in H_2(S^2 \times S^2)$ containing a plane in $\mathcal{M}^0(\Gamma_{\eta_i})$ whose asymptotic is equal to the orbit $\gamma$. This building moreover arises as the limit of embedded pseudoholomorphic spheres when stretching the neck, and the given plane in $\mathcal{M}^0(\Gamma_{\eta_i})$ is embedded. Finally, two planes obtained in this way that are asymptotic to different orbits must be disjoint.

Under the further assumption that the Lagrangian torus is monotone, the same properties are also satisfied for the moduli spaces $\mathcal{M}^i(\Gamma_{-\eta_i})$ of big planes. For a general (not necessarily monotone) Lagrangian torus the same holds with the exception that there might be a finite number of orbits in $\Gamma_{-\eta_i}$ that are not the asymptotic of any big plane.
\end{lemma}
\begin{proof} This property is established in a similar manner as the proof of Proposition \ref{prop:existssplit}. Namely, consider the limit of spheres passing through the geodesic on $L \subset S^2 \times S^2$ corresponding to $\gamma$ when stretching the neck around the unit cotangent bundle of $L$. Applying Proposition \ref{prop:splitsphere} to the obtained split sphere, together with the classification of cylinders in $T^*L$ given by Section \ref{sec:cylinders}, the existence of the plane can now be seen.

The embeddedness property follows from positivity of intersection \cite{McDuff:Local} together with the fact that the building is obtained as a limit of embedded pseudoholomorphic spheres. Here we have also used the fact that the plane is simply covered by Proposition \ref{prop:splitsphere}.

Finally, two different planes are disjoint by the positivity of intersection result for split spheres established in Lemma \ref{lem:nodiscrete2}.
\end{proof}
The following crucial result was proven in \cite{Hind-Lisi} by R. Hind and S. Lisi, and uses ideas due to R. Siefring and C. Wendl \cite{Siefring-Wendl} concerning the asymptotic intersection numbers in the Bott setting. In particular, we refer to \cite[Lemma 6.2 \& Appendix A]{Hind-Lisi}.
\begin{lemma}
\label{lem:asymptoticintersection}
In each of the moduli spaces $\mathcal{M}^0(\Gamma_{\eta_i})$ and $\mathcal{M}^i(\Gamma_{-\eta_i})$, there can be at most one plane asymptotic to a given periodic Reeb orbit.
\end{lemma}
\begin{proof}
Use $\mathcal{M}(\Gamma)$ to denote either of these moduli spaces of planes.

We begin by briefly recalling the definition of the extended intersection number $u \star v \in \Z$ for two planes $u,v \in \mathcal{M}(\Gamma)$; see \cite[Section 4]{Hind-Lisi} for more details. First, by extending an arbitrarily small isotopy of the Bott manifold $\Gamma \subset S^*L$ of periodic Reeb orbits to the symplectization, we can perturb one of the planes to a (no longer pseudoholomorphic) plane having a puncture asymptotic to a nearby orbit in $\Gamma$. Then, we compute the ordinary intersection number between the two planes after the perturbation, where the planes now are asymptotic to different orbits. The resulting integer is the so-called \emph{extended intersection number} $u \star v \in \Z$.

First, we claim that $u \star v=0$ holds for any pair consisting of two small or big planes from the same moduli space. Namely, for two generic such planes $u$ and $v$, Lemma \ref{lem:smallplanes} implies that they can be completed to form split spheres of type I in the homology class $A_i \in H_2(S^2 \times S^2)$. In the case when $u$ and $v$ have \emph{different} asymptotic orbits, we then necessarily have $u \star v = u \bullet v=0$ by positivity of intersection together with $A_i \bullet A_i=0$. The fact that the extended intersection number is independent up to homotopy in the appropriate sense (see \cite[Theorem 4.1]{Hind-Lisi}) then shows that $u \star v=0$ holds in general.

Second, we argue as in the latter part of the proof of \cite[Lemma 5.2]{Hind-Lisi}. Consider the asymptotic operator associated to the linearized Cauchy--Riemann operator at the negative puncture of one of these planes. It was shown in \cite[Appendix A]{Hind-Lisi} that all positive eigenvalues of this operator have a component in the direction of the contact planes whose winding number is positive (using a suitable trivialization along the Reeb orbits). Since two planes having the same asymptotic, but which do not coincide, differ by such an eigenvalue asymptotically (see \cite[Theorem 4.2]{Hind-Lisi} as well as \cite{HWZ}), we conclude that two different planes $u,v \in \mathcal{M}(\Gamma)$ sharing an asymptotic orbit thus necessarily satisfies $u \star v>0$. (This again uses positivity of intersection.) The sought statement now follows.
\end{proof}

\subsection{Constructing an embedded solid torus by straightening the ends}
Above we have shown that the union of all small planes in $\mathcal{M}^0(\Gamma_{\eta_i})$ form an embedded open solid torus. The compactification of this solid torus induced by the compactification of $S^2 \times S^2 \setminus L \subset S^2 \times S^2$ clearly is continuous, smooth in the interior, and has boundary equal to $L$. In this section we show how to deform these planes in order for the compactified solid torus to become a smooth embedding up to and including the boundary as well. To that end we make use of asymptotic properties satisfied by punctured pseudoholomorphic curves of finite energy.

Consider a family $u_\lambda$, $\lambda \in K$, of finite energy pseudoholomorphic planes which is compact in the sense of the topology defined in the same article. We begin by recalling the basic asymptotic properties satisfied near the puncture of such a finite energy plane, as established in \cite{HWZ}. Fix coordinates on the domain $(\C,i)$ given by the holomorphic parametrization
\begin{gather*}
\R \times S^1 \to \C,\\
(s,\theta) \mapsto e^{-(s+i\theta)},
\end{gather*}
where the conformal structure of the domain thus is determined by $i\partial_s=\partial_\theta$. Observe that the puncture of $\C$ at $\infty$ corresponds to $s=-\infty$ in these coordinates.

We also fix a Riemannian metric on $S^*\TT^2$ and consider the product metric on the symplectization $\R \times S^*\TT^2$, where the factor $\R$ has been endowed with the standard Euclidean metric. Similarly, consider the product metric on $\R \times S^1$. The space of maps $\R \times S^1 \to \R \times S^*\TT^2$ can now be given the metric of uniform $C^k$-distance with respect to these choices in the usual manner.

The maps $u_\lambda(s,\theta)$ above satisfy the following properties:
\begin{itemize}
\item The restriction $u_\lambda|_{(-\infty,s_\lambda]\times S^1}$ takes values in the cylindrical end $(-\infty,A] \times S^*\TT^2$ for some number $s_\lambda \in \R$, and we use $u_\lambda=(a_\lambda,\widetilde{u}_\lambda)$ to denote the components of the corresponding restriction;
\item The component $\widetilde{u}_\lambda(s,\theta) \in S^*\TT^2$ satisfies the uniform convergence
\[\lim_{s\to -\infty}\widetilde{u}_\lambda(s, \theta)=\gamma_\lambda(T\theta/2\pi)\]
to a Reeb orbit $\gamma_\lambda$ of period $T >0$ (here the Reeb orbit is parametrized using the Reeb flow for an appropriate choice of starting point);
\item The component $a_\lambda(s,\theta) \in (-\infty,A]$ satisfies the uniform convergence
\[\lim_{s \to -\infty}(a_\lambda(s,\theta)-Ts/2\pi-a_0)=0\]
for some constant $a_0 \in \R$.
\end{itemize}

The following lemma provides an extension of the previously mentioned uniform convergence of a single plane, to a version that holds for a compact family of planes. This is a crucial ingredient in the proof of our smoothing result which, in turn, is based on the ``straightening'' near the ends of a foliation by a one-parameter family of planes. The lemma follows from an argument involving the SFT compactness theorem for punctured pseudoholomorphic curves proven in \cite{BEHWZ}.

\begin{lemma}\label{lem:asymptotics}
Let $u_\lambda$, $\lambda \in K$, be a compact family of pseudoholomorphic planes of finite energy in a symplectic manifold having a concave cylindrical end $((-\infty,A] \times S^*\TT^2,d(e^t\alpha_0))$. For any $\epsilon>0$ and $k \in \Z_{\ge 0}$ and sufficiently small number $t_0 \ll 0$, the following holds:
\begin{enumerate}
\item There exists a family $S_\lambda \in \R$ of numbers bounded from below for which the restrictions
\[ u_\lambda|_{(-\infty,S_\lambda] \times S^1}(s,\theta), \:\: \lambda \in K,\]
all are $\epsilon$-close to the corresponding (translation and reparametrization of a) trivial cylinder
\[(s,\theta) \mapsto (T(\psi_\lambda(s-S_\lambda))/2\pi+t_0,\gamma_\lambda(T\theta/2\pi)),\]
in the given metric of uniform $C^k$-convergence. Here the functions $\psi_\lambda \colon \R \to \R$ are diffeomorphisms satisfying $\| \psi_\lambda'(s)-1\|_{C^{k-1}} < \epsilon$ and $\psi_\lambda(0)=0$; and
\item For all $\lambda \in K$, we have the inclusion $$u_\lambda^{-1}((-\infty,t_0-1] \times Y) \subset \{s \le S_\lambda \}$$
where the right-hand side is a neighborhood of the puncture in the domain.
\end{enumerate}
\end{lemma}
\begin{remark}
In fact, as shown in \cite{HWZ}, a finite energy plane converges \emph{exponentially} in the $C^k$-norm to the trivial strip near its puncture for every $k \ge 0$. We expect that the above lemma can be enhanced to show that such an exponential convergence holds uniformly for the whole family, i.e.~with constants independent of $\lambda \in K$. However, the weaker result established here is more than sufficient for our purposes of ``straightening'' the above foliation of pseudoholomorphic maps near the ends.
\end{remark}
\begin{proof}
The asymptotic convergence implies that each plane $u_\lambda$ has a continuous compactification $\overline{u}_\lambda$, with domain a compactification of $\C$ by the disk, into the compactified concave end
\begin{gather*}
\R \times Y \hookrightarrow [0,+\infty) \times Y, \\
(t,y) \mapsto (e^t,y).
\end{gather*}
Moreover, this compactification maps the boundary of the disk $\overline{u}_\lambda$ to the periodic Reeb orbit $\{0\} \times \gamma_\lambda \subset [0,+\infty) \times Y$.

By the definition of the topology on the moduli space of finite energy pseudoholomorphic curves in \cite{BEHWZ}, in particular see property (CHCE3) therein, the family $\{\overline{u}_\lambda\}_{\lambda \in K}$ of compactified planes is a compact subset with respect to the metric of uniform convergence. In combination with the asymptotic convergence properties of the planes $u_\lambda$ it follows that, for any $t_0 \ll 0$ sufficiently small, there are numbers $S_\lambda \ll 0$ depending on $\lambda \in K$ for which the following holds:
\begin{itemize}
\item The restrictions $u_\lambda|_{(-\infty, S_\lambda] \times S^1}$ all take values in the cylindrical end $(-\infty,t_0] \times S^*\TT^2$;
\item The components in $S^*\TT^2$ of these maps are all $C^0$-uniformly $\epsilon$-close to the periodic Reeb orbit $\gamma_\lambda(T\theta/2\pi)$ when restricted to the same subset, while $\max_{\theta \in S^1}(a_\lambda(S_\lambda,\theta))=t_0$;
\item Part (ii) of the lemma is satisfied for these numbers; and
\item The numbers $S_\lambda \ge C$ satisfy a universal bound from below.
\end{itemize}

In order to obtain the uniform $C^k$-convergence claimed in Part (i), we may need to first shrink $t_0 \ll 0$ even further. To see that the sought number $t_0$ exists we argue by contradiction; if not, we find sequences $\lambda_n \in K$ and $s_n \to -\infty$, for which $u_{\lambda_n}|_{[s_n-1/2, s_n+1/2] \times S^1}$ and any holomorphic parametrization of the trivial cylinder are of distance at least $\epsilon>0$ in the metric of uniform $C^k$-convergence.

Below we will argue that the sequence $u_{\lambda_n}(s-s_n,\theta)$ of maps has a subsequence that is convergent in the metric of uniform $C^k$-convergence on compact subsets. (It might first be necessary to translate each $u_{\lambda_n}(s-s_n,\theta)$ appropriately in the $\R$-factor of $\R \times S^*\TT^2$.) The limit is thus a finite energy $J_{\OP{cyl}}$-holomorphic cylinder $\R \times S^1 \to \R \times S^*\TT^2$. By construction, this limit is not equal to a trivial cylinder, while its component in $S^*\TT^2$ still is uniformly $\epsilon$-close to the Reeb orbit $\gamma_\lambda(T\theta/2\pi)$. However, by computing the $d\alpha_0$-energy, any sufficiently small neighborhood $\R \times U \subset \R \times S^*\TT^2$ of the periodic Reeb orbits in the family $\Gamma$ can be seen to contain only trivial cylinders (and their multiple covers). This contradiction implies the claim.

In view of the SFT compactness theorem \cite{BEHWZ}, the above convergent subsequence can be seen to exist given that the gradient of the sequence of pseudoholomorphic maps satisfy an uniform bound on every compact subset of the domain. Such a bound indeed holds, as follows from the uniform convergence in the $S^*\TT^2$-factor established above. Namely, since all finite energy pseudoholomorphic curves inside the subset $\R \times U \subset \R \times S^*\TT^2$ above are branched covers of trivial cylinders, the compactness theorem from \cite{BEHWZ} precludes bubbling from occurring. This implies the sought gradient bound.
\end{proof}

We now fix $i \in \{0,1\}$. Recall Definition \ref{def:smallbig} of the moduli spaces $\mathcal{M}^0(\Gamma_{\eta_i})$ and $\mathcal{M}^i(\Gamma_{-\eta_i})$ of big and small planes, and that the asymptotic evaluation map provides a diffeomorphism $\mathcal{M}^0(\Gamma_{\eta_i}) \to \Gamma_{\eta_i} \cong S^1$ by Proposition \ref{prop:small-planes}. Consider a closed connected arc $K \subset S^1$ which is in the image of the asymptotic evaluation map $\mathcal{M}^i(\Gamma_{-\eta_i}) \to \Gamma_{-\eta_i} \cong S^1$ from the big planes. For each big plane $P_{\OP{big}}(\theta) \in \mathcal{M}^i(\Gamma_{-\eta_i})$, $\theta \in K$, there is a unique small plane $P_{\OP{small}}(\theta) \in \mathcal{M}^i(\Gamma_{\eta_i})$ which is asymptotic to a periodic Reeb orbit corresponding to the same \emph{unoriented} geodesic on $L$ as the asymptotic of $P_{\OP{big}}(\theta)$. Recall that this Reeb orbit (and geodesic) is simply covered -- a fact which will be important in the proof of the proposition below.

In particular, the compactifications of $P_{\OP{small}}(\theta)$ and $P_{\OP{big}}(\theta)$ to disks $D_{\OP{small}}(\theta)$ and $D_{\OP{big}}(\theta)$ inside $S^2 \times S^2$ both have boundaries corresponding to the aforementioned geodesic. These two disks thus combine to form a sphere in the homology class $A_i$ intersecting $L$ precisely in this geodesic, where this sphere moreover is smooth away from its intersection with $L$.

\begin{figure}[htp]
\centering
\vspace{3mm}
\labellist
\pinlabel $\infty$ at 8 65
\pinlabel $\infty$ at 80 65
\pinlabel $D_{\OP{big}}(\theta)$ at -22 40
\pinlabel $D_{\OP{small}}(\theta)$ at -26 8
\pinlabel $\color{red}\gamma$ at 13 14
\pinlabel $\color{red}\gamma$ at 85 14
\endlabellist
\includegraphics{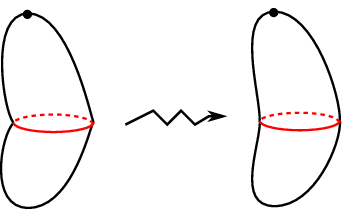}
\label{fig:smoothing}
\caption{After smoothing pairs of big and small planes asymptotic to the same closed geodesic $\gamma$, they join to form a smooth foliation of spheres intersecting $L$ in a foliation by geodesics.}
\end{figure}

\begin{prop}\label{prop:smoothdisks}
After modifying the small planes $P_{\OP{small}}(\theta)$ for all $\theta \in S^1$ as well as the big planes $P_{\OP{big}}(\theta)$ for all $\theta \in K$ in the subset $U \setminus L \subset S^2 \times S^2\setminus L$, where $U \subset S^2 \times S^2$ is an arbitrarily small neighborhood of $L$, we may assume that the following is satisfied:
\begin{enumerate}
\item The families of disks $D_{\OP{small}}(\theta)$, $\theta \in S^1$, and $D_{\OP{big}}(\theta)$, $\theta \in K$, obtained by compactifying the above deformed planes inside $S^2 \times S^2$, form smooth embeddings of $S^1 \times D^2$ and $K \times D^2$, respectively;
\item The union of disks $D_{\OP{small}}(\theta) \cup D_{\OP{big}}(\theta)$, $\theta \in K$, fit together to form a smooth embedding of $K \times S^2$, where each sphere moreover is symplectic and lives in the class $A_i \in H_2(S^2 \times S^2)$; and
\item The deformed planes may above may all be assumed to be $J_\infty$-holomorphic after deforming the almost complex structure in a compact neighborhood of $U \setminus L$.
\end{enumerate}
\end{prop}
\begin{remark}
Since the space of tame almost complex structures is contractible, the families of symplectic spheres in (ii) may clearly be assumed to be pseudoholomorphic for some global tame almost complex structure defined on $(S^2 \times S^2,\omega_1 \oplus \omega_1)$.
\end{remark}
\begin{proof}
The asymptotic properties established by Lemma \ref{lem:asymptotics} shows that we can  deform the family of planes inside the neighborhood $U \setminus L$ so that
\begin{itemize}
\item A deformed plane is asymptotic to the same cylinder as before the deformation, but coincides with the image of a trivial cylinder over the Reeb orbits inside $V \setminus L$ for the smaller neighborhood $V \subset U \subset S^2 \times S^2$ of $L$; 
\item The family of deformed planes are still disjoint and symplectic, and still provides a smooth foliation of a hypersurface in $S^2 \times S^2 \setminus L$ by symplectic planes.
\end{itemize}
Here we have used the fact that the asymptotics all are simply covered, and that two different planes are asymptotic to different Reeb orbits, together with the embeddedness properties shown in Proposition \ref{prop:small-planes}.

The sought properties can now be seen to follow from the following, to us very favorable, coincidence: Take a primitive vector $\mathbf{m} \in \Z^2 \setminus \{0\}$ and extend it to a basis $\langle \mathbf{m},\mathbf{n} \rangle = \Z^2$. Consider the one-parameter family
\begin{gather*}
u_t^\pm \colon \R \times S^1 \to \R \times S^*\TT^2,\\
(s,\theta) \mapsto (s\|\mathbf{m}\|,\pm \theta\mathbf{m}+t\mathbf{n},\pm \mathbf{m}/\|\mathbf{m}\|), \:\: t \in \R,
\end{gather*}
of $J_{\OP{cyl}}$-holomorphic trivial cylinders over periodic Reeb orbits in the families $\Gamma_{\pm \mathbf{m}} \cong S^1$. Recall the symplectic identification in Part (3) of Example \ref{ex:cyl}, identifying $\R \times S^*\TT^2$ with $T^*\TT^2 \setminus 0_{\TT^2}$. In the latter symplectic manifold, the above cylinders compactify to \emph{smooth} $J_0$-holomorphic disks (with an interior puncture removed) having a boundary equal to the geodesic on the zero-section $0_{\TT^2} \subset T^*\TT^2$ corresponding to its asymptotic Reeb orbit. Moreover, the compactifications of the two cylinders $u_t^\pm$ combine to form a smooth embedded symplectic $J_0$-holomorphic cylinder in $T^*\TT^2$ intersecting $0_{\TT^2}$ precisely along the latter geodesic. In fact, together they form the cylinder $u^{\mathbf{m}}_{t\mathbf{n},0}$ described explicitly in Section \ref{sec:cylinders}. It is hence also clear that the family of cylinders in $T^*\TT^2$ obtained in this way foliate a smoothly embedded hypersurface containing $0_{\TT^2}$.
\end{proof}

\subsection{Proof of Theorem \ref{thm:fibration}}
\label{sec:prooffibration}
Recall that we here consider the case of a monotone Lagrangian torus $L$. Using Corollary \ref{cor:monotone} together with Proposition \ref{prop:small-planes} we produce one-parameter families of small planes $P_{\OP{small}}(\theta) \in \mathcal{M}^0(\Gamma_{\eta_i})$ and big $P_{\OP{big}}(\theta) \in \mathcal{M}^i(\Gamma_{-\eta_i})$ for each homology class $A_i \in H_2(S^2 \times S^2)$, $i=1,2$. From now on we fix $i \in \{1,2\}$.

By Proposition \ref{prop:smoothdisks} we can arrange so that the families $D_{\OP{small}}(\theta)$ and $D_{\OP{big}}(\theta) \subset S^2 \times S^2$ of disks corresponding to the compactifications of these planes form symplectic foliations of two smoothly embedded solid tori whose boundaries coincide with $L$. 

Consider a closed geodesic $\widetilde{\gamma}$ on $L$ in the homology class $\eta_i \in H_1(L)$, together with the geodesic $-\widetilde{\gamma}$ having the same image but with reversed orientation in the homology class $-\eta_i$. The cogeodesic lifts of these two geodesics $\pm \widetilde{\gamma}$ are the periodic Reeb orbits $\pm \gamma \in \Gamma_{\pm \eta_i}$. Observe that there are unique planes $P_{\OP{small}}(\theta) \in \mathcal{M}^0(\Gamma_{\eta_i})$ and $P_{\OP{big}}(\theta) \in \mathcal{M}^i(\Gamma_{-\eta_i})$ that are asymptotic to $\gamma$ and $-\gamma$, respectively, and whose compactifications $D_{\OP{small}}(\theta), D_{\OP{big}}(\theta) \subset S^2 \times S^2$ fit together to form an embedded symplectic sphere in the class $A_i$.

In this way we obtain a three-dimensional embedding $S^1 \times S^2 \hookrightarrow S^2 \times S^2$ foliated by the above $S^1$-family of embedded symplectic spheres, where the intersection of the spheres with the Lagrangian torus $L$ induces the foliation by closed geodesics in the homology class $\eta_i \in H_1(L)$. Since the space of tame almost complex structures on $(S^2 \times S^2,\omega_1 \oplus \omega_1)$ is contractible, we can find a tame almost complex structure $J$ that makes all the above symplectic spheres $J$-holomorphic. Gromov's result in \cite{Gromov} together with positivity of intersection now shows that these spheres all are leaves in a smooth foliation of $S^2 \times S^2$ by $J$-holomorphic spheres in the class $A_i$, where the leaf space moreover is diffeomorphic to $S^2$. This foliation induces the sought symplectic $S^2$-fibration $p_i \colon (S^2 \times S^2,\omega_1 \oplus \omega_1) \to S^2$ by a standard argument (see \cite{Gromov}).

Let us now make the additional assumption that the spheres $\{u\} \times S^2$, $u\in U$, and $S^2 \times \{v\}$, $v \in V$, all are disjoint from $L$. Recall that the splitting construction can be performed by deforming the almost complex structure in an arbitrarily small neighborhood of $L$, while the almost complex structure is chosen arbitrarily outside of some slightly bigger neighborhood. Choosing the sequence of almost complex structures to coincide with the standard complex structure $i$ in a neighborhood of the above spheres, the sought properties are readily seen to follow by alluding to positivity of intersection.
\qed

\subsection{Proof of Corollary \ref{cor:fibration}}
First, using Theorem \ref{thm:Lagrangian-isotopy}, we Hamiltonian isotope $L$ into $S^2 \times S^2 \setminus D_\infty$. We then apply Theorem \ref{thm:fibration} to the monotone Lagrangian torus
$$L \subset (S^2 \times S^2 \setminus D_\infty,\omega_1 \oplus \omega_1),$$
thus producing the two fibrations
$$p_i \colon (S^2 \times S^2,\omega_1 \oplus \omega_1) \to S^2, \:\: i=1,2,$$
compatible with $L$, and where $D_\infty = p_1^{-1}(\infty) \cup p_2^{-1}(\infty)$ for a point $\infty \in S^2$. See Figure \ref{fig:fibers}.

\begin{figure}[htp]
\centering
\vspace{3mm}
\labellist
\pinlabel $\infty$ at 32 26
\pinlabel $0$ at 150 26
\pinlabel $\color{blue}p_1(L)$ at 58 31
\endlabellist
\includegraphics{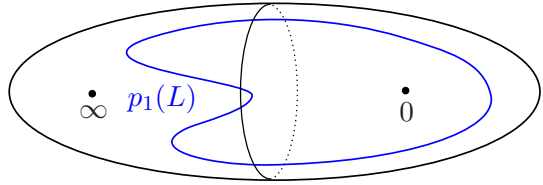}
\label{fig:fibers}
\caption{The divisor $D_\infty$ contains the fiber over $\infty \in S^2$ by assumption. We pick a point $0 \in S^2$ in the component of $S^2 \setminus p_1(L)$ which does not contain $\infty$.}
\end{figure}

We pick a fiber $p_1^{-1}(0)$ in the component of $S^2 \setminus p_1(L)$ that does not contain $\infty \in S^2$. Using Gromov's classification of pseudoholomorphic foliations of $(S^2 \times S^2,\omega_1 \oplus \omega_1)$ together with Corollary \ref{cor:hamiso}, we can construct a Hamiltonian isotopy that fixes $D_\infty$ after which $p_1^{-1}(0)=\{0\} \times S^2$.

After a second application of Theorem \ref{thm:fibration}, the union $\{ 0,\infty\} \times S^2$ of spheres can be assumed to be fibers of $p_1$ and sections of $p_2$, respectively, while $S^2 \times \{\infty\}$ is a fiber of $p_2$ and a section of $p_1$, respectively. Finally, we pick a fiber $p_2^{-1}(0)$ in the component of $S^2 \setminus p_2(L)$ which does not contain $\infty$. As above, we may assume that this fiber coincides with $S^2 \times \{0\}$.

The prescribed linking behavior of $L$ and
$$ S^2 \times \{0,\infty\} \: \cup \: \{0,\infty\} \times S^2$$
can now be seen to follow from a topological consideration.
\qed

\section{Constructing Lagrangian isotopies}
Consider a Lagrangian torus $L \subset (S^2 \times S^2 \setminus D_\infty,\omega_1 \oplus \omega_1)$. The above Propositions \ref{prop:splitsphere} and \ref{prop:smoothdisks} establish the existence of a smooth solid torus $\mathcal{T} \subset S^2 \times S^2 \setminus D_\infty$ with boundary equal to $L$. By construction this solid torus is moreover foliated by the symplectic disks $D_{\OP{small}}(\theta)$, $\theta \in S^1$, being compactifications of the small pseudoholomorphic planes obtained from the splitting construction (see Definition \ref{def:smallbig}).

The so called \emph{characteristic distribution} of $\mathcal{T}$ is the one-dimensional kernel $\ker (\omega_1 \oplus \omega_1)|_{T\mathcal{T}}$ of the restriction of the symplectic form to the tangent space of this solid torus. Observe that this distribution is tangent to the boundary $\partial \mathcal{T}=L$, as implied by the Lagrangian condition of $L$. Integrating a suitably normalized non-vanishing vector field tangent to this distribution, the induced flow gives a monodromy map for each leaf $D_{\OP{small}}(\theta)$. This monodromy map is a symplectomorphism of each disk preserving its boundary set-wise.

The following theorem, which only uses ``soft'' techniques, shows that solid tori are useful for producing Lagrangian isotopies.
\begin{thm}[Proposition 3.4.6 in \cite{Ivrii-thesis}] \label{adjust}
Assume that we are given two smooth embeddings
\[\varphi_i \colon S^1 \times D^2 \hookrightarrow (X^4,\omega), \:\:i=0,1, \]
of solid tori into a symplectic four-dimensional manifold satisfying:
\begin{itemize}
\item The disks $\varphi_i(\{\theta\} \times D^2) \subset (X^4,\omega)$ are symplectic while the boundary $\varphi_i(S^1 \times \partial D^2) \subset (X^4,\omega)$ is Lagrangian for both $i=0,1$;
\item The characteristic distribution on both solid tori $\varphi_i(S^1 \times D^2) \subset (X^4,\omega)$, $i=0,1$, induces monodromy maps equal to the identity; and
\item The two maps $\varphi_0$ and $\varphi_1$ are homotopic.
\end{itemize}
Then the two Lagrangian tori $\varphi_i(S^1 \times \partial D^2)$, $i=0,1$ are Lagrangian isotopic.
\end{thm}
In view of this result, our main task in this section will be deforming a given solid torus to one for which the monodromy map of the leaves is the identity.

It will be crucial to use additional properties that are satisfied by the solid tori considered here, i.e.~arising from the construction in Section \ref{sec:fibration}. Namely, using methods similar to the proof of Theorem \ref{thm:fibration} in Section \ref{sec:prooffibration}, such a solid torus $\mathcal{T}$ can be shown to be compatible with a symplectic fibration
$$p \colon (S^2 \times S^2,\omega_1 \oplus \omega_1) \to S^2,$$
$$[p^{-1}(q)] =A_2 =[\{\pt\} \times S^2] \in H_2(S^2 \times S^2),$$
in the following sense:
\begin{enumerate}[label=(C.\roman*)]
\item There exists a neighborhood $U:=p^{-1}(V)$, where $V \subset S^2$ is a neighborhood diffeomorphic to a square $[-1,1]^2$, for which the following is satisfied:
\begin{enumerate}
\item The solid torus satisfies the property that $p(\mathcal{T}) \cap V \subset V$ is a single embedded arc $\gamma$ that can be identified with $[-1,1] \times \{0\} \subset [-1,1]^2$;
\item The intersection of the solid torus with a fiber $p^{-1}(v)$, $v \in V$, is either empty or consists of a single embedded disk;
\end{enumerate}
\item There exists a symplectic section $\Sigma$ of $p$ in the homology class $A_1$ which is disjoint from $\mathcal{T}$.
\end{enumerate}
Given that the above properties are satisfied, after a small perturbation of the fibers inside $p^{-1}(V\setminus\gamma)$ and after shrinking the subset $V \subset S^2$ in the base, Lemma \ref{lma:trivialization} below shows that we in addition can assume that:
\begin{enumerate}
\item[(C.iii)] There exists a symplectic trivialization
\[\psi \colon ([-a,a]^2 \times S^2,(dx \wedge dy) \oplus \omega_1) \xrightarrow{\cong} (U=p^{-1}(V),\omega_1 \oplus \omega_1) \]
for some $a>0$, under which $p \circ \psi$ is the canonical projection $[-a,a]^2 \times S^2 \to [-a,a]^2$ followed by an embedding $[-a,a]^2 \hookrightarrow S^2$, and for which the solid torus takes the form
\[\psi^{-1}(\mathcal{T})=[-a,a]\times\{0\} \times D \subset [-a,a]^2 \times S^2\]
for a smooth embedding $D \subset S^2$ of a disk; See Figure \ref{fig:trivialization}.
\end{enumerate}

\begin{figure}[htp]
\centering
\vspace{3mm}
\labellist
\pinlabel $p^{-1}(l)$ at 85 35
\pinlabel $D$ at 48 85
\pinlabel $\color{blue}L\cap p^{-1}(l)$ at 50 117
\pinlabel $\color{blue}p(L)$ at 187 48
\pinlabel $V\cong [-a,a]^2$ at 157 78
\pinlabel $\color{blue}l$ at 157 46
\endlabellist
\includegraphics{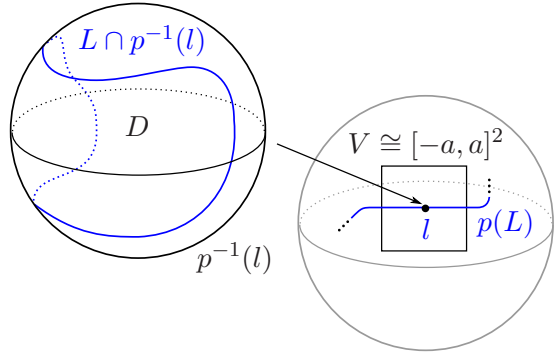}
\caption{After deforming the fibration over the neighborhood $V \subset S^2$ of the base, we may assume that it is symplectically trivial above this neighborhood. The goal of the inflation procedure is making the symplectic area in the ``horizontal component'' arbitrarily highly concentrated above $V$.}
\label{fig:trivialization}
\end{figure}

Using Properties (C.i)--(C.iii) we are able to show the following.
\begin{thm}\label{monodromyidentity}
Let $L \subset (S^2 \times S^2,\omega_1 \oplus \omega_1)$ be a Lagrangian torus given as the boundary of the solid torus $\mathcal{T}$, and for which the symplectic fibration $p \colon (S^2 \times S^2,\omega_1 \oplus \omega_1)\to S^2$ satisfies (C.i)--(C.iii). After a Lagrangian isotopy of $L$, we can find such a bounding solid torus with monodromy map equal to the identity.

Further, assume that the spheres $\{\infty\} \times S^2$ and $S^2 \times \{\infty\}$ is a fiber and a section of $p$, respectively, both contained in the complement of $\mathcal{T}$. Then the above Lagrangian isotopy as well as the produced solid torus may be confined to the complement of $D_\infty=(\{\infty\} \times S^2) \cup (S^2 \times \{\infty\})$.
\end{thm}
In order to make the necessary modification of the solid torus, we must first guarantee that there is a sufficient amount of space in its ``normal direction''. The so-called inflation procedure is a method which can be used to create the space needed, but with the caveat that we first have to deform the Lagrangian torus by a Lagrangian isotopy.

\subsection{The inflation procedure}
\label{sec:inflation}
Before formulating the inflation procedure, we present the following facts, which are to be proven below.

Assume that the symplectic fibration $p \colon (S^2 \times S^2,\omega_1 \oplus \omega_1) \to S^2$ above satisfies (C.i)--(C.iii). In Lemma \ref{lem:twoform} below we produce a closed two-form $\tau$ for which the following is satisfied:
\begin{enumerate}[label=(F.\roman*)]
\item The support of $\tau$ is disjoint from $\mathcal{T}$;
\item $\int_\Sigma\tau=0$, i.e.~the two-form $\tau$ integrates to 0 over any sphere in the class of a section of $p$;
\item The two-form $\tau$ is non-negative when restricted to any fiber of $p$, in particular $\int_{p^{-1}(q)}\tau>0$; and
\item For any non-negative constant $c>0$ the form $\omega_1 \oplus \omega_1 +c\tau$ on $S^2 \times S^2$ is symplectic;
\item Inside the trivialization $\psi$ given by (C.iii), after possibly shrinking the subset $V \subset S^2$ in the base, the form $\tau$ is obtained as the pull-back of a two-form on the $S^2$-factor.
\end{enumerate}
Using the existence of the above closed two-form, we are now ready to prove the following so-called inflation procedure.
\begin{thm}[Proposition 3.3.1 in \cite{Ivrii-thesis}] \label{inflation}
There exists a smooth isotopy $\phi_t \colon S^2 \times S^2 \to S^2 \times S^2$, $t \ge 0$, $\phi_0=\id_{S^2 \times S^2}$ satisfying the following properties:
\begin{enumerate}
\item $p_t := p \circ (\phi_t)^{-1}$ is a family of symplectic $S^2$-fibrations on $(S^2 \times S^2,\omega_1 \oplus \omega_1)$;
\item There exists a fixed neighborhood $W$ of $\mathcal{T}$ in which $\phi_t^*(\omega_1 \oplus \omega_1)|_W=r_t(\omega_1 \oplus \omega_1)$ is satisfied, and where $r_t > 0$ is a constant satisfying $0<r_t<\epsilon$ for all $t \gg 0$ sufficiently large and arbitrary $\epsilon>0$; and
\item The above symplectic trivialization $\psi$ of $p$ induces a symplectic trivialization
$$\psi_t \colon ([-a,a]^2 \times S^2,(\alpha_t \, dx \wedge dy) \oplus \beta_t \omega_1) \hookrightarrow (S^2 \times S^2,\omega_1 \oplus \omega_1),$$
$$ \psi_t :=\phi_t \circ \psi, $$
of the fibration $p_t$, for functions $\alpha_t \colon [-a,a]^2 \to \R_{>0}$ and $\beta_t \colon S^2 \to \R_{>0}$ that are constant near $[-a,a] \times \{0\}$ and $ D \subset S^2$, respectively (recall that $[-a,a] \times \{0\}\times D=\psi^{-1}(\mathcal{T})$). Moreover, for an arbitrary $\epsilon>0$ we have
$$1-\epsilon \le \int_{[-a,a]^2} \alpha_t \, dx\wedge dy \le 1$$
for all $t \gg 0$ sufficiently large.
\end{enumerate}
Further, assume that the spheres $\{\infty\} \times S^2$ and $S^2 \times \{\infty\}$ is a fiber and a section of $p$, respectively, both contained in the complement of $\mathcal{T}$. Then the isotopy $\phi_t$ may be assumed to fix the subset $(\{\infty\} \times S^2) \cup (S^2 \times \{\infty\})=D_\infty$ set-wise, while $\psi_t^{-1}(D_\infty) = [-a,a]^2 \times \{\infty\}$ is satisfied.
\end{thm}
\begin{remark}
\label{rmk:inflation}
By Part (ii) it follows that $\mathcal{T}_t:=\phi_t(\mathcal{T}) \subset (S^2 \times S^2,\omega_1 \oplus \omega_1)$ is a family of solid tori foliated by symplectic disks. The boundaries $L_t := \partial \mathcal{T}_t \subset (S^2 \times S^2,\omega_1 \oplus \omega_1)$ of these solid tori provide a smooth family of Lagrangian tori starting at $L_0=L$. In the trivialization $\psi_t$ given by Part (iii) these solid tori moreover take the form
\[ \psi_t^{-1}(\mathcal{T}_t)=[-a,a] \times \{0\} \times D \subset [-a,a]^2 \times S^2 \]
for a smooth embedding $D \subset S^2$ of a disk, as follows from the properties of $\psi$ postulated by Property (C.iii).
\end{remark}
\begin{proof}
Here we write $\omega:=\omega_1 \oplus \omega_1$ to denote the product symplectic form on $S^2 \times S^2$.

Using the symplectic trivialization
\[\psi \colon ([-a,a]^2 \times S^2,(dx \wedge dy) \oplus \omega_1) \xrightarrow{\cong} U=p^{-1}(V) \subset (S^2 \times S^2,\omega_1 \oplus \omega_1) \]
provided by Property (C.iii), it follows that $\omega + t\kappa$ is symplectic for each $t \ge 0$, where we have defined
\[\kappa := (\psi^{-1})^*(\rho\, dx\wedge dy)\]
for a bump-function $\rho \colon [-a,a]^2 \to [0,1]$ on the $[-a,a]^2$-factor. Here we moreover require the support of $\rho$ to be compact and disjoint from $[-a,a]\times\{0\}$ (i.e.~it is disjoint from the image of the solid torus under the projection to the base in the above coordinates).

It follows that $p \colon (S^2 \times S^2,\omega + t\kappa) \to S^2$ still defines a symplectic fibrations for each of the symplectic forms $\omega + t\kappa$, $t \ge 0$.

By Properties (F.i)--(F.iv) above, there exist a closed two-form $\tau$ on $S^2 \times S^2$, for which $\omega + t\kappa+c_t\tau_t$ remains symplectic for an arbitrary family of constants $c_t\ge 0$ (here we use Properties (F.iv) and (F.v)) while, moreover, $p \colon (S^2 \times S^2,\tilde\omega_s+c_t\tau)\to S^2$ still defines a symplectic fibration (here we use Properties (F.iii) and (F.v)). Recall that $\tau$ has support disjoint from $\mathcal{T}$, that $\int_\Sigma\tau_t=0$, and that $\int_{p^{-1}(q)}\tau>0$.

Consider the area ratio
\[ A_t:=\frac{\int_{\Sigma}(\omega+t\kappa)}{\int_{\Sigma}\omega}, \:\: t \ge 0,\]
where thus $A_0=1$, and for which $\dot{A_t}\ge C>0$ is bounded from below.  We choose the constants $c_t \ge 0$ to be defined via the equation
\[\frac{\int_{p^{-1}(q)}(\omega+c_t\tau)}{\int_{p^{-1}(q)}\omega}=A_t, \:\: t \ge 0.\]
In particular it follows that $c_0=0$, and that $\dot{c_t} \ge C' >0$ is bounded from below as well.

Finally, we can construct the smooth family
\[\omega_t:=\frac{1}{A_t}(\omega+t\kappa+c_t\tau), \:\: t \ge 0, \]
of two-forms. By Property (F.iv) the these forms are symplectic, and using the fact that $\int_\Sigma \omega_t=\int_{p^{-1}(q)}\omega_t=1$ holds for all $t \ge 0$, they are all in the same cohomology class as $\omega_0=\omega$.

By the latter property, Moser's trick \cite{McDuff:IntroSympTop} can be applied to the above family $\omega_t$ of symplectic forms. This yields a smooth isotopy $\phi_t: S^2 \times S^2 \rightarrow S^2 \times S^2$ for which $\phi_0=\id_{S^2 \times S^2}$, while $(\phi_t)^*\omega=\omega_t$. Parts (i)--(iii) can now readily seen to hold.

Given the additional assumptions, instead of applying Moser's trick to the family $\omega_t$ of symplectic forms, we will instead apply it to a family $\omega_t'=(\phi'_t)^*(\omega_t)$ of symplectic forms which are required to coincide with $\omega$ in some neighborhood of $D_\infty$, and where $\phi'_t \colon S^2 \times S^2 \to S^2 \times S^2$ is a suitable smooth isotopy fixing $D_\infty$ set-wise. Observe that such a smooth isotopy can be constructed by hand, given that the closed two-form $\tau$ used above is constructed with some care. To that end, the symplectic section $\Sigma$ used in Lemma \ref{lem:twoform} for producing $\tau$ will be taken of the form $S^2 \times \{w\} \subset D_\infty$ for some $w \in S^2 \setminus \{\infty\}$ that is close, but not equal, to $\infty$.

Using the fact that $D_\infty$ has a neighborhood that is simply connected, and examining the proof of Moser's trick in the case of a family $\omega_t'$ as above, the produced smooth isotopy $\phi_t''$ may be taken to fix such a simply connected neighborhood of $D_\infty$ pointwise. The sought isotopy is finally taken to be $\phi_t := \phi_t' \circ \phi_t''$.
\end{proof}

\begin{lemma}[Lemma 3.3.2 in \cite{Ivrii-thesis}]
\label{lem:twoform}
Under the assumptions (C.i)--(C.iii), the closed two-form $\tau$ on $S^2\times S^2$ satisfying (F.i)--(F.v) above exists and may be assumed to have support inside an arbitrarily small neighborhood of the symplectic section $\Sigma$.
\end{lemma}
\begin{proof}
Since $\Sigma$ is a symplectic section of $p$, all symplectic fibers $p^{-1}(q)$ intersect $\Sigma$ transversely in a single point. This intersection is counted with positive sign given that both spheres are endowed with the symplectic orientation.

After a deformation of $\Sigma$ and shrinking the open subset $V \subset S^2$ of the base where $p$ is trivialized (as postulated by Property (C.iii)), we may assume that this section is constant with respect to the trivialization $\psi$.

Standard techniques can be used to produce a smooth trivialization of $p$ restricted to some neighborhood $W \supset \Sigma$, i.e.~a diffeomorphism
\[\widetilde{\psi} \colon S^2 \times D_\epsilon^2 \xrightarrow{\cong} W \subset S^2 \times S^2\]
which identifies $S^2 \times \{0\} \subset S^2 \times D_\epsilon^2$ with $\widetilde{\psi}(S^2 \times \{0\})=\Sigma$, and for which $p \circ \widetilde{\psi}$ is equal to the canonical projection $S^2 \times D_\epsilon^2 \to S^2$. After shrinking $\epsilon>0$, and taking some extra care we may moreover assume that
\begin{itemize}
\item all constant sections $S^2 \times \{q\}$ in the trivialization $\widetilde{\psi}$ are symplectic, and
\item the trivializations $\psi$ and $\widetilde{\psi}$ coincide inside the subset $p^{-1}(V) \cap W$.
\end{itemize}

Giving $S^2 \times D_\epsilon^2$ the appropriate product orientation, while making the above diffeomorphism orientation preserving, we may write
\[\widetilde{\psi}^*(\omega_1 \oplus \omega_1)=f \,dx \wedge dy+\eta_1 \wedge dx +\eta_2 \wedge dy +g\, \omega_1. \]
Here $(x,y) \in D_\epsilon^2 \subset \R^2$ denote the standard coordinates, $\omega_1$ denotes the area form on the $S^2$-factor, $\eta_1,\eta_2 \in \Omega^1(S^2)$ are one-forms on the $S^2$-factor, and $f,g >0$ are positive functions. Moreover, since $\widetilde{\psi}$ preserves orientations, it follows that
\[\widetilde{\psi}^*(\omega_1 \oplus \omega_1) \wedge \widetilde{\psi}^*(\omega_1 \oplus \omega_1)  =h \, \omega_1 \wedge dx \wedge dy  \]
holds for a positive function $h>0$.

An explicit calculation using the formulas above shows that the closed two-form
$$\widetilde{\psi}^*(\omega_1 \oplus \omega_1)+\rho\, dx \wedge dy$$
also is symplectic for any choice of non-negative function $\rho \colon D_\epsilon^2 \to \R_{\ge 0}$ on the factor $D_\epsilon^2$ corresponding to the fiber. The sought two-form can now be taken to be $\tau:=(\widetilde{\psi}^{-1})^*(\rho\,dx \wedge dy)$ for a smooth and compactly supported bump-function $\rho \colon D_\epsilon^2 \to [0,1]$.
\end{proof}

\begin{lemma}[Lemma 3.2.4 in \cite{Ivrii-thesis}]
\label{lma:trivialization}
Assume we are given a symplectic $S^2$-fibration $p \colon (S^2 \times S^2, \omega_1 \oplus \omega_1) \to S^2$ satisfying Properties (C.i) and (C.ii) above. After an arbitrarily small deformation of the fibers in the subset $p^{-1}(V \setminus \gamma)$, and after possibly shrinking $V$, we may assume that Property (C.iii) concerning the existence of the symplectic trivialization $\psi$ over $V$ holds as well.

Moreover, given that $S^2 \times \{0\} \subset S^2 \times S^2 \setminus \mathcal{T}$ is a section of the original fibration $p$, and that the fibers of $p$ above $\gamma \subset V$ coincide with the standard spheres $\{ g \} \times S^2$, $g\in \gamma$, in some neighborhood of $S^2 \times \{0\}$, we may assume that this section is constant in the produced trivialization $\psi$.
\end{lemma}
\begin{proof}
Using the characteristic distribution of the hypersurface $p^{-1}(\gamma) \subset (S^2 \times S^2,\omega_1 \oplus \omega_1)$ we define a trivialization $\widetilde{\psi} \colon \gamma \times S^2 \to U$ over an arc that can be identified with $\gamma=[-a,a] \times \{0\} \subset [-a,a]^2 \cong V$. In other words, all $\widetilde{\psi}( \{ g \} \times S^2)=p^{-1}(g)$, $g \in \gamma$, are fibers, while $D\widetilde{\psi}(T\gamma \oplus 0)$ is tangent to the characteristic distribution. It follows that the solid torus is of the required form $\widetilde{\psi}^{-1}(\mathcal{T})=\gamma \times D$.

Using the standard symplectic neighborhood theorem (see e.g.~\cite[Lemma 3.14]{McDuff:IntroSympTop}) we may then extend $\widetilde{\psi}$ to a symplectomorphism
\[\psi \colon ([-a,a]^2 \times S^2, (dx \wedge dy) \oplus \omega_1) \xrightarrow{\cong} U=p^{-1}(V) \subset (S^2 \times S^2,\omega_1 \oplus \omega_1),\]
where the pair $(V,\gamma)$ has been identified with $([-a,a]^2,[-a,a] \times \{0\})$.

Observe that, except over the arc $[-a,a] \times \{0\} \subset [-a,a]^2$, the symplectomorphism $\psi$ does not necessarily preserve the fibers. However, this can be amended by performing a suitable smooth interpolation between the fibers of $p$ and the symplectic spheres $\psi (\{(x,y)\} \times S^2)$. After possibly shrinking the subset in the base and choosing a smaller number $a>0$, we have obtained the sought trivialization of the deformed fibration.
\end{proof}

\subsection{Proof of Theorem \ref{monodromyidentity}}
\label{sec:proofmonodromyidentity}
The result follows by combining the inflation technique of Theorem \ref{inflation} with the following elementary but crucial lemma.
\begin{lemma}[Proposition 3.4.2 in \cite{Ivrii-thesis}]\label{diskmonodromy}
Consider the symplectic manifold $([0,1]\times\R \times S^2,(dx\wedge dy) \oplus \omega_1)$. Let $D\subset S^2$ be a disk with smooth boundary, and consider and a Hamiltonian diffeomorphism $\phi^1_{H_t}: (D, \partial D) \rightarrow (D, \partial D)$ preserving the boundary set-wise, and satisfying $H_0 \equiv H_1 \equiv 0$. The cylinder $C$ being the image of $[0,1] \times \{0\} \times D$ under the symplectic suspension
\begin{gather*}
([0,1]\times \R \times D,(dx\wedge dy) \oplus \omega_1) \to ([0,1]\times \R \times D,(dx\wedge dy) \oplus \omega_1),\\
(x,y,z) \mapsto (x,y+H_x(z),\phi^x_{H_t}(z)),
\end{gather*}
satisfies
\[C \cap \{ x=0,1\} = \{0,1\} \times \{0\} \times D \subset [0,1]\times\R\times S^2, \]
is foliated by the symplectic disks $C \cap \{ x = x_0\}$, $x_0 \in [0,1]$, and has a characteristic distribution inducing the map $\phi^1_{H_t}$ when flowing from $C \cap \{x=0\}$ to $C \cap \{x=1\}$.
\end{lemma}
The following well-known result is also needed.
\begin{lemma}\label{lma:hamiso}
Any symplectomorphism of the two-dimensional disk $(D^2,\omega_0)$ which fixes the boundary set-wise is a Hamiltonian diffeomorphism $\phi^1_{H_t}$ for a time-dependent Hamiltonian $H_t$ on $D^2$. This Hamiltonian can be taken to satisfy $H_t \equiv 0$ for all $t$ in a neighborhood of $\{0,1\}$, as well as $H_t|_{\partial D^2} \equiv 0$ for all $t \in [0,1]$.
\end{lemma}
\begin{proof}
By Moser's trick (see e.g.~\cite{McDuff:IntroSympTop}) together with the fact that the space of symplectic forms on a surface inducing a fixed orientation is a contractible space, we get a weak homotopy equivalence $\OP{Symp}(D^2) \sim \OP{Diff}^+(D^2)$. Since the latter group of orientation preserving diffeomorphisms is contractible by \cite{SmaleDiffeo}, and since $D^2$ is simply connected, we thus conclude that $\OP{Symp}(D^2)=\OP{Ham}(D^2)$.

The vanishing $H_t|_{\partial D} \equiv 0$ is easily achieved by adding a suitable time-dependent constant to the Hamiltonian. Namely, since the boundary is a Lagrangian submanifold which is fixed set-wise, the Hamiltonian must be constant there for each fixed $t \in [0,1]$.
\end{proof}
Now assume that the solid torus $\mathcal{T} \subset S^2 \times S^2$ induces a monodromy map $\phi \colon (D,\omega) \to (D,\omega)$, and that it satisfies Properties (C.i)--(C.iii) for some symplectic fibration. Consider a fixed Hamiltonian isotopy $\phi^t_{H_t} \colon (D,\omega) \to (D,\omega)$ for which $\phi^1_{H_t}=\phi^{-1}$, where $H_t$ moreover satisfies the properties described in Lemma \ref{lma:hamiso}.

The inflation procedure in Theorem \ref{inflation} deforms the Lagrangian torus by a Lagrangian isotopy $L_t$ that extends to the isotopy $\mathcal{T}_t$ of the solid torus. Given any $\epsilon>0$, the monodromy map of $\mathcal{T}_t$ can be assumed to be given by given by $\phi \colon (D,\epsilon\omega) \to (D,\epsilon\omega)$ for some large $t>0$ (this is by Part (ii) of Theorem \ref{inflation}).

Utilizing the symplectic trivialization near a neighborhood of the deformed solid torus $\mathcal{T}_t$ as provided by Part (iii) of Theorem \ref{inflation}, we can now apply the above lemma to the Hamiltonian isotopy $\phi^t_{\epsilon H_t}$ given that $\epsilon>0$ was chosen sufficiently small and $t \gg 0$ sufficiently large. The solid torus obtained by replacing the cylinder $[0,1] \times \{0\} \times D$ in this trivialization with the obtained cylinder $C$ finally produces the sought solid torus. Indeed, its monodromy map is equal to $(\phi^t_{\epsilon H_t})^{-1} \circ \phi^t_{\epsilon H_t}=\id_{D}$. Moreover, since $H_t|_{\partial D} \equiv 0$, the Lagrangian boundary of this solid torus is not affected by the latter deformation.

Strictly speaking, some additional care must be taken here due to the fact that the symplectic form $\alpha_t\, dx \wedge dy$ on the base of the trivialization $\psi_t$ produced by Theorem \ref{inflation} may be small on one of the subsets $\{ \pm y \ge 0 \} \subset [-a,a]^2$. However, this can always be amended after an additional explicit deformation of the solid torus inducing a Lagrangian isotopy of its boundary: Inside the trivialization we can replace the solid torus with one that is fibered over a different curve in the base.
\qed

\subsection{Proof of Theorem \ref{adjust}}
First of all, using the standard symplectic neighborhood theorem (see e.g.~\cite[Lemma 3.14]{McDuff:IntroSympTop}), we can extend both embeddings of solid tori $\varphi_i \colon S^1 \times D^2 \to X$ to symplectic embeddings
\[ \Phi_i \colon (T_\epsilon^*S^1 \times D^2, d\lambda_{S^1} \oplus R \omega_0) \hookrightarrow (X,\omega), \:\: i=0,1,\]
for some $\epsilon,R>0$, satisfying the property that $\Phi_i|_{0_{S^1} \times D^2}=\varphi_i$. Our goal is to construct a Lagrangian isotopy connecting the two Lagrangian tori $\Phi_i(0_{S^1} \times \partial D^2)$, $i=0,1$.

Observe that a one-dimensional submanifold automatically is isotropic, i.e.~the symplectic form vanishes when pulled back to it. Since the solid tori are homotopic, the two isotropic cores $\varphi_i(S^1 \times \{0\})$, $i=0,1$, of the solid tori are homotopic as well. By a standard general position argument, we may even assume that there is a smooth isotopy connecting these two cores.

We now recall the Lagrangian circle bundle construction due to M. Audin, F. Lalonde, and L. Polterovich in \cite{Audin-Lalonde-Polterovich} in the special case of a four-dimensional symplectic manifold. This construction enables us to create a smooth family of Lagrangian tori associated to any smooth family of isotropic curves.

Namely, assume that we are given a curve $\gamma \colon S^1 \hookrightarrow (X,\omega)$ and observe that its symplectic normal bundle $\nu \subset \gamma^*TX \to S^1$ is symplectically trivial (this is a two-dimensional real vector bundle). Given the choice of a symplectic trivialization $\tau$ of $\nu$, the isotropic neighborhood theorem proven by A. Weinstein in \cite{Weinstein:Isotropic} (also, see \cite{McDuff:IntroSympTop}) provides a symplectic embedding
\[ \Gamma \colon (T^*_{\epsilon}S^1 \times D^2_\epsilon,d\lambda_{S^1} \oplus \omega_0) \hookrightarrow (X,\omega)\]
for some $\epsilon>0$, where $\Gamma|_{0_{S^1} \times \{0\}}=\gamma$ and under which the canonical trivialization of $0\oplus TD^2_\epsilon$ along $0_{S^1} \times \{0\}$ coincides with the above trivialization $\tau$. We will call any symplectic embedding $\Gamma$ of this form a \emph{parametrized symplectic standard neighborhood of $\gamma$} induced by our choice of trivialization.

The Lagrangian circle bundle construction associates the embedding $\Gamma(0_{S^1} \times S^1_\epsilon) \subset (X,\omega)$ of a Lagrangian torus to any parametrized symplectic standard neighborhood of an isotropic curve. Given a one-parameter family $\Gamma_t$, $t \in [0,1]$ of such symplectic parametrizations, we obviously get an induced Lagrangian isotopy of the corresponding Lagrangian circle bundles. Observe that we first may have to shrink $\epsilon>0$, but that $t \mapsto \Gamma(0_{S^1} \times S^1_t)$ again is a Lagrangian isotopy.

We now make the following claims, from which the sought result can be seen to follow:
\begin{enumerate}[label=(\arabic*)]
\item The parametrization produced by the isotropic neighborhood theorem smoothly depends on the data (the isotropic curve and the trivialization of its symplectic normal bundle), together with the contractible space of choices made in the construction (e.g.~a Riemannian metric on $(X,\omega)$). (For a family of such parametrized neighborhoods it might obviously be necessary to also vary $\epsilon>0$);
\item Any parametrized symplectic standard neighborhood of an isotropic curve with a fixed trivialization may be assumed to arise from an application of the previously mentioned isotropic neighborhood theorem for suitable choices of data. In particular, this is the case for the parametrizations $\Phi_i$, $i=0,1$, above, which can be taken to be induced by the isotropic curves $\Phi_i(0_{S^1} \times \{0\})$, $i=0,1$; and
\item Given a curve $\gamma$, any two homotopy classes of trivializations of its symplectic normal bundle give rise to parametrized symplectic normal neighborhoods for which the associated Lagrangian circle bundles are Lagrangian isotopic.
\end{enumerate}

Claim (1) is standard. Claim (2) is immediate, given that the data used in the construction of the isotropic neighborhood theorem is adapted to the parametrized standard symplectic neighborhood already given.

We finish by showing Claim (3) by giving the following explicit models which is sufficient in view of Claim (2). Recall that $\pi_1(\OP{Sp}_2)=\pi_1(U(1))=\Z$, and hence that there is an integer worth of homotopy classes of trivializations of the symplectic normal bundle of an isotropic curve in $(X,\omega)$. To each such homotopy class $m \in \Z$ associated to the isotropic curve $0_{S^1} \times \{0\} \subset (T^*S^1 \times \R^2,d\lambda \oplus \omega_0)$, we explicitly construct the parametrized symplectic standard neighborhood
\begin{gather*}
\Gamma_m \colon (T^*_{\epsilon}S^1 \times D^2_\epsilon,d\lambda_{S^1} \oplus \omega_0)\to(T^*S^1 \times \R^2,d\lambda_{S^1} \oplus \omega_0),\\
(\theta,p_\theta,z)\mapsto(\theta,p_\theta- m\|z\|^2/2,e^{im\theta}z),
\end{gather*}
where we have used the identification $S^1=\R/2\pi\Z$. The claim now follows since, in fact, the Lagrangian tori $\Gamma_m(0_{S^1} \times S^1_\epsilon)$, $m \in \Z$, produced by the Lagrangian circle bundle construction all coincide.
\qed

\subsection{Proof of Theorem \ref{thm:Lagrangian-isotopy}}
By Theorem \ref{thm:complement} it suffices to show that two Lagrangian tori inside $(\R^4,\omega_0)$ are Lagrangian isotopic. This we now show as follows. After a rescaling, it suffices to consider a Lagrangian torus
\[L \subset (D^2 \times D^2,\omega_0) \cong (S^2 \times S^2 \setminus D_\infty,\omega_1 \oplus \omega_1).\]
Propositions \ref{prop:small-planes} and \ref{prop:smoothdisks} produce an embedded solid torus $\bigcup_{\theta \in S^1} D_{\OP{small}}(\theta)$ with boundary on $L$ contained inside $D^2 \times D^2 \subset S^2 \times S^2$. This solid torus moreover satisfies Properties (C.i)--(C.ii), and hence also (C.iii) without loss of generality, from which it follows that Theorem \ref{monodromyidentity} can be applied. In other words, after a Lagrangian isotopy of $L \subset (D^2 \times D^2,\omega_0)$, we can find such a solid torus for which the monodromy map of a leaf $D_{\OP{small}}(\theta)$ induced by the characteristic distribution is the identity. Finally, using Theorem \ref{adjust} we conclude that the Lagrangian isotopy class of such a torus is unique.\qed

\section{The nearby Lagrangian conjecture}
\label{sec:nearby}
Using the above construction of symplectic $S^2$-fibrations of $(S^2 \times S^2,\omega_1 \oplus \omega_1)$ compatible with a Lagrangian torus, together with the technique of inflation, we are also able to obtain the following strong result concerning homologically essential Lagrangian tori of the cotangent bundle of a torus. 
\begin{thm}\label{T22}
Any embedded Lagrangian torus $L\subset T^*\TT^2$ that is non-zero in $H_2(\TT^2)$ is Hamiltonian isotopic to a section of $T^*\TT^2 \to \TT^2$.
\end{thm}
Combining this theorem with the results by Abouzaid--Kragh \cite{Abouzaid}, \cite{Kragh}, we obtain a positive answer to the nearby Lagrangian conjecture for $(T^*\TT^2,d\lambda)$, i.e.~this establishes Theorem \ref{thm:nearby}.

\subsection{The geometric setup}
\label{sec:geomsetup}

We begin by discussing the geometric setup, relating the cotangent bundle of the torus to $(S^2 \times S^2,\omega_1 \oplus \omega_1)$.

Recall that Arnold has shown that a Lagrangian torus $L$ satisfying the assumptions of Theorem \ref{T22} in fact is homotopic to the zero-section \cite{Arnold:FirstSteps}. We can thus fix a parametrization $\varphi \colon \TT^2 \hookrightarrow T^*\TT^2$ of $L$ which is homotopic (as a map into $T^*\TT^2$) to the standard embedding $\TT^2 \to 0_{\TT^2} \subset T^*\TT^2$ of the zero-section. After the fiber-wise translation of $T^*\TT^2$ obtained by adding the constant section given by the unique closed one-form $p_1d\theta_1+p_2d\theta_2 \in \Omega^1(\TT^2)$ homologous to $-\varphi^*\lambda$, the resulting Lagrangian submanifold can be seen to be exact. (The section being constant is not relevant here.) In this manner the general case of Theorem \ref{T22} is reduced to the case when $L$ is exact.

Furthermore, assume that we are given a Lagrangian isotopy $\varphi_t \colon \TT^2 \to T^*\TT^2$, $t \in [0,1]$, starting with the exact Lagrangian embedding $\varphi_0=\varphi$ above, and ending with the standard embedding $\varphi_1$ of the zero-section. Deform this isotopy by adding the path of sections being the graphs of the closed one-forms $p_1(t)d\theta_1+p_2(t)d\theta_2 \in \Omega^1(\TT^2)$ homologous to $-\varphi_t^*\lambda$. Since the symplectic action class is constant along this deformed path of Lagrangians, it is now a standard fact that it can be generated by a Hamiltonian isotopy. In other words, since any Lagrangian section obviously is Lagrangian isotopic to the zero-section, it suffices to show that an exact Lagrangian torus is \emph{Lagrangian isotopic} to a \emph{Lagrangian section} of $T^*\TT^2 \to \TT^2$.

Recall the definition of the holomorphic divisor
\[D_\infty := (S^2\times \{\infty\}) \: \cup \: (\{\infty\}\times S^2) \subset S^2\times S^2,\]
and write
\[D_0 := (S^2\times \{0\}) \: \cup \: (\{0\}\times S^2) \subset S^2\times S^2.\]
There is a symplectic identification
\[(T^*_{1/4\pi}S^1 \times T^*_{1/4\pi}S^1,d\lambda_{S^1} \oplus d\lambda_{S^1}) \cong (S^2 \times S^2 \setminus (D_\infty \cup D_0),\omega_1 \oplus \omega_1),\]
where $(T^*_{1/4\pi}S^1 \times T^*_{1/4\pi}S^1,d\lambda_{S^1} \oplus d\lambda_{S^1}) \subset (T^*\TT^2,d\lambda)$. The image of an exact Lagrangian submanifold under this identification can moreover be seen to be a monotone Lagrangian submanifold of $(S^2 \times S^2,\omega_1 \oplus \omega_1)$, given that its Maslov class vanishes. By the result in \cite{Kragh}, the latter is indeed always the case. Furthermore, after a fiber-wise rescaling in $T^*\TT^2$ (which induces a Hamiltonian isotopy of exact Lagrangian submanifolds), we may assume that any exact Lagrangian submanifold is contained in the bounded subset $T^*_{1/4\pi}S^1 \times T^*_{1/4\pi}S^1$.

In light of the discussion above, our goal will be to show that an exact Lagrangian torus $L \subset (T^*_{1/4\pi}S^1 \times T^*_{1/4\pi}S^1,d\lambda)$ is Lagrangian isotopic to a section of $(T^*\TT^2,d\lambda) \supset T^*_{1/4\pi}S^1 \times T^*_{1/4\pi}S^1$ inside the same subset.

\subsection{A compatible fibration}
\label{sec:proofT22fib}
We will apply the techniques in Section \ref{sec:fibration} but, first, we make the additional requirement that $J_\infty=i$ holds in a neighborhood of the divisor $D_\infty \cup D_0 \subset S^2 \times S^2$ (the assumptions made in the same section guarantees that this already is the case in a neighborhood of $D_\infty$). In particular, this divisor is thus assumed to be holomorphic. Since the exactness implies that $L \subset (S^2 \times S^2,\omega_1 \oplus \omega_1)$ is monotone, Theorem \ref{thm:fibration} can be used to show the existence of a symplectic $S^2$-fibration
\[ p_1 \colon S^2 \times S^2 \to S^2 \]
with fibers in the class $A_2 \in H_2(S^2 \times S^2)$ and satisfying the following properties:
\begin{enumerate}[label=(C'.\roman*)]
\item The restriction $p_1|_L$ is a trivial smooth $S^1$-fibration over the equator $S^1 \subset S^2$ of the base of $p_1$. For each $\theta \in S^1 \subset S^2$ there are precisely two components of $p_1^{-1}(\theta) \setminus L$ whose respective closures form two disk families $D_{\OP{small}}(\theta)$ and $D_{\OP{big}}(\theta)$ having boundary on $L$;
\item The fibers of $p_1$ may all be assumed to be pseudoholomorphic for an almost complex structure which coincides with $i$ in a neighborhood of the divisor $D_\infty \cup D_0$; and
\item The spheres $S^2 \times \{0,\infty\} \subset D_\infty \cup D_0$, which both are disjoint from $L$ by construction, are symplectic sections of the above fibration $p_1$.
\end{enumerate}
Since the inclusion $L \subset S^2 \times S^2 \setminus (D_\infty \cup D_0)$ is a homotopy equivalence by construction, it also follows that
\begin{enumerate}
\item[(C'.iv)] The unique intersection point of the above sections $S^2 \times \{\infty\}$ and $S^2 \times \{0\}$ of $p_1$ with a fiber $\pi_1^{-1}(\theta)$ over the equator is contained in the interior of $D_{\OP{big}}(\theta)$ and $D_{\OP{small}}(\theta)$, respectively.
\end{enumerate}
By $\mathcal{T}:=\bigcup_\theta D_{\OP{small}}(\theta)$ we denote the solid torus formed by the union of small disks (see Definition \ref{def:smallbig}) intersecting $D_0$ and having boundary equal to $L$. Note that $\mathcal{T} \subset S^2 \times S^2 \setminus D_\infty$.
\begin{remark} Property (C'.iv) above implies that the assumptions of the main theorem in \cite{Cieliebak-Schwingenheuer} are satisfied. Hence, $L \subset (S^2 \times S^2,\omega_1 \oplus \omega_1)$ is Hamiltonian isotopic to $0_{\TT^2} \subset T^*_{1/4\pi}S^1 \times T^*_{1/4\pi}S^1 \subset (S^2 \times S^2,\omega_1 \oplus \omega_1)$ (i.e.~the Clifford torus) inside the latter symplectic manifold. With a bit of more work, it is even possible to assume that the divisor $D_\infty \cup D_0$ is fixed under this Hamiltonian isotopy, thus giving the sought result.
\end{remark}
Here we instead choose another path, and give a self-contained construction of this Hamiltonian isotopy based upon the previously established technique of inflation.

Consider the standard trivial symplectic $S^2$-fibration
\[ \pi_1 \colon S^2 \times S^2 \to S^2 \]
given as the canonical projection onto the first $S^2$-factor. Note that, for any simple closed curve $S \subset S^2$, the preimage $\pi_1^{-1}(S) = S \times S^2$ is foliated by Lagrangian tori after removing the two embedded closed curves $S \times \{0,\infty\} \subset S \times S^2$. Furthermore, in the case when $S=S^1_r \subset S^2$, $r>0$, (using a suitable identification $S^2 \setminus \{0\} \cong \R^2$), the leaves $S^1_r \times S^1_s \subset S \times (S^2 \setminus \{0,\infty\})$, $s >0$, all correspond to Lagrangian \emph{sections} of
\[S^2 \times S^2 \setminus (D_\infty \cup D_0) \cong T^*_{1/4\pi}S^1 \times T^*_{1/4\pi}S^1 \subset T^*\TT^2 \to \TT^2,\]
given that appropriate identifications have been used.

After a Hamiltonian isotopy which deforms the solid torus $\mathcal{T}=\bigcup_\theta D_{\OP{small}}(\theta)$ (as well as the fibration $p_1$) followed by a deformation of the fibers of $p_1$ disjoint from the solid torus, both having support in an arbitrarily small neighborhood of $D_0$, we can arrange the following to hold in addition to the above Properties (C'.i)--(C'.iv):
\begin{enumerate}
\item[(C'.v)] Each leaf $D_{\OP{small}}(\theta) \subset p_1^{-1}(\theta)$, $\theta \in S^1$, of the solid torus coincides with the fiber $\pi_1^{-1}(\theta)=\{\theta\} \times S^2$, $\theta \in S^1_{r_0}$ in a neighborhood of $D_0$ for some fixed $r_0>0$. Moreover, each leaf $p_1^{-1}(q)$ for $q \in S^2$ contained in some small neighborhood of the equator $S^1$ coincides with a fiber of $\pi_1$ inside some fixed neighborhood of $D_0$.
\end{enumerate}
To see that this deformation exists, we argue as follows. We only show the first claim concerning the leaves $D_{\OP{small}}(\theta)$. The latter claim can be seen to follow by the methods in the proof of Lemma \ref{lma:trivialization}.

First, observe that $\bigcup_{\theta \in S^1}D_{\OP{small}(\theta)} \cap D_0$ is an embedded closed curve $\gamma \subset (S^2 \setminus \{0,\infty\}) \times \{0\} \subset D_0$ inside a symplectic annulus in the class of the generator of its fundamental group.

The symplectic fibers contained in $p_1^{-1}(S^1)$ all intersect $D_0$ transversely along this curve $\gamma$, where both $D_0$ and these fibers are holomorphic for the standard complex structure near these intersections by Property (C'.ii). After an explicit modification of these fibers, we may assume that each such fiber is of the form $\{p\} \times U$ for a neighborhood $U \subset S^2$ and point $p \in \gamma$.

A standard construction now provides a Hamiltonian isotopy $\phi^t$ of the symplectic annulus $(S^2 \setminus \{0,\infty\})\times \{0\} \subset (S^2 \times S^2,\omega_1 \oplus \omega_1)$ taking the closed curve $\gamma$ to one of the simple closed curves
\[S^1_{r_0} \times \{0\} \subset (S^2 \setminus \{0,\infty\}) \times \{0\} \subset S^2\times \{0\} \subset D_0,\]
in the foliation of $S^2 \setminus \{ 0,\infty\}$ by the circles $S^1_r$, $r>0$, as described above. The extension $(\phi^t \times \id_{S^2})$ of the Hamiltonian isotopy to all of $S^2 \times S^2$ can clearly be cut off in order to make it supported in an arbitrarily small neighborhood of $D_0$.

\subsection{Trivializing the symplectic monodromy}
Property (C'.v) above implies the following. First, the solid torus $\mathcal{T}:=\bigcup_{\theta \in S^1} D_{\OP{small}}(\theta) \subset S^2 \times S^2$ foliated by the small disks has a characteristic distribution which is tangent to the embedded closed curve $D_0 \cap \mathcal{T} \subset \mathcal{T}$. (Recall that this curve intersects each leaf $D_{\OP{small}}(\theta) \subset \mathcal{T}$ transversely in a single interior point.) Second, the monodromy map on a leaf $D_{\OP{small}}(\theta)$ induced by this characteristic distribution preserves the boundary and is equal to the identity in a neighborhood of the point $D_0 \cap D_{\OP{small}}(\theta) \subset D_{\OP{small}}(\theta)$. In order to get sufficient control of the monodromy on the whole disk, we need the following refinement of Theorem \ref{monodromyidentity}.
\begin{prop}\label{monodromyidentity2}
Assume that we are given a solid torus $\mathcal{T}$ as above. After a Lagrangian isotopy of its boundary $L$ inside $S^2 \times S^2 \setminus (D_\infty \cup D_0)$, together with a deformation of the solid torus $\mathcal{T}$ inside $S^2 \times S^2 \setminus D_\infty$, in addition to Properties (C'.iv) and (C'.v), the solid torus can also be made to satisfy the following: The monodromy map induced by the characteristic distribution is a symplectomorphism of $D_{\OP{small}}(\theta)$ that is equal to the identity in a neighborhood of $D_0\cap D_{\OP{small}}(\theta)$ together with its boundary, while it preserves a foliation of the punctured disk $D_{\OP{small}}(\theta) \setminus D_0$ by simple closed curves.
\end{prop}
\begin{proof}
Let $\varphi \in \mathrm{Symp}(D_{\OP{small}}(\theta))$ be the induced monodromy map. The assumptions of Lemma \ref{lma:concentric} below are satisfied, and we thus get a compactly supported Hamiltonian diffeomorphism $\phi^1_{H_t}$ of the punctured disk $D_{\OP{small}}(\theta) \setminus D_0$ for which $\phi^1_{H_t} \circ \varphi$ is a symplectomorphism of the form needed. In particular, it preserves a foliation by simple closed curves.

We can now follow the inflation procedure given by Theorem \ref{inflation}. While the obtained deformations $\phi_t(L)$ and $\phi_t(\mathcal{T})$ may be assumed to take place inside the complement $S^2 \times S^2 \setminus D_\infty$, this deformation need not fix $D_0$. To amend this we argue as follows.

First, consider the trivialization $\psi$ produced in Lemma \ref{lma:trivialization} above (also, see Property (C.iii)). Using the additional hypothesis satisfied by the section $S^2 \times \{0\}$ provided by Property (C'.v), we may assume that this section is constant in the obtained trivialization $\psi$. Examining the proof of Theorem \ref{inflation}, we now see that
$$\phi_t(D_0 \cup D_\infty)=\phi_t(D_0) \cup D_\infty \subset (S^2 \times S^2 ,\omega_1 \oplus \omega_1)$$
is a family of nodal symplectic spheres disjoint from $\partial \phi_t(\mathcal{T})$. Corollary \ref{cor:hamiso} together with Gromov's classification of pseudoholomorphic fibrations on $(S^2 \times S^2,\omega_1 \oplus \omega_1)$ shows the existence of a Hamiltonian isotopy $\phi_{G_t}^t$ for which $\phi_{G_t}^t \circ \phi_t$ fixes $D_\infty \cup D_0$ set-wise.

Using the above assumptions on the trivialization $\psi$, the induced trivializations
\[\phi_{G_t}^t \circ \psi_t= \phi_{G_t}^t \circ \phi_t \circ \psi \colon ([-a,a]^2 \times S^2,(\alpha_t \, dx \wedge dy) \oplus \beta_t \omega_1) \hookrightarrow (S^2 \times S^2,\omega_1 \oplus \omega_1) \]
produced by Theorem \ref{inflation} satisfies
\begin{eqnarray*}
&&(\phi_{G_t}^t \circ \psi_t)^{-1}(\phi_{G_t}^t(\mathcal{T}_t))=[-a,a]\times\{0\} \times D,\\
&&(\phi_{G_t}^t \circ \psi_t)^{-1}(D_0)=[-a,a]^2 \times \{q\},
\end{eqnarray*}
for some $q \in \mathrm{int}D \subset D \subset S^2$ (also, see Remark \ref{rmk:inflation}). The sought solid torus and Lagrangian submanifold will be $\phi_{G_t}^t(\mathcal{T}_t)$ and $\phi_{G_t}^t(L_t)$ for $t \gg 0$ sufficiently large. Finally, in order to make the new solid torus satisfy Property (C'.v) in some neighborhood of its intersection with $D_0$, it may be necessary to again perform a Hamiltonian isotopy supported near $D_0$ (as done in Section \ref{sec:proofT22fib}).

The end of the proof is similar to the proof of Theorem \ref{monodromyidentity} as given in Section \ref{sec:proofmonodromyidentity}. Apply Lemma \ref{lma:concentric} to the monodromy map $\varphi$ induced by the solid torus, thus producing a Hamiltonian $H_t$ on the disk for which $\phi^t_{H_t} \circ \varphi$ is a symplectomorphism that preserves a foliation by circles.

After the above inflation, the monodromy map of the torus is still given by $\varphi$. However, given that $t \gg 0$ was chosen sufficiently large, the rescaled Hamiltonian $\epsilon H_t$ for some arbitrarily small $\epsilon>0$ now produces the symplectomorphism $\phi^t_{\epsilon H_t} \circ \varphi$ of the sought form. In the coordinates given by the above trivialization $\phi_{G_t}^t \circ \psi_t$, we now invoke Lemma \ref{diskmonodromy} in order to replace a piece of the solid torus with the symplectic suspension of the Hamiltonian isotopy $\phi^t_{\epsilon H_t}$ considered above. This finishes the construction of the required solid torus.
\end{proof}

\begin{lemma}[Remark 3.4.4 in \cite{Ivrii-thesis}]
\label{lma:concentric}
Assume that we are given a compactly supported symplectomorphism $\varphi$ of $(D^2 \setminus \{0\},\omega_0)$. There exists a Hamiltonian isotopy $\phi^t_{H_t}\colon (D^2 \setminus \{0\},\partial D^2) \to (D^2 \setminus \{0\},\partial D^2)$ generated by a compactly supported Hamiltonian $H_t \colon D^2 \setminus \{0\} \to \R$ that satisfies $H_t|_{\partial D^2} \equiv 0$, and for which $\phi^1_{H_t} \circ \varphi$ is compactly supported, equal to the identity near the boundary, and preserves every concentric circle set-wise.
\end{lemma}
\begin{proof}
Using a standard argument as in the proof of Lemma \ref{lma:hamiso}, we find a Hamiltonian $G_t$ satisfying the above properties, except that it might depend on $t$ near the puncture (i.e.~it is constant in a neighborhood of the puncture for each fixed $t \in [0,1]$, but not necessarily vanishing there). The sought Hamiltonian $H_t$ can then be obtained from $G_t$ by the multiplication with a suitable bump-function $\rho \colon D^2 \to [0,1]$ that vanishes near the puncture, is equal to 1 in some subset of the form $D^2 \setminus B^2_\epsilon$, while every level-set $\rho^{-1}(y) \subset D^2$ for $y \in (0,1)$ is a concentric circle.
\end{proof}

\subsection{Proof of Theorem \ref{T22}}
\label{sec:proofT22}
Recall the arguments in Section \ref{sec:geomsetup} by which it is sufficient consider the case when
$$L \subset (T_{1/4\pi}^*S^1 \times T_{1/4\pi}^*S^1,d\lambda) \subset (S^2 \times S^2,\omega_1 \oplus \omega_1)$$
is a monotone Lagrangian torus, and where the first inclusion is a homotopy equivalence.

An application of Proposition \ref{monodromyidentity2} gives the following. After a Lagrangian isotopy of $L \subset (T_{1/4\pi}^*S^1 \times T_{1/4\pi}^*S^1,d\lambda)$, we can find a solid torus $\mathcal{T} \subset S^2 \times S^2 \setminus D_\infty$ with boundary $\partial \mathcal{T}=L$ that satisfies Properties (C'.iv) and (C'.v) of Section \ref{sec:proofT22fib}, the characteristic distribution of which moreover satisfies the following property. The induced monodromy map of the symplectic disks $D_{\OP{small}}(\theta)$, $\theta \in S^1$, preserves a foliation of the annulus $D_{\OP{small}}(\theta) \setminus D_0$ by closed embedded curves.

The existence of the Lagrangian isotopy taking the torus to a section is now immediate. Namely, the foliation of the leaf $D_{\OP{small}}(\theta) \setminus D_0$ by simple closed curves extends to a foliation of the solid torus by Lagrangian tori. Moreover, by choosing the foliation appropriately near the puncture $D_0 \cap D_{\OP{small}}(\theta) \subset D_{\OP{small}}(\theta)$, which is possible by Property (C'.v), the following can be achieved: The Lagrangian tori corresponding to closed curves in the foliation that are sufficiently close to the puncture of $D_{\OP{small}}(\theta) \setminus D_0$ all correspond to sections of $T_{1/4\pi}^*S^1 \times T_{1/4\pi}^*S^1 \subset T^*\TT^2 \to \TT^2$.
\qed

\appendix
\section[An explicit fibration compatible with the Chekanov torus]{An explicit fibration compatible with the Chekanov torus.}
\label{sec:appendix}

It is immediate that the two canonical projections $S^2 \times S^2 \rightrightarrows S^2$ both are symplectic fibrations compatible with the Clifford torus $S^1 \times S^1 \subset (S^2 \times S^2,\omega_1 \oplus \omega_1)$ (see the introduction for the definition). Another well-studied monotone torus inside $(S^2 \times S^2,\omega_1 \oplus \omega_1)$ is the so-called Chekanov torus. By Theorem \ref{thm:fibration} this torus also admits a compatible symplectic fibration which, since these two tori are not Hamiltonian isotopic, should be slightly more complicated. In this appendix we use elementary methods in order to produce a compatible fibration for the Chekanov torus by hand.

\subsection{A general construction of trivial symplectic $S^2$-fibrations with base $T^* S^1$.}
We start by describing general constructions of trivial symplectic $S^2$-fibration with the total space being the symplectic manifold $(T^*S^1 \times S^2,d\lambda_{S^1} \oplus \omega_1)$, and with fibers in the homology class $[\{\OP{pt}\} \times S^2] \in H_2(T^*S^1 \times S^2)$.

\subsubsection{Trivial symplectic $S^2$-fibrations via suspension}
For a Hamiltonian $H_t \colon S^2 \to \R$, $ t \in \R$, we use $\phi^t_{H_t} \colon (S^2,\omega) \to (S^2,\omega)$ to denote the corresponding Hamiltonian isotopy. Recall the definition of the \emph{symplectic suspension}, which is the symplectomorphism
\begin{gather*}
\Phi_{H_t} \colon (T^*\R \times S^2,(dp \wedge dq)\oplus \omega_1 ) \to (T^*\R \times S^2,(dp \wedge dq)\oplus \omega_1 ), \\
\Phi_{H_t}(q,p,x) \mapsto (q,p-H_q(\phi^q_{H_t}(x)),\phi^q_{H_t}(x)).
\end{gather*}
Here $q$ denotes the standard coordinate on $\R$ and $(q,p)$ denotes the induced canonical coordinates on $T^*\R=\R^2$ for which the Liouville form can be written as $p\,dq$.

In the case when $H_t=H_{t+2\pi}$, i.e.~when the Hamiltonian is periodic in the $t$-variable, the family of symplectic spheres
\[S_{(q,p)}:=\Phi_{H_t}(\{(q,p)\}\times S^2) \subset T^*\R \times S^2, \:\: (q,p) \in T^*\R,\]
satisfies the property that $S_{(q,p)}$ and $S_{(q+2\pi,p)}$ are identified under the canonical symplectic covering map
\[ (T^*\R \times S^2,(dp \wedge dq)\oplus \omega_1 ) \to (T^*S^1 \times S^2, d\lambda_{S^1} \oplus \omega_1 )\]
induced by the universal cover $\R \to \R/2\pi\Z =S^1$. In other words, the above symplectic suspension induces an $S^2$-fibration over $T^*S^1$ with fibers $S_{(q,p)} \subset T^*S^1 \times S^2$.

Observe that we may well have $\phi^{t+2\pi}_{H_t} \neq \phi^{t}_{H_t}$ here. In fact, the monodromy of the characteristic distribution of the $S^1$-family
$$\bigcup_{q \in S^1} S_{(q,p)} \subset (T^*S^1 \times S^2,d\lambda_{S^1} \oplus \omega_1),$$
for a fixed $p \in \R$, is given by this map $\phi^{2\pi}_{H_t} \colon (S^2,\omega_1) \to (S^2,\omega_1)$.

Finally we note that, since the Hamiltonians that are $2\pi$-periodic in the $t$-variable obviously form a contractible space, the above fibrations are all isomorphic to the trivial $S^2$-fibration $T^*S^1 \times S^2 \to S^2$ in the category of smooth fibrations.

\subsubsection{Symplectic $S^2$-fibrations via families of suspensions}
\label{sec:construction}
Above we saw that a Hamiltonian $H_t \colon S^2 \to \R$ periodic in the $t$-variable gives rise to a symplectic fibration with $T^*S^1$, whose fiber over $(q,p) \in T^*S^1$ is given by
\[S_{(q,p)}=\Phi_{H_t}(\{(q,p)\} \times S^2)=\Phi_{-p+H_t}(\{(q,0)\} \times S^2).\]
We now restrict our attention to autonomous Hamiltonians $h_s \colon S^2 \to \R$ which moreover smoothly depend on an additional variable $s \in \R$. Under the assumption that the graphs $\{(x,h_s(x))\} \subset S^2 \times \R$ provide a smooth foliation of $S^2 \times \R$, we obtain a symplectic fibration over $T^*S^1$ with fiber over $(q,p)$ given by
\[ S_{(q,p)}:=\Phi_{h_p}(\{(q,0)\} \times S^2) \subset (T^*S^1 \times S^2 , d\lambda \oplus \omega_1).\]
Again, the produced fibration is isomorphic to the trivial fibration $T^*S^1 \times S^2 \to T^*S^1$ in the category of smooth fibrations.

\subsection{Constructing symplectic fibrations with total space symplectomorphic to $(S^2 \times S^2,\omega_1\oplus\omega_1)$.}

\subsubsection{Fixing an identification of $(S^2,\omega_1)$}
In the following we consider the sphere $(S^2,\omega_1)$ of area $\int_{S^2}\omega_1=1$. We will represent this sphere as the round sphere $S^2_{1/\sqrt{4\pi}} \subset \R^3$ of radius $1/\sqrt{4\pi}$ endowed with the area form induced by the Euclidean metric on $\R^3$. We consider the Hamiltonian isotopy $\phi^t_h \colon (S^2,\omega) \to (S^2,\omega)$ which rotates the round sphere in $\R^3$ by $t \in S^1=\R / 2\pi \Z$ radians around the $x_1$-axis, where $(x_1,x_2,x_3) \in \R^3$ denotes the standard coordinates. Observe that this rotation fixes the points $N:=(1/\sqrt{4\pi},0,0)$ and $S:=(-1/\sqrt{4\pi},0,0)$ on the sphere.

The autonomous Hamiltonian generating the above rotation will be taken to be the ``height function'' $h \colon S^2 \to [-1/4\pi,1/4\pi]$ given by $h(x_1,x_2,x_3) = -\frac{x_1}{\sqrt{4\pi}}$. In particular, this Hamiltonian maps surjectively to the interval $[-1/4\pi,1/4\pi]$.

Away from the points $N,S \in S^2$, we can extend $h \colon S^2 \to [-1/4\pi,1/4\pi]$ to a symplectomorphism
\begin{gather*}
(S^2 \setminus \{N,S\},\omega_1) \to (T^*_{1/4\pi}S^1 =S^1 \times (-1/4\pi,1/4\pi),d\lambda_{S^1}), \\
x \mapsto \left(g(x_1,x_2,x_3),h(x_1,x_2,x_3)\right),
\end{gather*}
for the function $g \colon S^2 \setminus \{N,S\} \to S^1$ defined by $(x_1,x_2,x_3) \mapsto \arctan(x_3/x_2)$.

\subsubsection{An explicit $S^2$-fibration over $T^*_{1/4\pi}S^1$.}
\label{sec:producingbundle}
After cutting off the Hamiltonian $h \colon S^2 \to \R$ near the points $N$ and $S$ by taking the composition with a suitable smooth function $\rho \colon \R \to (-(1/4\pi-\epsilon),1/4\pi-\epsilon)$ satisfying $\rho' \ge 0$ and $\rho|_{(-(1/4\pi-2\epsilon),1/4\pi-2\epsilon)}=\id_\R$, we may assume that the obtained Hamiltonian $h_0:=\rho \circ h$ is as shown in Figure \ref{fig:interpolation}. We note that the induced Hamiltonian isotopy fixes each level-set of $h_0$ set-wise, while it coincides with the rotation $\phi^t_h \colon (S^2,\omega) \to (S^2,\omega)$ in the complement of a small neighborhood of $\{N,S\} \subset S^2$. The function $h_0$ extends to a family of functions $h_s \colon S^2 \to [-1/4\pi,1/4\pi]$, for which the graphs provide a foliation of $S^2 \times [-1/4\pi,1/4\pi]$, and where $h_s \equiv s$ is constant for all $1/4\pi-\epsilon/2 \le |s| \le 1/4\pi$; see Figure \ref{fig:interpolation}. Applying the construction in Section \ref{sec:construction} we obtain a symplectic $S^2$-fibration over $T^*_{1/4\pi} S^1$.

\begin{figure}[htp]
\centering
\vspace{3mm}
\labellist
\pinlabel $\color{red}h_0(p_\theta)$ at 142 140
\pinlabel $p_\theta$ at 177 75
\pinlabel $p_\theta$ at 365 75
\pinlabel $1/4\pi$ at 102 145
\pinlabel $1/4\pi-\epsilon$ at 50 129
\pinlabel $-1/4\pi$ at 105 8
\pinlabel $-(1/4\pi-\epsilon)$ at 118 24 
\pinlabel $1/4\pi-\epsilon$ at 127 62
\pinlabel $1/4\pi$ at 150 89
\pinlabel $-1/4\pi$ at 11 89
\pinlabel $\color{red}h_{1/4\pi}(p_\theta)$ at 316 155
\pinlabel $\color{red}h_0(p_\theta)$ at 334 118
\pinlabel $\color{red}h_{-1/4\pi}(p_\theta)$ at 311 -4

\endlabellist
\includegraphics{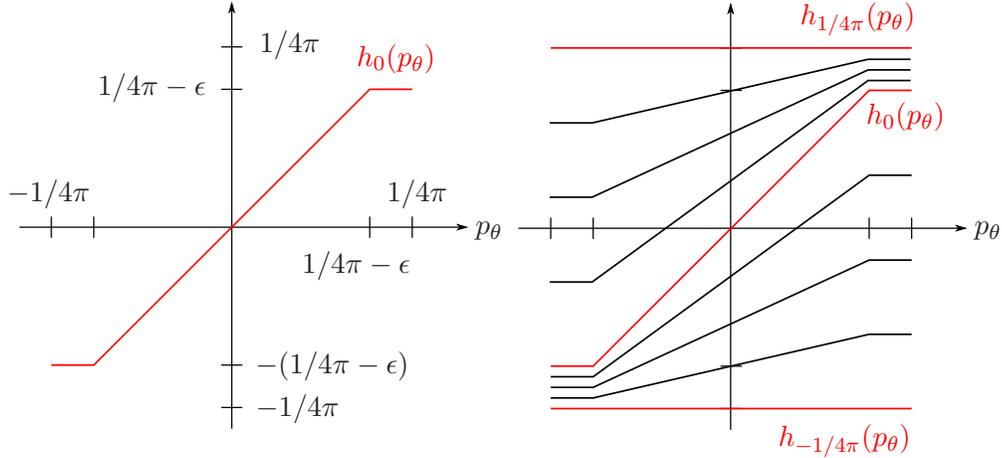}
\label{fig:interpolation}
\caption{On the left: the graph of the cut-off Hamiltonian $h_0=\rho \circ h$ for the coordinate $p_\theta:=h=-x_1/\sqrt{4\pi}$. On the right: the graphs of the family of autonomous Hamiltonians $h_s \colon S^2 \to [-1/4\pi,1/4\pi]$, $s \in [-1/4\pi,1/4\pi]$.}
\end{figure}

\subsubsection{An induced $S^2$-fibration on $(S^2 \times S^2,\omega_1 \oplus \omega_1)$.}
The above symplectic fibration extends to the closed symplectic manifold $\overline{T^*_{1/4\pi} S^1} \times S^2 \supset T^*_{1/4\pi} S^2 \times S^2$ with boundary. Write $S^*_{1/4\pi} S^1 := \partial\overline{T^*_{1/4\pi} S^1} \cong S^1 \sqcup S^1$ for the boundary of the base. The characteristic distribution of the boundary $\partial(\overline{T^*_{1/4\pi} S^1} \times S^2)$, which consists of the fibers $\{r\} \times S^2$, $r \in S^*_{1/4\pi}S^1$ by construction, induces the trivial monodromy map on $S^2$. After a symplectic reduction of these boundary components, the resulting symplectic manifold can be (rather canonically) identified with the monotone symplectic product manifold $(S^2 \times S^2,\omega_1 \oplus \omega_1)$. Under this quotient, the above fibration on $T^*_{1/4\pi} \times S^2$ moreover descends to a symplectic fibration
\[ \pi \colon (S^2 \times S^2,\omega_1 \oplus \omega_1) \to  S^2\]
with fiber in the homology class $[\{\OP{pt}\} \times S^2] \in H_2(S^2 \times S^2)$. Observe that the fibers near the ``north'' and ``south'' pole of the base all coincide with fibers $\{\OP{pt}\} \times S^2$ of the standard fibration. The zero-section of $T^*_{1/4\pi}S^1$ will be identified with the equator $S^1_{\OP{eq}} \subset S^2$.

\begin{remark}
\label{rmk:additionalfamily}
Consider the image of the suspension
\[ \Phi_h((\R \times \{0\}) \times S^2) \subset T^*\R \times S^2\]
under the canonical covering
\[T^*\R \times S^2 \to T^*S^1 \times S^2\]
followed by the projection to $S^2 \times S^2$ defined above. Note that the latter image is foliated by the symplectic ``diagonal'' spheres $\{ (x,\phi^t_h(x)) \in S^2 \times S^2 \}$, $t \in S^1$.
\end{remark}

\subsection{Constructing the Chekanov torus fibered over the equator.}
Recall that the rotation $\phi^t_h \colon (S^2,\omega_1) \to (S^2,\omega_1)$ fixes the two points $\{N,S\} \subset S^2$ identified with the intersection of $S^2_{1/\sqrt{4\pi}} \subset \R^3$ with the $x_1$-axis. We consider an oriented simple closed curve
\[\gamma \colon S^1 \to (S^2 \setminus \{N,S\},\omega_1)\cong (T^*_{1/4\pi} S^1,d\lambda)\]
satisfying the property that it is the oriented boundary of a domain of symplectic area $1/2$ (i.e.~half of the total area of $(S^2,\omega_1)$), while both points $\{N,S\}$ lie in the same connected component of $S^2 \setminus \gamma(S^1)$. For example, one may take the ``tennis ball curve'' as shown in Figure \ref{fig:tennisball}.

\begin{figure}[htp]
\centering
\vspace{5mm}
\labellist
\pinlabel $p_\theta$ at 97 164
\pinlabel $\color{blue}\gamma$ at 64 130
\pinlabel $\color{blue}\gamma$ at 304 78
\pinlabel $\theta$ at 199 72
\pinlabel $1/4\pi$ at 114 140
\pinlabel $-1/4\pi$ at 122 3
\pinlabel $\pi$ at 164 85
\pinlabel $N$ at 269 135
\pinlabel $\phi^t_{h_0}$ at 296 110
\pinlabel $S$ at 269 20
\pinlabel $-\pi$ at 28 85
\pinlabel $x_1$ at 279 165
\pinlabel $T^*_{1/4\pi}S^1$ at 14 10
\pinlabel $S^2_{1/\sqrt{4\pi}}\subset\R^3$ at 320 13
\endlabellist
\includegraphics{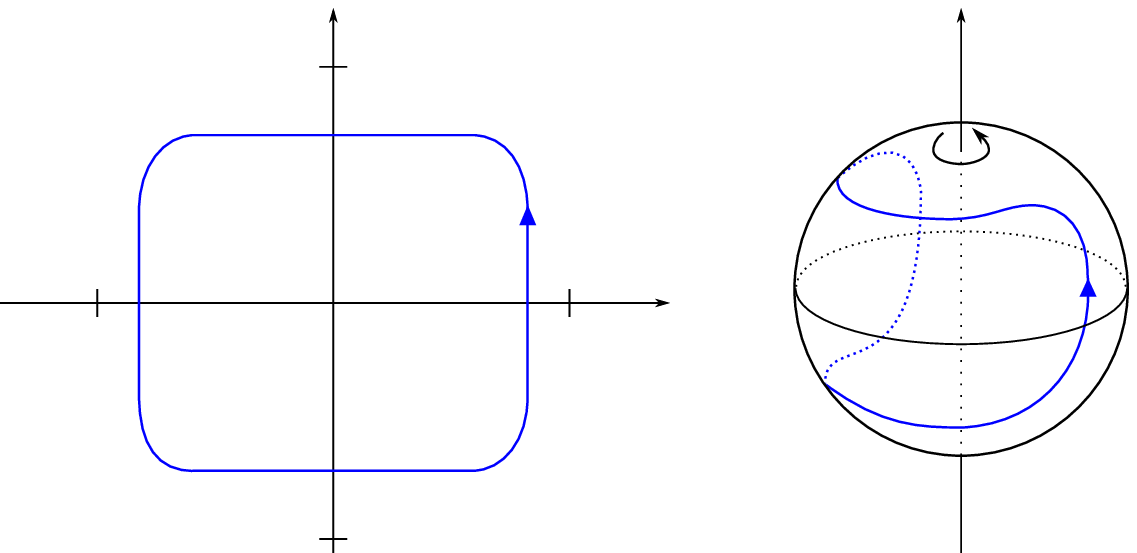}
\label{fig:tennisball}
\caption{The ``tennis ball curve'' bounding area $1/2$, and for which the two points $\{x_2=x_3=0\} \cap S^2_{1/\sqrt{4\pi}}=\{N,S\}$ on the sphere lie in the same component of its complement.}
\end{figure}

The monodromy on the union of fibers $\pi^{-1}(S^1_{\OP{eq}}) \cong S^1 \times S^2$ for the symplectic fibration $\pi \colon (S^2 \times S^2,\omega_1\oplus\omega_1) \to S^2$ constructed above is given by $\phi^{2\pi}_{h_0} \colon (S^2,\omega) \to (S^2,\omega)$.

In particular, $\phi^1_{h_0}$ coincides with $\phi^1_h=\id_{S^2}$ on the complement of some small neighborhood of $\{N,S\} \subset S^2$. We may hence assume that the parallel transport of the curve $\gamma \subset S^2 \setminus \{N,S\}$ described in Figure \ref{fig:tennisball} taken over the equator $S^1_{\OP{eq}} \subset S^2$ in the base closes up to form an embedded Lagrangian torus $L \subset (S^2 \times S^2,\omega_1 \oplus \omega_1)$.
\begin{prop}
The torus $L$ produced above is the monotone Chekanov torus in $(S^2 \times S^2,\omega_1 \oplus \omega_1)$.
\end{prop}
\begin{proof}
This can be seen using the presentation of the Chekanov torus given in e.g.~\cite{Gadbled}. Indeed, it lives inside the $S^1$-family of symplectic diagonal spheres $x \mapsto (x,\phi^t_h(x))$, $t \in S^1$, (see Remark \ref{rmk:additionalfamily}), and moreover intersects each such diagonal sphere in precisely two simple closed curves. See Figure \ref{fig:chekanov}.

\begin{figure}[htp]
\centering
\vspace{3mm}
\labellist
\pinlabel $p$ at 97 166
\pinlabel $\color{blue}\sqrt{\gamma}$ at 71 127
\pinlabel $q$ at 201 71
\pinlabel $1/4\pi$ at 114 140
\pinlabel $-1/4\pi$ at 119 3
\pinlabel $\pi$ at 164 85
\pinlabel $-\pi$ at 25 85
\endlabellist
\includegraphics{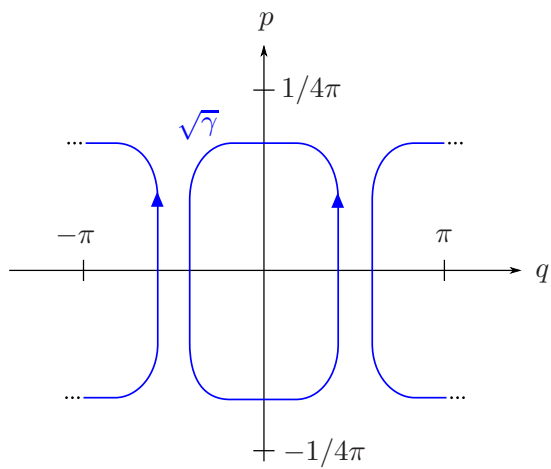}
\label{fig:chekanov}
\caption{The intersection of the Lagrangian torus $L \subset (S^2 \times S^2,\omega_1 \oplus \omega_1)$ and the diagonal sphere $\{(x,x)\}$, where the latter has been parametrized by $T^*_{1/4\pi}S^1$. Here we have made the identification $\sqrt{\gamma}:=\{ (q,p); \: (2q,p) \in \gamma \} \subset T^*_{1/4\pi}S^1$.}
\end{figure}
\end{proof}

\bibliographystyle{plain}
\bibliography{references}

\end{document}